\documentclass[reqno]{amsart}
\usepackage{graphicx}
\usepackage[dvips]{epsfig}
\usepackage{bm}
\numberwithin{figure}{section}

\newtheorem{theorem}{Theorem}[section]
\newtheorem{lemma}{Lemma}[section]
\newtheorem{proposition}[lemma]{Proposition}
\newtheorem{definition}[lemma]{Definition}
\newtheorem{remark}[lemma]{Remark}
\newtheorem{problem}{{\bf{Problem}}}[section]

\numberwithin{equation}{section}

\begin{document}

\title[Contact discontinuities for 3-D axisymmetric flows]{Contact discontinuities for 3-D axisymmetric inviscid compressible flows in infinitely long cylinders}

\author{Myoungjean Bae}
\address{Department of Mathematics, POSTECH, 77 Cheongam-Ro, Nam-Gu, Pohang, Gyeongbuk, Korea 37673; Korea Institute for Advanced Study, 85 Hoegiro, Dongdaemun-Gu, Seoul 02455, Republic of Korea}
\email{mjbae@postech.ac.kr}

\author{Hyangdong Park}
\address{Department of Mathematics, POSTECH, 77 Cheongam-Ro, Nam-Gu, Pohang, Gyeongbuk, Korea 37673}
\email{hyangdong.park@postech.ac.kr}

\keywords{angular momentum density, asymptotic state, axisymmetric,  contact discontinuity, free boundary problem, Helmholtz decomposition, infinite cylinder, steady Euler system, subsonic, vorticity}

\subjclass[2010]{
35J47, 35J57, 35J66, 35Q31, 35R35, 74J40, 76N10}

\date{\today}

\begin{abstract}
We prove the existence of a subsonic axisymmetric weak solution $({\bf u},\rho,p)$ with ${\bf u}=u_x{\bf e}_x+u_r{\bf e}_r+u_\theta{\bf e}_{\theta}$ to steady Euler system in a three-dimensional infinitely long cylinder $\mathcal{N}$ when prescribing the values of the entropy $(=\frac{p}{\rho^{\gamma}})$ and angular momentum density $(=ru_{\theta})$ at the entrance by piecewise $C^2$ functions with a discontinuity on a curve on the entrance of $\mathcal{N}$. Due to the variable entropy and angular momentum density (=swirl) conditions with a discontinuity at the entrance, the corresponding solution has a nonzero vorticity, nonzero swirl, and contains a contact discontinuity $r=g_D(x)$. We construct such a solution via Helmholtz decomposition. The key step is to decompose the Rankine-Hugoniot conditions on the contact discontinuity via Helmholtz decomposition so that the compactness of approximated solutions can be achieved. Then we apply the method of iteration to obtain a solution and analyze the asymptotic behavior of the solution at far field.
\end{abstract}
\maketitle


\section{Introduction}
In $\mathbb{R}^3$, the steady flow of inviscid compressible gas is governed by {\emph{the Euler system}} \cite{courant1999supersonic}:
\begin{equation}\label{E-System-1}
\left\{
\begin{split}
&\mbox{div}(\rho{\bf u})=0,\\
&\mbox{div}(\rho{\bf u}\otimes{\bf u}+p\,{\mathbb I}_3)=0\quad({\mathbb I}_3: \mbox{$3\times 3$ identity matrix}),\\
&\mbox{div}\left(\rho\Bigl(E+\frac{p}{\rho}\Bigr){\bf u}\right)=0.
\end{split}
\right.
\end{equation}
In \eqref{E-System-1}, the functions $\rho=\rho({\bf x})$, ${\bf u}=(u_1{\bf e}_1+u_2{\bf e}_2+u_3{\bf e}_3)({\bf x})$, $p=p({\bf x})$, and $E=E({\bf x})$ represent the density, velocity, pressure, and the total energy density of the flow, respectively, at ${\bf x}=(x_1, x_2, x_3)\in \mathbb{R}^3$. In this paper, we consider an {\emph{ideal polytropic gas}} for which $E$ is given by
\begin{equation}
\label{definition-energy}
E=\frac 12|{\bf u}|^2+\frac{p}{(\gamma-1)\rho}
\end{equation}
for a constant $\gamma>1$, called the {\emph{adiabatic exponent}}. With the aid of \eqref{definition-energy}, the system \eqref{E-System-1} is closed, and can be rewritten as
\begin{equation}\label{E-System}
\left\{
\begin{split}
&\mbox{div}(\rho{\bf u})=0,\\
&\mbox{div}(\rho{\bf u}\otimes{\bf u}+p\,{\mathbb I}_3)=0,\\
&\mbox{div}(\rho{\bf u}B)=0,
\end{split}
\right.
\end{equation}
for the {\emph{Bernoulli invariant}} $B$ given by
\begin{equation}\label{Ber-inv}
B=\frac{1}{2}|{\bf u}|^2+\frac{\gamma p}{(\gamma-1)\rho}=\frac{1}{2}|{\bf u}|^2+\frac{\gamma}{\gamma-1}S\rho^{\gamma-1}.
\end{equation}
Here, $S=p/\rho^{\gamma}$ denotes the entropy.

Let $\Omega\subset \mathbb{R}^3$ be an open and connected set. Suppose that a non-self-intersecting $C^1$-surface $\Gamma$ divides $\Omega$ into two disjoint open subsets $\Omega^{\pm}$ such that $\Omega=\Omega^-\cup\Gamma\cup \Omega^+$.
Suppose that ${\bf U}=({\bf u},\rho,p)$ satisfies the following properties:
\begin{itemize}
\item[$(w_1)$] ${\bf U}\in [L^{\infty}_{\rm loc}(\Omega)\cap C^1_{\rm loc}(\Omega^{\pm})\cap C^0_{\rm loc}(\Omega^{\pm}\cup \Gamma)]^5$;
\item[$(w_2)$] For any $\xi\in C_0^{\infty}(\Omega)$ and $k=1,2,3$,
\begin{equation*}
\int_{\Omega}\rho{\bf u}\cdot \nabla\xi\,d{\bf x}=\int_{\Omega}(\rho u_k{\bf u}+p{\bf e}_k)\cdot \nabla\xi\, d{\bf x}=\int_{\Omega}\rho{\bf u}B\cdot \nabla\xi \,d{\bf x}=0.
\end{equation*}
Here, ${\bf e}_k$ is the unit vector in the $x_k$-direction.
\end{itemize}

By integration by parts, one can directly check that ${\bf U}$ satisfies the properties $(w_1)$ and $(w_2)$ if and only if
\begin{itemize}
\item[$(w_1^*)$] ${\bf U}$ satisfies the property $(w_1)$;
\item[$(w_2^*)$] ${\bf U}$ is a classical solution to \eqref{E-System} in $\Omega^{\pm}$, and satisfies the Rankine-Hugoniot conditions
\begin{eqnarray}
\label{R-H-1}&&[\rho{\bf u}\cdot{\bf n}]_{\Gamma}=[\rho{\bf u}\cdot{\bf n}B]_{\Gamma}=0,\\
\label{R-H-2}&&[\rho({\bf u}\cdot{\bf n}){\bf u}+p{\bf n}]_{\Gamma}={\bf 0},
\end{eqnarray}
for a unit normal vector field ${\bf n}$ on $\Gamma$, where $[F]_{\Gamma}$ is defined by
$$[F({\bf x})]_{\Gamma}:=\left.F({\bf x})\right|_{\overline{\Omega^-}}-\left.F({\bf x})\right|_{\overline{\Omega^+}}\quad\mbox{for}\quad {\bf x}\in\Gamma.$$
\end{itemize}

Let ${\bm \tau}_1$ and ${\bm \tau}_2$ be tangent vector fields on $\Gamma$ such that they are linearly independent at each point on $\Gamma$.
 Due to $[\rho{\bf u}\cdot{\bf n}]_{\Gamma}=0$ in \eqref{R-H-1}, the condition \eqref{R-H-2} can be rewritten as
\begin{equation}\label{R-H-3}
[\rho({\bf u}\cdot{\bf n})^2+p]_{\Gamma}=0,\quad\rho({\bf u}\cdot{\bf n})[{\bf u}\cdot{\bm \tau}_k]_{\Gamma}=0\quad\text{for $k=1,2$.}
\end{equation}

Suppose that $\rho>0$ in $\Omega$. Then, the second condition in \eqref{R-H-3} holds if either ${\bf u}\cdot{\bf n}=0$ holds on $\Gamma$, or $[{\bf u}\cdot{\bm \tau}_k]_{\Gamma}=0$ for all $k=1,2$.

\begin{definition}
\label{definition-wsol}

We define ${\bf U}=({\bf u}, \rho, p)$ to be a weak solution to \eqref{E-System} in $\Omega$ with a {\emph{contact discontinuity $\Gamma$}} if the following properties hold:
\begin{itemize}
\item[(i)] $\Gamma$ is a non-self-intersecting $C^1$-surface dividing $\Omega$ into two open subsets $\Omega^{\pm}$ such that $\Omega=\Omega^+\cup\Gamma\cup \Omega^-$;

\item[(ii)] ${\bf U}$ satisfies $(w_1)$ and $(w_2)$, or equivalently
$(w_1^*)$ and $(w_2^*)$;

\item[(iii)] $\rho>0$ in $\overline{\Omega}$;

\item[(iv)] $\left({\bf u}|_{\overline{\Omega^-}\cap \Gamma}-{\bf u}|_{\overline{\Omega^+}\cap \Gamma}\right)({\bf x})\neq {\bf 0}$ holds for all ${\bf x}\in \Gamma$;

\item[(v)] ${\bf u}\cdot{\bf n}|_{\overline{\Omega^-}\cap \Gamma}={\bf u}\cdot{\bf n}|_{\overline{\Omega^+}\cap \Gamma}=0$, where ${\bf n}$ is a unit normal vector field on $\Gamma$.

\end{itemize}
\end{definition}

One can directly check from \eqref{R-H-1} and \eqref{R-H-3} that ${\bf U}=({\bf u}, \rho, p)$ is a weak solution to \eqref{E-System} in $\Omega$ with a contact discontinuity $\Gamma$ if and only if the following properties hold:
\begin{itemize}
\item[$(i')$] The properties (i)-(iv) stated in Definition \ref{definition-wsol} hold;
\item[$(ii')$] $[p]_{\Gamma}=0$ and  ${\bf u}\cdot{\bf n}=0$ on $\Gamma$.
\end{itemize}
\begin{figure}[!h]
\centering
\includegraphics[scale=0.7]{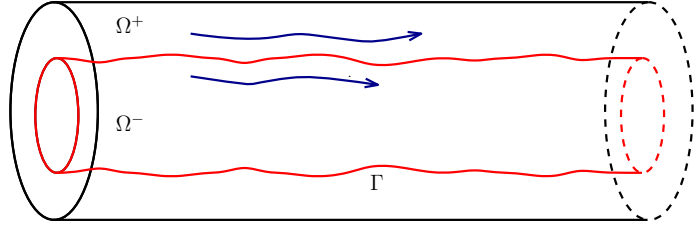}
\caption{Contact discontinuity}
\end{figure}

Let $(x,r,\theta)$ be the cylindrical coordinates of $(x_1,x_2, x_3)\in\mathbb{R}^3$, that is,
$$(x_1,x_2,x_3)=(x,r\cos\theta,r\sin\theta),\quad r\ge0,\quad \theta\in\mathbb{T},$$
where $\mathbb{T}$ is a one dimensional torus with period $2\pi$.
Any function $f({\bf x})$ can be represented as $f({\bf x})=f(x,r,\theta)$, and a vector-valued function ${\bf F}({\bf x})$ can be represented as
 $${\bf F}({\bf x})=F_x(x,r,\theta){\bf e}_x+F_r(x,r,\theta){\bf e}_r+F_{\theta}(x,r,\theta){\bf e}_{\theta},$$ where
$${\bf e}_x=(1,0,0),\quad{\bf e}_r=(0,\cos\theta,\sin\theta),\quad{\bf e}_{\theta}=(0,-\sin\theta,\cos\theta).$$
\begin{definition}
\begin{itemize}
\item[(i)]
A function $f({\bf x})$ is axially symmetric (=axisymmetric) if its value is independent of $\theta$.
\item[(ii)] A vector-valued function ${\bf F}$ is axially symmetric (=axisymmetric) if each of functions $F_x({\bf x})$, $F_r({\bf x})$, and $F_{\theta}({\bf x})$ is axially symmetric.
\end{itemize}
\end{definition}

The goal of this paper is to prove the existence of subsonic axisymmetric weak solutions to \eqref{E-System} with contact discontinuities in the sense of Definition \ref{definition-wsol} in a three-dimensional infinitely long cylinder.
In particular, we seek a solution with nonzero vorticity and nonzero angular momentum (=swirl).
Furthermore, we analyze asymptotic behaviors of the contact discontinuities at far field.

There are many studies of smooth subsonic solutions to Euler system, see \cite{chen2016subsonic,chen2012global,du2011global,du2014steady,du2011subsonic,duan2013three,duan2015subsonic,xie2007global,xie2010existence,xie2010global} and references cited therein.
As far as we know, there are few results on the existence of solutions to Euler system with contact discontinuities \cite{bae2013stability,bae2018contact,chen2017steady,chen2013stability,chen2013well,chen2006stability,chen2008mach,chen2008stability,wang2015structural}.
In \cite{wang2015structural}, supersonic contact discontinuities in three-dimensional isentropic steady flows were studied.

In this paper, we prove the existence of a subsonic axisymmetric weak solution $({\bf u},\rho,p)$ with ${\bf u}=u_x{\bf e}_x+u_r{\bf e}_r+u_\theta{\bf e}_{\theta}$ to steady Euler system in a three-dimensional infinitely long cylinder $\mathcal{N}$ when prescribing the values of the entropy $(=\frac{p}{\rho^{\gamma}})$ and angular momentum density $(=ru_{\theta})$ at the entrance by piecewise $C^2$ functions with a discontinuity on a curve on the entrance of $\mathcal{N}$. Due to the variable entropy and angular momentum density (=swirl) conditions with a discontinuity at the entrance, the corresponding solution has a nonzero vorticity, nonzero swirl, and contains a contact discontinuity $r=g_D(x)$. We construct such a solution via Helmholtz decomposition.
By using Helmholtz decomposition, smooth subsonic solutions for the full Euler-Poisson system with nonzero vorticity were studied in \cite{bae2014subsonic,bae20183}.
To construct subsonic solutions with contact discontinuities, the challenge is to decompose the Rankine-Hugoniot conditions on contact discontinuities via Helmholtz decomposition so that the compactness of approximated solutions can be achieved.

The first work to construct subsonic weak solutions with contact discontinuities to steady Euler system via Helmholtz decomposition is given in \cite{bae2018contact}, in which new formulations of steady Euler system and Rankine-Hugoniot conditions via Helmholtz decomposition are introduced, and the existence of subsonic weak solutions with contact discontinuities and nonzero vorticity is proved  in a two-dimensional infinitely long nozzle. Furthermore, it is proved that a two dimensional weak solution converges to a constant pressure state at far-field($x=\infty$), if one side of the contact discontinuity has uniform state with $(p, {\bf u})=(p_0, {\bf 0})$ for a constant $p_0>0$. In this paper, we consider a three-dimensional infinitely long circular cylinder with the same assumption. Namely, we prescribe boundary condition at the entrance of the cylinder so that the resultant subsonic weak solution to steady Euler system contains a contact discontinuity, and its one side has uniform state with $(p, {\bf u})=(p_0, {\bf 0})$ for a constant $p_0>0$.
Differently from the two dimensional case, however, the three dimensional problem that we consider in this paper requires a more subtle approach.  If we seek a weak solution via Helmholtz decomposition  with a contact discontinuity so that its inner layer flow has nonzero vorticity and nonzero angular momentum, we first need to establish the unique solvability of a singular-coefficient elliptic equation, which concerns the angular component of the vorticity in its cylindrical-coordinate representation. Also, a careful treatment is needed in analysis of streamlines near the $x$-axis ($r = 0$).
To resolve these difficulties, we employ the method developed in \cite{bae20183}, but with more sophisticated computations to handle nonlinear boundary conditions on the contact discontinuity, which are derived from the Rankine-Hugoniot conditions.

To analyze the asymptotic behavior of the solution, we use the stream function formulation and energy estimates. We emphasize that the asymptotic behavior of three dimensional subsonic weak solution with a contact discontinuity is completely different from the two dimensional solution, which are studied in \cite{bae2018contact}. Due to the non-zero angular momentum generated by the boundary condition at the entrance, the asymptotic limit of pressure $p$ of three dimensional subsonic weak solution with a contact discontinuity does not converge to a constant $p_0$ at $x=\infty$. And, this is purely three dimensional phenomenon.
To our best knowledge, this is the first result on the three-dimensional subsonic flows to steady Euler system with contact discontinuities.

The rest of the paper is organized as follows.
In Section \ref{3D-sec-Main}, we formulate the main problem of this paper, and state its solvability (Theorem \ref{3D-MainThm}(a)) and the asymptotic limit of the solution (Theorem \ref{3D-MainThm}(b)) as the main theorem.
In Section \ref{3D-sec-Hel}, we reformulate the problem introduced in Section \ref{3D-sec-Main} by using the method of Helmholtz decomposition, and state its solvability as Theorem \ref{3D-Thm-HD}.
As we shall see later, the problems given in Section \ref{3D-sec-Main} and \ref{3D-sec-Hel} are free boundary problems in an unbounded domain.
To construct a solution to the free boundary problems in an unbounded domain, free boundary problems in cut-off domains will be formulated and solved in Section \ref{3D-sec-Cut}.
Based on the results of Section \ref{3D-sec-Cut}, we prove Theorem \ref{3D-Thm-HD} from which Theorem \ref{3D-MainThm}(a) follows.
Finally, the asymptotic behavior of the solution at far field is analyzed in Section \ref{3D-sec-ex}.

\section{Main Theorems}\label{3D-sec-Main}
We define an infinitely long cylinder
\begin{equation}\label{def-N-cyl}
\mathcal{N}:=\left\{(x_1,x_2,x_3)\in\mathbb{R}^3:\mbox{ }x_1>0,\mbox{ }\sqrt{x_2^2+x_3^2}<1\right\}.
\end{equation}
As we defined in the previous section,
let $(x,r,\theta)$ be the cylindrical coordinates of $(x_1,x_2, x_3)\in\mathbb{R}^3$, that is,
$$(x_1,x_2,x_3)=(x,r\cos\theta,r\sin\theta),\quad r\ge0,\quad \theta\in\mathbb{T},$$
where $\mathbb{T}$ is a one dimensional torus with period $2\pi$.
 Then, the wall $\Gamma_{\rm w}$ and the entrance $\Gamma_{\rm en}$ of $\mathcal{N}$ are defined as
\begin{equation*}
\Gamma_{\rm w}:=\partial\mathcal{N}\cap\{r=1\},\quad
\Gamma_{\rm en}:=\partial\mathcal{N}\cap\{x_1=0\}.
\end{equation*}
To prescribe a boundary condition which causes an occurrence of a contact discontinuity, we define an inner layer of the entrance $\Gamma_{\rm en}^-$ by
\begin{equation*}
\Gamma_{\rm en}^-:=\Gamma_{\rm en}\cap\left\{r\le \frac{1}{2}\right\}.
\end{equation*}

Let us consider two layers of flow in $\mathcal{N}$ separated by the cylindrical surface $r=\frac{1}{2}$ with satisfying the following properties:
\begin{itemize}
\item[(i)] For fixed $\rho_0^{\pm}>0$ and $u_0>0$,  the velocity and density of outer and inner layers are given by $(0,0,0),$ $\rho_0^+$ and $(u_0,0,0)$, $\rho_0^-$ respectively;
\item[(ii)] The pressure of both outer and inner layers is given by a constant $p_0>0$;
\item[(iii)] The outer and inner layers are subsonic flows, i.e.,
$${u_0}<c_0\quad\mbox{for the sound speed}\quad c_0=\sqrt{\frac{\gamma p_0}{\rho_0^-}}.$$
\end{itemize}
Then a piecewise constant vector
\begin{equation*}
U_0(x_1,x_2,x_3):=\left\{
\begin{split}
(0,0,0,\rho_0^+,p_0)\quad&\mbox{for}\quad r>\frac{1}{2},\\
(u_0,0,0,\rho_0^-,p_0)\quad&\mbox{for}\quad r <\frac{1}{2}
\end{split}\right.
\end{equation*}
 is a weak solution of the Euler system \eqref{E-System}  in $\mathcal{N}$ with a contact discontinuity $\mathcal{N}\cap \{r=\frac{1}{2}\}$.
In this case, the entropy $S_0$ and Bernoulli function $B_0$ are piecewise constant functions with
\begin{equation}\label{def-S0-B0}
\begin{split}
&S_0(x_1,x_2,x_3)=\left\{\begin{split}
	\frac{p_0}{(\rho_0^+)^{\gamma}}=:S_0^+\quad&\mbox{for}\quad\frac{1}{2}<r<1,\\
	\frac{p_0}{(\rho_0^-)^{\gamma}}=:S_0^-\quad&\mbox{for}\quad0\le r <\frac{1}{2},\end{split}\right.\\
& B_0(x_1,x_2,x_3)=\left\{\begin{split}
	\frac{\gamma p_0}{(\gamma-1)\rho_0^+}=:B_0^+\quad&\mbox{for}\quad\frac{1}{2}<r<1,\\
	\frac{1}{2}u_0^2+\frac{\gamma p_0}{(\gamma-1)\rho_0^-}=:B_0^-\quad&\mbox{for}\quad0\le r<\frac{1}{2}.\end{split}\right.
\end{split}
\end{equation}

\begin{figure}[!h]
\centering
\includegraphics[width=0.78\columnwidth]{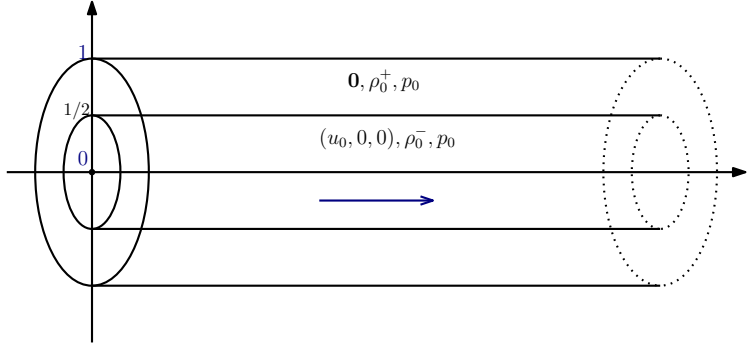}
\caption{Background state}
\end{figure}

Our main goal is to solve the following problem.
\begin{problem}\label{3D-Prob1}
Fix $\epsilon\in(0,1/10)$ and $\alpha\in(0,1)$.
For given radial functions
$S_{\rm en}(r)$, $\nu_{\rm en}(r)$ and $u_r^{\rm en}(r)$, define
\begin{equation}
\label{definition-sigma}
\sigma(S_{\rm en}, \nu_{\rm en}, u_r^{\rm en}):=\|S_{\rm en}-S_0\|_{2,\alpha,\Gamma_{\rm en}^-}+\|\nu_{\rm en}\|_{2,\alpha,\Gamma_{\rm en}^-}+\|u_r^{\rm en}\|_{1,\alpha,\Gamma_{\rm en}^-}.
\end{equation}
Assume that
\begin{equation}\label{def-entrance-ep}
\begin{split}
(S_{\rm en},\nu_{\rm en})\equiv(S_0^+,0)\quad&\mbox{on}\quad \Gamma_{\rm en}\setminus\Gamma_{\rm en}^-=\Gamma_{\rm en}
\cap\left\{ r\ge \frac{1}{2}\right\},\\
u_r^{\rm en}\equiv0\quad&\mbox{on}\quad \Gamma_{\rm en}
\cap\left\{ r\ge \frac{1}{2}-\epsilon\right\},
\end{split}
\end{equation}
and
\begin{equation}\label{3D-rem11}
\sigma(S_{\rm en}, \nu_{\rm en}, u_r^{\rm en})\le \sigma_0
\end{equation}
with sufficiently small $\sigma_0>0$ to be specified later.

Find a weak solution $U=({\bf u},\rho,p)$ to \eqref{E-System} with a contact discontinuity
$$\Gamma_{g_D}:r=g_D(x_1)$$
in the sense of Definition \ref{definition-wsol} in $\mathcal{N}$ such that
\begin{itemize}
\item[(a)] $g_D(0)=\frac{1}{2}$.
\item[(b)] Subsonicity: $${|{\bf u}|}<c\quad\mbox{for the sound speed}\quad c=\sqrt{\frac{\gamma p}{\rho}}\quad\mbox{in}\quad\overline{\mathcal{N}}.$$
\item[(c)] Positivity of density: $\rho>0$ {in} $\overline{\mathcal{N}}.$
\item[(d)] At the entrance $\Gamma_{\rm en}$, $U$ satisfies the boundary conditions:
\begin{equation*}\label{3D-BC-ent1}
\begin{split}
\frac{p}{\rho^{\gamma}}=S_{\rm en},\quad {\bf u}\cdot{\bf e}_{\theta}=\nu_{\rm en},\quad{\bf u}\cdot{\bf e}_r=u_r^{\rm en}\quad\mbox{on}\quad \Gamma_{\rm en}.
\end{split}
\end{equation*}
\item[(e)] On $\Gamma_{g_D}$,
$U$ satisfies the Rankine-Hugoniot conditions, i.e.,
\begin{equation*}\label{3D-BC-Cont1}
 [p]_{\Gamma_{g_D}}=0,\quad{\bf u}\cdot {\bf n}_{g_D}=0\quad\mbox{on}\quad \Gamma_{g_D},
\end{equation*}
where ${\bf n}_{g_D}$ denotes a unit normal vector field on $\Gamma_{g_D}$.
\item[(f)] On the wall $\Gamma_{\rm w}$,
$U$ satisfies the slip boundary condition, i.e.,
$${\bf u}\cdot{\bf e}_r=0\quad\mbox{on}\quad\Gamma_{\rm w}.$$
\item[(g)] The Bernoulli function $B$ is a piecewise constant function,
\begin{equation*}
B(x_1,x_2,x_3)=\left\{\begin{split}
&B_0^+\quad\mbox{for}\quad r>g_D(x_1),\\
&B_0^-\quad\mbox{for}\quad r<g_D(x_1),\\
\end{split}\right.
\end{equation*}
where $B_0^{\pm}$ are given by \eqref{def-S0-B0}.
\end{itemize}
\end{problem}

\begin{remark}[Compatibility conditions] \label{Rem1}If an axisymmetric vector field
$${\bf V}({\bf x})=V_x(x,r){\bf e}_x+V_r(x,r){\bf e}_r+V_{\theta}(x,r){\bf e}_{\theta}\quad\mbox{is }C^1,$$
then it must satisfy
$$V_r(x,0)=V_{\theta}(x,0)\equiv 0.$$
Since it is assumed in \eqref{3D-rem11} that the axisymmetric functions $(S_{\rm en},\nu_{\rm en})(r)$ are $C^1$ on $\Gamma_{\rm en}$, the compatibility conditions
\begin{equation}
\label{natural-comp-cond}
\partial_r(S_{\rm en},\nu_{\rm en})(0)={\bf 0}
\end{equation}
are naturally imposed.
\end{remark}

One can easily see that ${\bf u}={\bf 0}$, $\rho=\rho_0^+$, $p=p_0$ satisfy the following properties:
\begin{itemize}
\item[(i)] (Subsonicity) $|{\bf u}|=0<\sqrt{\gamma p/\rho}=\sqrt{(\gamma p_0)/\rho_0^+};$
\item[(ii)] (Positivity of density) $\rho_0^+>0$;
\item[(iii)] As in \eqref{def-S0-B0},
$$\frac{p_0}{(\rho_0^+)^{\gamma}}=S_0^+,\quad \frac{\gamma p_0}{(\gamma-1)\rho_0^+}=B_0^+;$$
\item[(iv)] ${\bf u}\cdot{\bf v}=0$ for any vector ${\bf v}\in\mathbb{R}^3$.
\end{itemize}
From this observation, we fix
${\bf u}={\bf 0}$, $\rho=\rho_0^+$, $p=p_0$ in $\mathcal{N}\cap\{r>g_D(x)\}$,
and we solve the following free boundary problem to find a solution to Problem \ref{3D-Prob1};
\begin{figure}[!h]
\centering
\includegraphics[width=0.85\columnwidth]{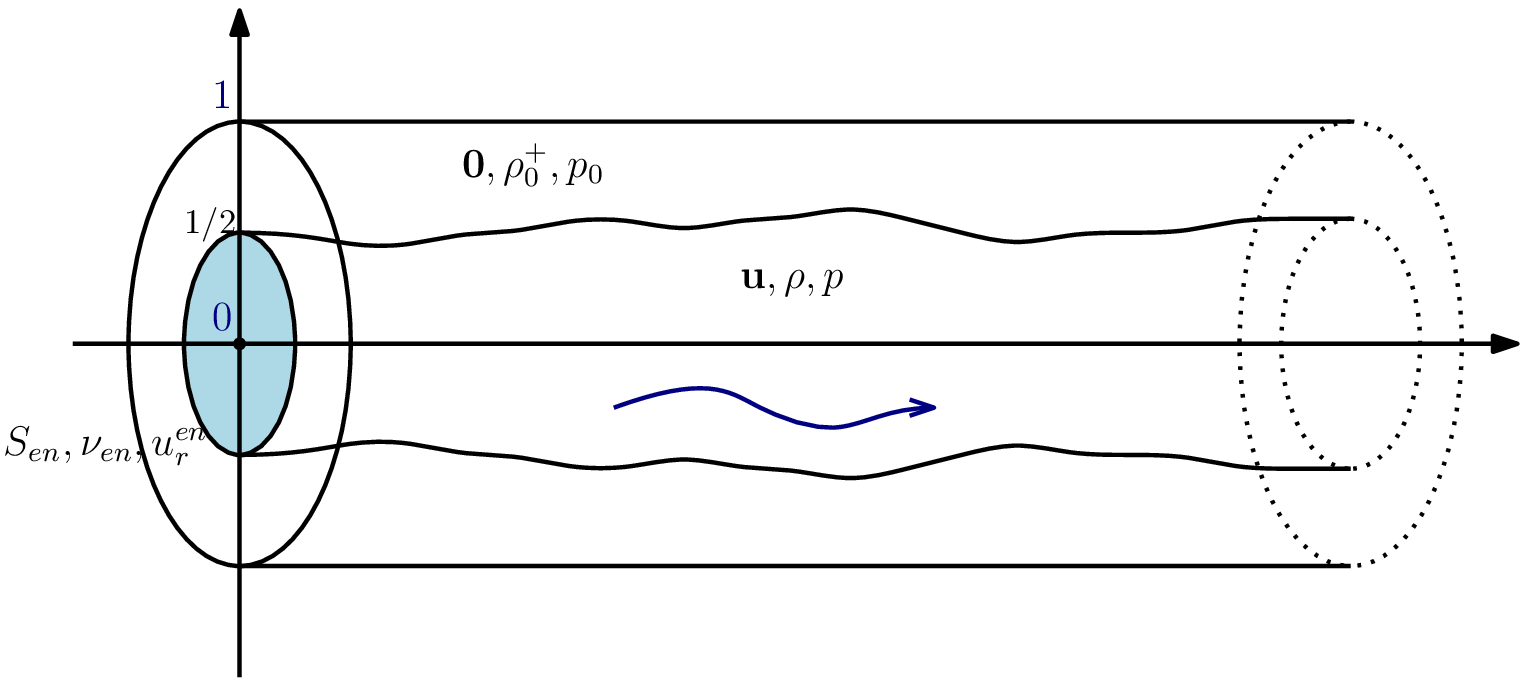}
\caption{Problem~\ref{3D-Problem2}}
\end{figure}

\begin{problem}\label{3D-Problem2}
Under the same assumptions of Problem \ref{3D-Prob1},
find $g_D:\mathbb{R}^+\longrightarrow (0,1)$ and a $C^1$ solution $U=({\bf u},\rho,p)$ to \eqref{E-System} in $\mathcal{N}_{g_D}^-:=\mathcal{N}\cap\{r<g_D(x)\}$ such that
\begin{itemize}
\item[(a)] \begin{equation}\label{GD0}
g_D(0)=\frac{1}{2}.
\end{equation}
\item[(b)] Subsonicity: $${|{\bf u}|}<c\quad\mbox{for the sound speed}\quad c=\sqrt{\frac{\gamma p}{\rho}}\quad\mbox{in}\quad\overline{\mathcal{N}^-_{g_D}}.$$
\item[(c)] Positivity of density: $\rho>0$ {in} $\overline{\mathcal{N}^-_{g_D}}.$
\item[(d)] At the entrance $\Gamma_{\rm en}^-$, $U$ satisfies the boundary conditions:
\begin{equation}\label{Prob2-BC-ent}
\begin{split}
\frac{p}{\rho^{\gamma}}=S_{\rm en},\quad {\bf u}\cdot{\bf e}_{\theta}=\nu_{\rm en},\quad{\bf u}\cdot{\bf e}_r=u_r^{\rm en}\quad\mbox{on}\quad \Gamma_{\rm en}^-.
\end{split}
\end{equation}

\item[(e)] On $\Gamma_{g_D}: r=g_D(x)$,
$U$ satisfies the boundary conditions
\begin{equation}\label{Prob2-BC-Cont}
 p=p_0,\quad{\bf u}\cdot {\bf n}_{g_D}=0\quad\mbox{on}\quad \Gamma_{g_D},
\end{equation}
where ${\bf n}_{g_D}$ denotes a unit normal vector field on $\Gamma_{g_D}$.
\item[(f)] The Bernoulli function $B$ is a constant function,
$$B(x_1,x_2,x_3)\equiv B_0^-\quad\mbox{in}\quad \overline{\mathcal{N}_{g_D}^-},$$
where $B_0^-$ is given by \eqref{def-S0-B0}.
\end{itemize}
\end{problem}

Since $p=S\rho^{\gamma}$, we can regard Problem \ref{3D-Problem2} as a problem for $({\bf u},p,S)$.
Assume that the smooth solution $({\bf u},\rho, S)$ of \eqref{E-System} is axially symmetric, i.e.,
$${\bf u}=u_x(x,r){\bf e}_x+u_r(x,r){\bf e}_r+u_{\theta}(x,r){\bf e}_{\theta},\quad
\rho=\rho(x,r),\quad  S=S(x,r).$$
Define the angular momentum density $\Lambda$ as follows
\begin{equation}
\label{definition-Lambda}
\Lambda(x,r):=ru_{\theta}(x,r).
\end{equation}
Then one can directly check that \eqref{E-System} is equivalent to the following system:
\begin{equation}\label{3D-ang}
\left\{\begin{split}
&\partial_x(\rho u_x)+\partial_r(\rho u_r)+\frac{\rho u_r}{r}=0,\\
&\rho(u_x\partial_x+u_r\partial_r)u_r-\frac{\rho u_{\theta}^2}{r}+\partial_r p=0,\\
&\rho(u_x\partial_x+u_r\partial_r)S=0,\\
&\rho(u_x\partial_x+u_r\partial_r)\Lambda=0.
\end{split}
\right.
\end{equation}

Now we state the main results in this paper.

\begin{theorem}\label{3D-MainThm}
For given radial functions
$S_{\rm en}(r)$, $\nu_{\rm en}(r)$ and $u_r^{\rm en}(r)$ on $\Gamma_{\rm en}$, assume that they satisfy \eqref{def-entrance-ep}, and let  $\sigma(S_{\rm en}, \nu_{\rm en}, u_r^{\rm en})$ be given by \eqref{definition-sigma}. For simplicity of notations, let $\sigma$ denote  $\sigma(S_{\rm en}, \nu_{\rm en}, u_r^{\rm en})$.
\begin{itemize}
\item[(a)](Existence)
 For any fixed $\alpha\in(0,1)$, there exists a small constant $\sigma_1>0$ depending only on $(u_0,\rho_0^-,p_0,S_0^-)$ and $\alpha$ so that if
\begin{equation*}
\sigma\le \sigma_1,
\end{equation*}
 then there exists an axially symmetric solution $U=({\bf u},\rho,p)$ of Problem \ref{3D-Problem2} with a contact discontinuity $r=g_D(x)$ satisfying \begin{equation}\label{Thm2.1-uniq-est}
\|g_D-\frac{1}{2}\|_{2,\alpha,\mathbb{R}^+}+\|({\bf u},\rho,p)-({\bf u}_0,\rho_0^-,p_0)\|_{1,\alpha,\mathcal{N}^-_{g_D}}\le C\sigma\quad\mbox{for}\,\,{\bf u}_0:=u_0{\bf e}_x,
\end{equation}
where the constant $C>0$ depends only on $(u_0,\rho_0^-,p_0,S_0^-)$ and $\alpha$.
\item[(b)](Asymptotic state)\label{3D-MainThm}
There exists a constant $\sigma_2\in(0,\sigma_1]$ depending only on $(u_0,\rho_0^-,p_0,S_0^-)$ and $\alpha$ so that if
$${\sigma}\le\sigma_2,$$
then the solution $U=({\bf u},\rho,p)$ in (a) satisfies
\begin{equation*}
\begin{split}
&\lim_{L\rightarrow\infty}\|{\bf u}\cdot{\bf e}_r(x,\cdot)\|_{C^1(\overline{\mathcal{N}^-_{g_D}\cap\{x>L\}})}=0,\\
&\lim_{L\rightarrow\infty}\|\partial_rp(x,\cdot)-\frac{\rho ({\bf u}\cdot{\bf e}_\theta)^2}{r}(x,\cdot)\|_{C^0(\overline{\mathcal{N}_{g_D}^-\cap\{x>L\}})}=0.
\end{split}
\end{equation*}
\end{itemize}
\end{theorem}

\begin{remark}[Zero swirl case]\label{Cor-Zero}
As we shall see later, the constant $\sigma_1$ in Theorem \ref{3D-MainThm}(a) will be chosen sufficiently small so that the estimate \eqref{Thm2.1-uniq-est} yields that
\begin{equation}
\label{nonvan-x-vel}
{\bf u}\cdot {\bf e}_x\ge \frac 12 u_0\quad\text{ in $\overline{\mathcal{N}_{g_D}^-}$.}
\end{equation}
If $\nu_{\rm en}=0$ on $\Gamma_{\rm en}$, by the definition of $\Lambda$ given by \eqref{definition-Lambda}, then it follows from \eqref{natural-comp-cond}, the transport equation $\rho(u_x\partial_x+u_r\partial_r)\Lambda=0$ given in \eqref{3D-ang}, and the estimate \eqref{nonvan-x-vel} that $\Lambda\equiv 0$ in $\overline{\mathcal{N}_{g_D}^-}$. And, this implies that
$\displaystyle{\frac{\rho ({\bf u}\cdot{\bf e}_\theta)^2}{r}\equiv 0}$ in $\overline{\mathcal{N}_{g_D}^-}$.
(See \eqref{S-Lambda} for further details.)
In this case, Theorem \ref{3D-MainThm}(b) yields that
$$\lim_{L\rightarrow\infty}\|p(x,\cdot)- p_0\|_{C^0(\overline{\mathcal{N}\cap\{x>L\}})}=0,$$
where we extend the definition of $p$ onto $\mathcal{N}\setminus \mathcal{N}^-_{g_D}$ by $p=p_0$ in $\mathcal{N}\setminus \mathcal{N}^-_{g_D}$. And, this coincides with the result obtained from \cite{bae2018contact}. From this perspective, the two dimensional subsonic weak solution with a contact discontinuity, constructed in \cite{bae2018contact}, can be considered as a three-dimensional subsonic weak solution with \emph{the zero-swirl} boundary condition for ${\nu}_{\rm en}$ at the entrance of the cylinder $\mathcal{N}$.
\end{remark}

\begin{figure}[!h]
\centering
\includegraphics[width=0.8\columnwidth]{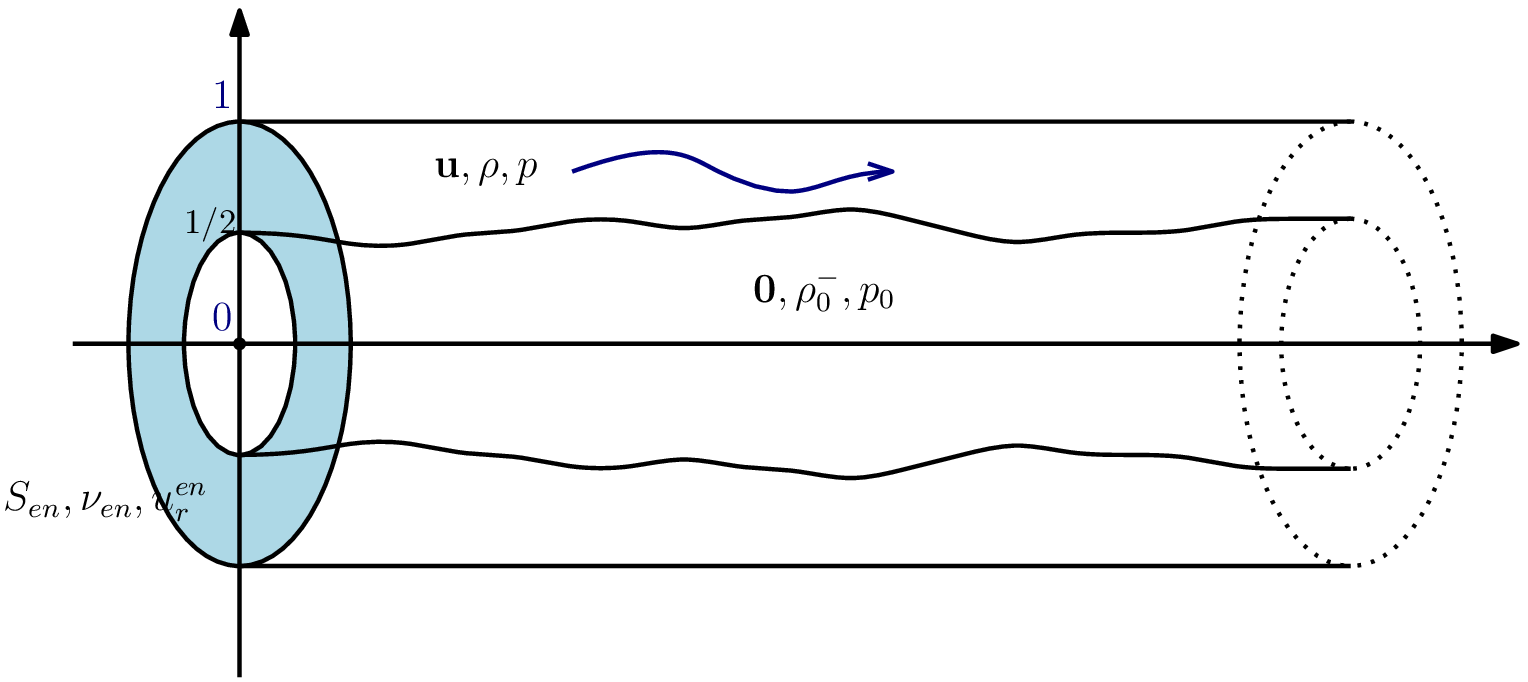}
\caption{Remark~\ref{rem-rev}}\label{figrem-rev}
\end{figure}
\begin{remark}\label{rem-rev}
In Problem \ref{3D-Problem2}, we seek a subsonic weak solution to steady Euler system with a contact discontinuity $r=g_D(x)$ by fixing the outer-layer flow in $\mathcal{N}\cap\{r>g_D(x)\}$ as a uniform state $({\bf u}, \rho, p)=({\bf 0}, \rho_0^+, p_0)$.
One can also consider a problem to seek a subsonic weak solution to steady Euler system with a contact discontinuity $r=\tilde{g}_D(x)$ by fixing {\emph{the inner-layer flow}} in $\mathcal{N}\cap\{r<\tilde{g}_D(x)\}$ as a uniform state $({\bf u}, \rho, p)=({\bf 0}, \rho_0^-, p_0)$ (See Fig. \ref{figrem-rev}).
Actually, this problem is even simpler than Problem \ref{3D-Problem2} for the following reason:
In order to solve Problem \ref{3D-Problem2}, we use Helmholtz decomposition
${\bf u}=\nabla\varphi+\mbox{curl}{\bf V}({\bf x})$ with ${\bf V}({\bf x})=h(x,r){\bf e}_r+\psi(x,r){\bf e}_{\theta}$.
With this representation, \eqref{E-System} is decomposed as a system of second order elliptic equations for $(\varphi,\psi)$, and transport equations for $(S,\Lambda)$. In this reformulation, one of the difficulties rises. Namely, the equation for $\psi$ becomes a singular-coefficient elliptic equation, with a coefficient blow-up on the $x$-axis ($r=0$).
If the inner-layer flow is fixed as a uniform state with $({\bf u}, \rho, p)=({\bf 0}, \rho_0^-, p_0)$, however, such a singularity issue is not needed to be considered, as the the inner-layer flow is fixed, and the outer-layer flow state is to be determined by solving nonlinear system of equations for $(\varphi, \psi, S, \Lambda)$. In particular, the outer-layer of $\mathcal{N}$ is away from the $x$-axis, therefore coefficients of all the equations are regular.

\end{remark}


\section{Reformulation of Problem \ref{3D-Problem2} via Helmholtz decomposition}\label{3D-sec-Hel}
For a function $g_D:\mathbb{R}^+\longrightarrow (0,1)$ to be determined along with $({\bf u}, {\rho}, p)$ in $\mathcal{N}^-_{g_D}$, we express the velocity vector field ${\bf u}=u_x(x,r){\bf e}_x+u_r(x,r){\bf e}_r+u_{\theta}(x,r){\bf e}_{\theta}$ as
\begin{equation*}\label{u-HD}
{\bf u}({\bf x})=\nabla\varphi({\bf x})+\mbox{curl}\,{\bf V}({\bf x})\quad\mbox{in}\quad \mathcal{N}_{g_D}^-
\end{equation*}
for axially symmetric functions
$$\varphi({\bf x})=\varphi(x,r),\quad{\bf V}({\bf x})=h(x,r){\bf e}_r+\psi(x,r){\bf e}_{\theta}.$$
If $(\varphi, {\bf V})$ are $C^2$ in $\mathcal{N}_{g_D}^-$,
then a direct computation yields
\begin{equation}\label{3D-u}
{\bf u}=\left(\partial_x\varphi+\frac{1}{r}\partial_r(r\psi)\right){\bf e}_x+(\partial_r\varphi-\partial_x\psi){\bf e}_r+(\partial_x h){\bf e}_{\theta},
\end{equation}
from which we derive that
\begin{equation*}
u_x=\partial_x\varphi+\frac{1}{r}\partial_r(r\psi),\quad u_r=\partial_r\varphi-\partial_x\psi,\quad u_{\theta}=\frac{\Lambda}{r}=\partial_xh.
\end{equation*}
Hereafter, we denote the velocity field ${\bf u}$ as
\begin{equation}
\label{definition-q}
{\bf u}={\bf q}(r,\psi,D\psi,D\varphi,\Lambda)\quad\mbox{for}\quad D=(\partial_x,\partial_r).
\end{equation}
For such ${\bf q}(r,\psi,D\psi,D\varphi,\Lambda)$, set
\begin{equation}\label{def-T}
{\bf t}(r,\psi,D\psi,\Lambda):={\bf q}(r,\psi,D\psi,D\varphi,\Lambda)-\nabla\varphi\left(=\mbox{curl}{\bf V}\right).
\end{equation}
By a simple adjustment of computations given in \cite{bae20183}, we can rewrite the system \eqref{3D-ang} as  follows:\begin{equation}\label{3D-H}
\left\{\begin{split}
&\mbox{div}\left(H(S,{\bf q}){\bf q}\right)=0,\\
&-\Delta(\psi{\bf e}_{\theta})=G(S,\Lambda,\partial_r S,\partial_r\Lambda,{\bf t},\nabla\varphi){\bf e}_{\theta},\\
&H(S,{\bf q}){\bf q}\cdot\nabla S=0,\\
&H(S,{\bf q}){\bf q}\cdot\nabla \Lambda=0,\\
\end{split}
\right.
\end{equation}
with
\begin{equation*}
{\bf q}={\bf q}(r,\psi,D\psi,D\varphi,\Lambda),\quad\text{and}\quad {\bf t}={\bf t}(r,\psi,D\psi,\Lambda),
\end{equation*}
for $(H, G)$ defined by
\begin{equation}\label{def-H-G}
\left.\begin{split}
&H(\eta,{\bf q}):=\left[\frac{\gamma-1}{\gamma\eta}\left(B_0^--\frac{1}{2}|{\bf q}|^2\right)\right]^{1/(\gamma-1)},\\
&G(\eta_1,\eta_2,\eta_3,\eta_4,{\bf t},{\bf v}):=\frac{1}{({\bf t}+{\bf v})\cdot{\bf e}_x}\left(\frac{H^{\gamma-1}(\eta_1, {\bf t}+{\bf v})}{\gamma-1}\eta_3+\frac{\eta_2}{r^2}\eta_4\right),\\
\end{split}
\right.
\end{equation}
for $\eta\in\mathbb{R}$, ${\bf q}\in\mathbb{R}^3$, $\eta_1,\eta_2,\eta_3,\eta_4\in\mathbb{R}$, and ${\bf t},{\bf v}\in\mathbb{R}^3$.
\smallskip

Next, we derive boundary conditions for $(g_D,S,\Lambda,\varphi,\psi)$ to satisfy the physical boundary conditions \eqref{Prob2-BC-ent}-\eqref{Prob2-BC-Cont}. We intend to derive the boundary conditions so that a compactness of approximated solutions to Problem \ref{3D-Problem2} can be established.

{\emph{(i) Boundary conditions on $\Gamma_{\rm en}^-$: }} We require $(S, \Lambda, \varphi,\psi)$ to satisfy
\begin{equation}\label{def-varphi-en}
\begin{cases}
(S,\Lambda)(0,r)=(S_{\rm en},r\nu_{\rm en})\\
\varphi(0,r)=\int_{1/2}^ru_r^{\rm en}(t)dt=:\varphi_{\rm en}\\
\partial_x\psi(0,r)=0
\end{cases}\quad\text{on $\Gamma_{\rm en}^-$}
\end{equation}
so that the boundary conditions given in \eqref{Prob2-BC-ent} hold on $\Gamma_{\rm en}^-$ for
\begin{equation}
\label{urp-qHS}
({\bf u}, \rho, p)=({\bf q}, H(S,{\bf q}), SH^{\gamma}(S,{\bf q}))\quad\text{ with ${\bf q}={\bf q}(r,\psi,D\psi,D\varphi,\Lambda)$.}
\end{equation}
\smallskip

{\emph{(ii) The Rankine-Hugoniot conditions \eqref{Prob2-BC-Cont} on $\Gamma_{g_D}$: }}  If a contact discontinuity $\Gamma_{g_D}$ is represented as $\Gamma_{g_D}=\{{\bf x}\in \mathcal{N}: r=g_D(x)\}$, then the unit normal ${\bf n}_{g_D}$ of $\Gamma_{g_D}$ pointing toward $\{r>g_D(x)\}$ is given by
\begin{equation*}
{\bf n}_{g_D}=\frac{-g_D'(x){\bf e}_x+{\bf e}_r}{\sqrt{1+|g_D'(x)|^2}}.
\end{equation*}
Therefore, if $g_D:\mathbb{R}^+\longrightarrow (0,1)$ solves the initial value problem
\begin{equation}\label{g-free-cond}
\left\{\begin{split}
&g_D'(x)=\frac{{\bf q}(r,\psi,D\psi,D\varphi,\Lambda)\cdot{\bf e}_r}{{\bf q}(r,\psi,D\psi,D\varphi,\Lambda)\cdot{\bf e}_x}(x,g_D(x),0)\quad\mbox{for}\quad x>0,\\
&g_D(0)=\frac{1}{2},
\end{split}\right.
\end{equation}
then the condition ${\bf u}\cdot{\bf n}_{g_D}=0$ holds on  $\Gamma_{g_D}$ for ${\bf u}$ given by \eqref{urp-qHS}. We use \eqref{g-free-cond} to find the location of the contact discontinuity $r=g_D(x)$.

Due to axi-symmetry of $\Gamma_{g_D}$, an orthonormal basis of $\Gamma_{g_D}$ can be given as
$$\left\{{\bm\tau}_{g_D}, {\bf e}_{\theta}\right\}\quad\mbox{for}\quad {\bm\tau}_{g_D}:=\frac{{\bf e}_x+g_D'(x){\bf e}_r}{\sqrt{1+|g_D'(x)|^2}}.$$ Then, it follows from the condition ${\bf u}\cdot{\bf n}_{g_D}=0$ on $\Gamma_{g_D}$ that
\begin{equation}\label{ab-u}
|{\bf u}|^2=|{\bf u}\cdot{\bm\tau}_{g_D}|^2+|{\bf u}\cdot{\bf e}_{\theta}|^2\quad\mbox{on}\quad\Gamma_{g_D}.
\end{equation}
By substituting the expression \eqref{3D-u} into \eqref{ab-u}, we get
\begin{equation}\label{3D-u2}
|{\bf u}|^2=\left|\left[\left(\partial_x\varphi+\frac{1}{r}\partial_r(r\psi)\right){\bf e}_x+(\partial_r\varphi-\partial_x\psi){\bf e}_r\right]\cdot{\bm\tau}_{g_D}\right|^2+\left|\frac{\Lambda}{r}\right|^2\mbox{ on }\Gamma_{g_D}.
\end{equation}
On the other hand, to satisfy the condition (f) stated in Problem \ref{3D-Problem2}, ${\bf u}$ should satisfy
\begin{equation}\label{3D-u22}
\begin{split}
|{\bf u}|^2
&=2\left(B_0^--\frac{\gamma p^{1-1/\gamma}S^{1/\gamma}}{\gamma-1}\right)\quad\mbox{on}\quad\Gamma_{g_D}.
\end{split}
\end{equation}
Therefore, if $(\varphi, \psi)$ satisfy
\begin{equation}\label{BC-vphipsi}
\nabla\varphi\cdot{\bm\tau}_{g_D}=\nabla\varphi_0\cdot{\bm\tau}_{g_D}\quad\mbox{and}\quad\frac{1}{r}\nabla(r\psi)\cdot{\bf n}_{g_D}=\mathcal{B}(g_D,g_D',S,\Lambda)
\end{equation}
for $\varphi_0$ and $\mathcal{B}$ defined by
\begin{equation}\label{def-varphi0-B}
\left.\begin{split}
&\varphi_0({\bf x}):=u_0x_1\quad\mbox{for}\quad{\bf x}=(x_1,x_2,x_3)\in\overline{\mathcal{N}_{g_D}^-},\\
&\mathcal{B}(g_D,g_D',S,\Lambda):=\sqrt{2\left(B_0^--\frac{\gamma p_0^{1-1/\gamma}S^{1/\gamma}}{\gamma-1}\right)-\left(\frac{\Lambda}{g_D}\right)^2}-\nabla\varphi_0\cdot{\bm\tau}_{g_D},
\end{split}\right.
\end{equation}
then one can directly check from \eqref{ab-u}--\eqref{3D-u22} that the condition $p=p_0$ on $\Gamma_{g_D}$ given in  \eqref{Prob2-BC-Cont} holds for $({\bf u}, p)$ given by \eqref{urp-qHS}.

We collect all the boundary conditions for $(g_D,S,\Lambda,\varphi,\psi)$ with \eqref{g-free-cond} as follows:
\begin{equation}\label{3D-BC-C}
\left\{
\begin{split}
(S,\Lambda)=(S_{\rm en},r\nu_{\rm en}),\quad\varphi=\varphi_{\rm en},\quad\partial_x\psi=0\quad&\mbox{on}\quad\Gamma_{\rm en}^-,\\
\nabla\varphi\cdot{\bm\tau}_{g_D}=\nabla\varphi_0\cdot{\bm\tau}_{g_D},\quad \frac{1}{r}\nabla(r\psi)\cdot{\bf n}_{g_D}=\mathcal{B}(g_D,g_D',S,\Lambda)\quad&\mbox{on}\quad\Gamma_{g_D}.\\
\end{split}\right.
\end{equation}

\begin{theorem}\label{3D-Thm-HD}
For given radial functions
$S_{\rm en}(r)$, $\nu_{\rm en}(r)$ and $u_r^{\rm en}(r)$ on $\Gamma_{\rm en}$, assume that they satisfy \eqref{def-entrance-ep}, and let  $\sigma(S_{\rm en}, \nu_{\rm en}, u_r^{\rm en})$ be given by \eqref{definition-sigma}. For simplicity of notations, let $\sigma$ denote  $\sigma(S_{\rm en}, \nu_{\rm en}, u_r^{\rm en})$.

For any fixed $\alpha\in(0,1)$, there exists a small constant $\sigma_3>0$ depending only on $(u_0,\rho_0^-,p_0,S_0^-)$ and $\alpha$ so that if
\begin{equation}\label{enu-sigma}
\sigma\le\sigma_3,
\end{equation}
then the free boundary problem \eqref{3D-H} with boundary conditions \eqref{g-free-cond} and \eqref{3D-BC-C} has a solution $(g_D, S, \Lambda, \varphi, \psi)$ that satisfies
\begin{equation}\label{Thm-HD-est}
\begin{split}
\|g_D-\frac{1}{2}\|_{2,\alpha,\mathbb{R}^+}&\le C\sigma,\\
\|\varphi-\varphi_0\|_{2,\alpha,\mathcal{N}^-_{g_D}}+\|\psi{\bf e}_{\theta}\|_{2,\alpha,\mathcal{N}^-_{g_D}}+\|(S,\Lambda)-(S_0^-,0)\|_{1,\alpha,\mathcal{N}^-_{g_D}}&\le C\sigma,
\end{split}
\end{equation}
where the constant $C>0$ depends only on $(u_0,\rho_0^-,p_0,S_0^-)$ and $\alpha$.
\end{theorem}

Hereafter, a constant $C$ is said to be chosen depending only on the data if $C$ is chosen depending only on $(u_0,\rho_0^-,p_0,S_0^-)$.

We first prove Theorem \ref{3D-Thm-HD}, then apply  this theorem to prove Theorem  \ref{3D-MainThm}.
We will prove Theorem \ref{3D-Thm-HD} by a limiting argument. So we introduce a free boundary problem in a cut-off domain of the finite length $L$, and solve it by the method of iteration in Section \ref{3D-sec-Cut}.
And, uniform estimates of the solutions to the free boundary problems in cut-off domains are established independently of the length $L$. In Section \ref{5-1}, we prove Theorem \ref{3D-Thm-HD} by taking a sequence of the solutions to the free boundary problems in cut-off domains, then passing to the limit $L\rightarrow\infty$.
The limit yields a solution to the free boundary problem \eqref{3D-H} with boundary conditions \eqref{g-free-cond} and \eqref{3D-BC-C}, then we can prove that $(g_D, {\bf u}, \rho, p)$ for $({\bf u}, \rho, p)$ given by \eqref{urp-qHS} yields a solution to Problem \ref{3D-Problem2}. This proves Theorem \ref{3D-MainThm}(a).
Finally, Theorem \ref{3D-MainThm}(b) is proved by using the stream function formulation and energy estimates.

\section{Free boundary problems in  cut-off domains}\label{3D-sec-Cut}

\subsection{Iteration framework}
Let $\mathcal{N}$ be  given by \eqref{def-N-cyl}.
For a constant $L>0$, define  $\mathcal{N}_L$  by
\begin{equation*}
\mathcal{N}_L:=\mathcal{N}\cap\{0<x<L\}.
\end{equation*}
For a function $f:[0,L]\rightarrow(0,1)$, we set
\begin{equation*}
\begin{split}
&\mathcal{N}_{L,f}^-:=\mathcal{N}_L\cap\{r<f(x)\},\\
&\Gamma^{L,f}_{\rm ex}:=\partial\mathcal{N}_{L,f}^-\cap\{x=L\},\quad
\Gamma_{\rm cd}^{L,f}:=\partial\mathcal{N}_{L,f}^-\cap\{r=f(x)\}.
\end{split}
\end{equation*}

\begin{problem}\label{Prob3-Cut} Find a solution $(f,S,\Lambda,\varphi,\psi)$ of the following free boundary problem:
\begin{equation*}\label{3D-Cut-Eq}
\eqref{3D-H}\quad\mbox{in}\quad\mathcal{N}_{L,f}^-
\end{equation*}
 with boundary conditions
\begin{equation}\label{3D-Cut-BC}
\left\{\begin{split}
(S,\Lambda)=(S_{\rm en},r\nu_{\rm en}),\quad\varphi=\varphi_{\rm en},\quad\partial_x\psi=0\quad&\mbox{on}\quad\Gamma_{\rm en}^-,\\
\partial_r\varphi=0,\quad\partial_x\psi=0\quad&\mbox{on}\quad\Gamma^{L,f}_{\rm ex},\\
\nabla\varphi\cdot{\bm\tau}_f=\nabla\varphi_0\cdot{\bm\tau}_f,\quad\frac{1}{r}\nabla(r\psi)\cdot{\bf n}_{f}=\mathcal{B}(f,f',S,\Lambda)\quad&\mbox{on}\quad\Gamma_{\rm cd}^{L,f},
\end{split}\right.
\end{equation}
and
\begin{equation}\label{g-free-cut}
\left\{\begin{split}
&f'(x)=\frac{{\bf q}(r,\psi,D\psi,D\varphi,\Lambda)\cdot{\bf e}_r}{{\bf q}(r,\psi,D\psi,D\varphi,\Lambda)\cdot{\bf e}_x}(x,f(x),0)\quad\mbox{for}\quad x>0,\\
&f(0)=\frac{1}{2},
\end{split}\right.
\end{equation}
where
$${\bm \tau}_f:=\frac{{\bf e}_x+f'(x){\bf e}_r}{\sqrt{1+|f'(x)|^2}},\quad {\bf n}_f:=\frac{-f'(x){\bf e}_x+{\bf e}_r}{\sqrt{1+|f'(x)|^2}}.$$
\end{problem}

\begin{proposition}\label{3D-Prop4.1}
For given radial functions
$S_{\rm en}(r)$, $\nu_{\rm en}(r)$ and $u_r^{\rm en}(r)$ on $\Gamma_{\rm en}$, assume that they satisfy \eqref{def-entrance-ep}, and let  $\sigma(S_{\rm en}, \nu_{\rm en}, u_r^{\rm en})$ be given by \eqref{definition-sigma}. For simplicity of notations, let $\sigma$ denote  $\sigma(S_{\rm en}, \nu_{\rm en}, u_r^{\rm en})$.

For a fixed $\alpha\in(0,1)$, there exists a small constant $\sigma_4>0$ depending only on the data and $\alpha$ so that if
\begin{equation*}
\sigma\le\sigma_4,
\end{equation*}
then Problem \ref{Prob3-Cut} has a unique solution $(f,S,\Lambda,\varphi,\psi)$ that satisfies
\begin{equation}\label{3D-Prop-est}
\begin{split}
\|f-\frac{1}{2}\|_{2,\alpha,(0,L)}&\le C\sigma,\\
\|\varphi-\varphi_0\|_{2,\alpha,\mathcal{N}_{L,f}^-}+\|\psi{\bf e}_{\theta}\|_{2,\alpha,\mathcal{N}_{L,f}^-}+\|(S,\Lambda)-(S_0^-,0)\|_{1,\alpha,\mathcal{N}_{L,f}^-}&\le C\sigma,
\end{split}
\end{equation}
where the constant $C>0$ depends only on the data and $\alpha$ but independent of $L$.
\end{proposition}

In order to find  $(S, \Lambda)$ as a solution to transport equation $H(S,{\bf q}){\bf q}\cdot\nabla (S, \Lambda)=0$  in $\mathcal{N}_{L,f}^-$, we first need $(f, {\bf q})$ to satisfy the condition \eqref{g-free-cut}. Furthermore, the vector field $H(S,{\bf q}){\bf q}$ needs to be divergence free  (See \cite[Proposition 3.5]{bae20183}).
Therefore, we need to solve a free boundary problem for $(f, \varphi, \psi)$ by fixing approximated entropy and angular momentum density $(\tilde{S}, \tilde{\Lambda})$, then solve $H(\tilde S,{\bf \tilde q}){\bf \tilde q}\cdot\nabla (S, \Lambda)=0$  in $\mathcal{N}_{L,f}^-$ to update $(S, \Lambda)$, where ${\bf \tilde q}$ is given by ${\bf \tilde q}={\bf q}(r, \psi, D\psi, D\varphi, \tilde{\Lambda})$. This procedure yields an iteration map in the iterations sets defined below.

For fixed constants $\epsilon\in(0,1/10)$, $\alpha\in(0,1)$ and $M_1>0$ to be determined later, we define an iteration set
$$\mathcal{P}(M_1):=\mathcal{P}_1(M_1)\times\mathcal{P}_2(M_1)$$
for
\begin{equation}\label{Ent-Ang-set}
\left.
\begin{split}
&\mathcal{P}_1(M_1):=\left\{S=S(x,r)\in C^{1,\alpha/2}(\overline{\mathcal{N}_{L,3/4}^-}):
\begin{split}
&\|S-S_0^-\|_{1,\alpha,\mathcal{N}_{L,3/4}^-}\le M_1\sigma,\\
&\partial_xS\equiv0\mbox{ on }(\Gamma_{\rm en}^-\cap  \{r\ge\frac{1}{2}-\epsilon\})\cup\Gamma^{L,3/4}_{\rm ex}
\end{split}\right\},\\
&\mathcal{P}_2(M_1):=
\left\{\Lambda=r\mathcal{V}(x,r)\in C^{1,\alpha/2}(\overline{\mathcal{N}_{L,3/4}^-}):
\begin{split}
&\|\mathcal{V}\|_{1,\alpha,\Omega_{L,3/4}^-}\le M_1\sigma,\\
&\mathcal{V}(x,0)=0,\,\forall x\in[0,L],\\
&\partial_x\Lambda\equiv0\mbox{ on }(\Gamma_{\rm en}^-\cap \{r\ge\frac{1}{2}-\epsilon\})\cup\Gamma^{L,3/4}_{\rm ex}
\end{split}\right\}.
\end{split}\right.
\end{equation}
for a two dimensional rectangular domain $\Omega_{L,3/4}^-$ given by
\begin{equation*}
\Omega_{L,3/4}^-:=\left\{(x,r)\in\mathbb{R}^2: 0<x<L, \mbox{ }0<r<\frac{3}{4}\right\}.
\end{equation*}

 \begin{problem}\label{Prob4-Fix-S}
For each $\mathcal{W}_{\ast}:=(S_{\ast},\Lambda_{\ast})\in\mathcal{P}(M_1)$, set
\begin{equation*}
{\bf q}_*:={\bf q}(r,\psi,D\psi,D\varphi,\Lambda_{\ast}),\quad
{\bf t}_*:={\bf t}(r,\psi,D\psi,\Lambda_{\ast})
\end{equation*}
for $({\bf q}, {\bf t})$ given by \eqref{definition-q} and \eqref{def-T}.
Then, find $(f,\varphi,\psi)$ satisfying \eqref{g-free-cut} and
\begin{equation}\label{S-Free-BP}
\left\{\begin{split}
	\left.\begin{split}
	&\mbox{div}\left(H(S_{\ast},{\bf q}_*){\bf q}_*\right)=0\\
	&-\Delta(\psi{\bf e}_{\theta})=G(S_{\ast},\Lambda_{\ast},\partial_r S_{\ast},\partial_r\Lambda_{\ast},{\bf t}_*,D\varphi){\bf e}_{\theta}\\
	\end{split}\right.\quad&\mbox{in}\quad\mathcal{N}_{L,f}^-,\\
\varphi=\varphi_{\rm en},\quad\partial_x\psi=0\quad&\mbox{on}\quad\Gamma_{\rm en}^-,\\
\partial_r\varphi=0,\quad\partial_x\psi=0\quad&\mbox{on}\quad\Gamma^{L,f}_{\rm ex},\\
\nabla\varphi\cdot{\bm\tau}_f=\nabla\varphi_0\cdot{\bm\tau}_f,\quad\frac{1}{r}\nabla(r\psi)\cdot{\bf n}_{f}=\mathcal{B}(f,f',S_{\ast},\Lambda_{\ast})\quad&\mbox{on}\quad\Gamma_{\rm cd}^{L,f},
\end{split}\right.
\end{equation}
where $H$, $G$, and $\mathcal{B}$ are given by \eqref{def-H-G} and \eqref{def-varphi0-B}.
\end{problem}
\begin{lemma}\label{Lem-S-free}
Under the same assumptions on $(S_{\rm en},\nu_{\rm en},u_r^{\rm en})$ as in Proposition \ref{3D-Prop4.1},
there exists a small constant $\sigma_5>0$ depending only on the data and $(\alpha, M_1)$ so that if
$$\sigma\le\sigma_5,$$
then, for each $\mathcal{W}_{\ast}\in\mathcal{P}(M_1)$, Problem \ref{Prob4-Fix-S} has a unique solution $(f,\varphi,\psi)$ satisfying
\begin{equation}\label{3D-pps-est}
\|f-\frac{1}{2}\|_{2,\alpha,(0,L)}+\|\varphi-\varphi_0\|_{2,\alpha,\mathcal{N}_{L,f}^-}+\|\psi{\bf e}_{\theta}\|_{2,\alpha,\mathcal{N}_{L,f}^-}\le C\left(M_1+1\right)\sigma,
\end{equation}
where the constant $C>0$ depends only on the data and $\alpha$ but independent of $L$.
\end{lemma}

We will prove this lemma in Section \ref{subsection_4_2}.
Once Lemma \ref{Lem-S-free} is proved, we prove Proposition \ref{3D-Prop4.1} by the following approach:
Let $(f,\varphi,\psi)$ be the unique solution of Problem \ref{Prob4-Fix-S} for a fixed $\mathcal{W}_{\ast}=(S_{\ast},\Lambda_{\ast})\in\mathcal{P}(M_1)$.
For such a solution, we find a unique solution $\mathcal{W}=(S,\Lambda)$ of the following initial value problem:
\begin{equation}\label{Ite-3D2}
\left\{\begin{split}
	\left.\begin{split}
	&H(S_{\ast},{\bf q}(r,\psi,D\psi,D\varphi,\Lambda_{\ast})){\bf q}(r,\psi,D\psi,D\varphi,\Lambda_{\ast})\cdot\nabla S=0\\
	&H(S_{\ast},{\bf q}(r,\psi,D\psi,D\varphi,\Lambda_{\ast})){\bf q}(r,\psi,D\psi,D\varphi,\Lambda_{\ast})\cdot\nabla \Lambda=0
	\end{split}\right.\quad&\mbox{in}\quad\mathcal{N}_{L,f}^-,\\
(S,\Lambda)=(S_{\rm en},r\nu_{\rm en})\quad&\mbox{on}\quad\Gamma_{\rm en}^-.
\end{split}\right.
\end{equation}
We take suitable extensions $\mathcal{E}_f(S,\Lambda)\in [C^{1,\alpha/2}(\overline{\mathcal{N}_{L,3/4}^-})]^2$ of $(S,\Lambda)$.
For such $\mathcal{E}_f(S,\Lambda)$, we define an iteration mapping $\mathcal{J}:\mathcal{P}(M_1)\rightarrow [C^{1,\alpha/2}(\overline{\mathcal{N}_{L,3/4}^-})]^2$ by
\begin{equation*}
\mathcal{J}(S_{\ast},\Lambda_{\ast})=\mathcal{E}_f(S,\Lambda).
\end{equation*}
Then we choose $M_1$ and $\sigma$ so that the mapping $\mathcal{J}$ maps $\mathcal{P}(M_1)$ into itself, and has a unique fixed point  $(S_{\sharp},\Lambda_{\sharp})\in\mathcal{P}(M_1)$ of $\mathcal{J}$. This will prove Proposition \ref{3D-Prop4.1}. A detailed proof of Proposition \ref{3D-Prop4.1} is given in Section \ref{subsection_4_3}.

\subsection{Proof of Lemma \ref{Lem-S-free} }
\label{subsection_4_2}

For a constant $M_2>0$ to be determined later with satisfying $M_2\sigma\le \frac{1}{8}$, we define an iteration set
\begin{equation}\label{F-set}
\mathcal{F}(M_2):=\left\{f\in C^{2,\alpha}([0,L]):
\begin{split}
&\|f-\frac{1}{2}\|_{2,\alpha,(0,L)}\le M_2\sigma, \\
&f(0)=\frac{1}{2},\,f'(0)=f'(L)=0
\end{split}\right\}.
\end{equation}
We fix $f_{\ast}\in\mathcal{F}(M_2)$, and solve the following boundary value problem in $\mathcal{N}_{L,f_{\ast}}^-$:
\begin{equation}\label{Fixed-BVP}
\left\{\begin{split}
	\left.\begin{split}
	&\mbox{div}\left(H(S_{\ast},{\bf q}(r,\psi,D\psi,D\varphi,\Lambda_{\ast})){\bf q}(r,\psi,D\psi,D\varphi,\Lambda_{\ast})\right)=0\\
	&-\Delta(\psi{\bf e}_{\theta})=G(S_{\ast},\Lambda_{\ast},\partial_r S_{\ast},\partial_r\Lambda_{\ast},{\bf t}(r,\psi,D\psi,\Lambda_{\ast}),D\varphi){\bf e}_{\theta}\\
	\end{split}\right.\quad&\mbox{in}\quad\mathcal{N}_{L,f_{\ast}}^-,\\
\varphi=\varphi_{\rm en},\quad\partial_x\psi=0\quad&\mbox{on}\quad\Gamma_{\rm en}^-,\\
\partial_r\varphi=0,\quad\partial_x\psi=0\quad&\mbox{on}\quad\Gamma^{L,f_{\ast}}_{\rm ex},\\
\nabla\varphi\cdot{\bm\tau}_{f_{\ast}}=\nabla\varphi_0\cdot{\bm\tau}_{f_\ast},\quad\frac{1}{r}\nabla(r\psi)\cdot{\bf n}_{f_\ast}=\mathcal{B}(f_{\ast},f'_{\ast},S_{\ast},\Lambda_{\ast})\quad&\mbox{on}\quad\Gamma_{\rm cd}^{L,f_{\ast}},
\end{split}\right.
\end{equation}
where $H$, $G$, and $\mathcal{B}$ are given by \eqref{def-H-G} and \eqref{def-varphi0-B}.

\begin{lemma} \label{Pro-fix-S}
Under the same assumptions on $(S_{\rm en},\nu_{\rm en},u_r^{\rm en})$ as in Proposition \ref{3D-Prop4.1},
there exists a small constant $\sigma_6>0$ depending only on the data and $(\alpha, M_1, M_2)$ so that if
\begin{equation*}
\sigma\le\sigma_6,
\end{equation*}
then the nonlinear boundary value problem \eqref{Fixed-BVP} has a unique solution $(\varphi,\psi)$ satisfying
\begin{equation}\label{fix-est}
\|\varphi-\varphi_0\|_{2,\alpha,\mathcal{N}_{L,f_{\ast}}^-}+\|\psi{\bf e}_{\theta}\|_{2,\alpha,\mathcal{N}_{L,f_{\ast}}^-}\le C\left(M_1+1\right)\sigma,
\end{equation}
where the constant $C>0$ depends only on the data and $\alpha$ but independent of $L$.
\end{lemma}

Hereafter, we regard any estimate constant $C$ to be chosen depending only on the data and $\alpha$ but independent of $L$ unless specified otherwise.

\begin{proof}
{\bf 1.} {\emph{(Iteration set)}} For two constants $M_3, M_4>0$ to be determined later,  let us define
\begin{equation*}
\left.
\begin{split}
&\mathcal{K}_1^{f_\ast}(M_3):=\left\{\phi(x,r)\in C^{2,\alpha}(\overline{\mathcal{N}_{L,f_\ast}^-}):\left.
\begin{split}
&\|\phi\|_{2,\alpha,\mathcal{N}_{L,f_\ast}^-}\le M_3\sigma,\\
&\partial_x^k\phi\equiv0\mbox{ on }(\Gamma_{\rm en}^-\cap\{r\ge\frac{1}{2}-\epsilon\})\cup\Gamma^{L,f_{\ast}}_{\rm ex}\\
&\mbox{\quad for }k=0,2,\\
&\phi(x,0)=0, \forall x\in[0,L]
\end{split}\right.\right\},\\
&\mathcal{K}_2^{f_\ast}(M_4):=\left\{\psi(x,r)\in C^{2,\alpha}(\overline{\Omega_{L,f_\ast}^-}):
\begin{split}
& \|\psi\|_{2,\alpha,\Omega_{L,f_{\ast}}^-}\le M_4\sigma,\\
&\partial_x\psi\equiv0\mbox{ on }\Gamma_{\rm en}^-\cup\Gamma^{L,f_{\ast}}_{\rm ex},\\
&\partial_r^k \psi(x,0)=0\mbox{ for }k=0,2, \forall x\in[0,L]
\end{split}\right\},
\end{split}\right.
\end{equation*}
for a two dimensional set $\Omega_{L,f_{\ast}}^-$ given by
\begin{equation*}
\Omega_{L,f_{\ast}}^-:=\left\{(x,r)\in\mathbb{R}^2:0<x<L,\mbox{ }0<r<f_{\ast}(x)\right\}.
\end{equation*}
Then, we define an iteration set of $(\phi, \psi)$ as
\begin{equation}\label{ell-set}
\mathcal{K}^{f_\ast}(M_3,M_4):=\mathcal{K}_1^{f_\ast}(M_3)\times\mathcal{K}_2^{f_\ast}(M_4).
\end{equation}
Note that the iteration set $\mathcal{K}_2^{f_\ast}(M_4)$ is defined through the norm $ \|\cdot\|_{2,\alpha,\Omega_{L,f_{\ast}}^-}$. This is to find an axisymmetric solution $\psi(x,r)$ to the equation
$-\Delta (\psi {\bf e}_{\theta})=G{\bf e}_{\theta}$ given in \eqref{Fixed-BVP}, and to make the function $\psi(x,r){\bf e}_{\theta}$ become $C^2$ in $\mathcal{N}_{L,f_\ast}^-$.

{\bf 2.} {\emph{(Linearized boundary value problem for $\psi{\bf e}_{\theta}$)}} For a fixed $(\tilde{\phi},\tilde{\psi})\in\mathcal{K}^{f_\ast}(M_3,M_4)$, set
 \begin{equation}\label{def-GB}
 \tilde{G}:=G(\mathcal{W}_{\ast},\partial_r\mathcal{W}_{\ast},{\bf t}(r,\tilde{\psi},D\tilde{\psi},\Lambda_{\ast}),D\tilde{\phi}+D\varphi_0),\quad \tilde{\mathcal{B}}:=\mathcal{B}(f_{\ast},f_{\ast}',\mathcal{W}_{\ast}),
 \end{equation}
 where $G$ and $\mathcal{B}$ are given by \eqref{def-H-G} and \eqref{def-varphi0-B}, respectively.
 The compatibility condition $\partial_r\mathcal{W}_{\ast}\equiv {\bf 0}$ on $\mathcal{N}_{L,f_{\ast}}^-\cap\{r=0\}$ implies that $\tilde{G}\equiv 0$ on $\mathcal{N}_{L,f_{\ast}}^-\cap\{r=0\}$ and
 $\tilde{G}{\bf e}_{\theta}\in C^{\alpha}(\overline{\mathcal{N}_{L,f_{\ast}}^-})$.
 Since $f_{\ast}\ge\frac{3}{8}>0$, ${\bf e}_{\theta}$ is smooth on $\Gamma_{\rm cd}^{L,f_{\ast}}$ and $\tilde{\mathcal{B}}{\bf e}_{\theta}\in C^{1,\alpha}(\Gamma_{\rm cd}^{L,f_{\ast}})$.
Then the standard elliptic theory yields that the linear boundary value problem
\begin{equation}\label{lin-psi}
\left\{\begin{split}
-\Delta{\bf W}=\tilde{G}{\bf e}_{\theta}\quad&\mbox{in}\quad\mathcal{N}_{L,f_\ast}^-,\\
\partial_x{\bf W}=0\quad&\mbox{on}\quad\Gamma_{\rm en}^-\cup\Gamma^{L,f_{\ast}}_{\rm ex},\\
\nabla{\bf W}\cdot{\bf n}_{f_\ast}-\mu(x){\bf W}=\tilde{\mathcal{B}}{\bf e}_{\theta}\quad&\mbox{on}\quad\Gamma_{\rm cd}^{L,f_\ast},
\end{split}\right.
\end{equation}
for $\mu$ defined by
$$\mu(x):=\frac{-1}{f_{\ast}(x)\sqrt{1+|f'_{\ast}(x)|^2}},$$
has a unique solution ${\bf W}\in C^{1,\alpha}(\overline{\mathcal{N}_{L,f_\ast}^-})\cap C^{2,\alpha}({\mathcal{N}_{L,f_{\ast}}^-})$.
By adjusting the proof of \cite[Proposition 3.3]{bae20183}, one can show that  ${\bf W}$ is represented as
$${\bf W}=\psi(x,r){\bf e}_{\theta}\quad\mbox{in}\quad \overline{\mathcal{N}_{L,f_{\ast}}^-},$$
where $\psi$ solves the boundary value problem
\begin{equation}\label{3D-psi}
-\left(\partial_{xx}+\frac{1}{r}\partial_r(r\partial_r)-\frac{1}{r^2}\right)\psi=\tilde{G}\quad\mbox{in}\quad\mathcal{N}_{L,f_{\ast}}^-
\end{equation}
with boundary conditions
\begin{equation}\label{3D-psi-BC}\left\{
\begin{split}
-\partial_x\psi=0\quad&\mbox{on}\quad\Gamma_{\rm en}^-,\\
\partial_x\psi=0\quad&\mbox{on}\quad\Gamma^{L,f_{\ast}}_{\rm ex},\\
\frac{1}{r}\nabla(r\psi)\cdot{\bf n}_{f_{\ast}}=\tilde{\mathcal{B}}\quad&\mbox{on}\quad\Gamma_{\rm cd}^{L,f_\ast},\\
\psi=0\quad&\mbox{on}\quad\mathcal{N}_{L,f_\ast}^-\cap\{r=0\}.
\end{split}\right.
\end{equation}
By taking the limit $r\rightarrow 0+$ to the equation \eqref{3D-psi} and using L'Hospital's rule, one can also check that
$$\partial_{rr}\psi\equiv0\quad\mbox{on}\quad\mathcal{N}_{L,f_\ast}^-\cap\{r=0\}.$$

{\bf Claim:}  Regarding $\psi$ as a function of $(x,r)\in\Omega_{L,f_{\ast}}^-$, we have
 \begin{equation}\label{psi-est-2}
 \begin{split}
 \|\psi\|_{k,\alpha,\Omega_{L,f_{\ast}}^-}
 &\le C\left(\|\tilde{G}\|_{\alpha,\mathcal{N}_{L,f_\ast}^-}+\|\tilde{\mathcal{B}}\|_{k-1,\alpha,\Gamma_{\rm cd}^{L,f_\ast}}\right)\quad\mbox{for } k=1,2.\\
\end{split}
\end{equation}

\begin{proof}[Proof of Claim.]
Since $\psi={\bf W}\cdot{\bf e}_{\theta}$, and ${\bf e}_{\theta}$ is smooth with respect to $(x,r,\theta)$, we have
\begin{equation}\label{psi-W-est}
\|\psi\|_{k,\alpha,\Omega_{L,f_{\ast}}^-}\le C\|{\bf W}\|_{k,\alpha,\mathcal{N}_{L,f_{\ast}}^-}\quad\mbox{for }k=1,2.
\end{equation}

Now we show that ${\bf W}$ satisfies
\begin{equation*}
\begin{split}
&\|{\bf W}\|_{k,\alpha,\mathcal{N}_{L,f_\ast}^-}\le C\left(\|\tilde{G}{\bf e}_{\theta}\|_{\alpha,\mathcal{N}_{L,f_\ast}^-}+\|\tilde{\mathcal{B}}{\bf e}_{\theta}\|_{k-1,\alpha,\Gamma_{\rm cd}^{L,f_\ast}}\right)\quad\mbox{for }k=1,2.
\end{split}
\end{equation*}
Define a function $\mathfrak{N}:\overline{\mathcal{N}_{L,f_\ast}^-}\rightarrow\mathbb{R}^+$ by
$$\mathfrak{N}(x_1,x_2,x_3):={2\mathfrak{a}}\left(-x_2^2+5\right)\quad\mbox{for}\quad \mathfrak{a}:=\|\tilde{G}{\bf e}_{\theta}\|_{0,\mathcal{N}_{L,f_\ast}^-}+\|\tilde{\mathcal{B}}{\bf e}_{\theta}\|_{0,\Gamma_{\rm cd}^{L,f_\ast}}.$$
Since $\|f_{\ast}-\frac{1}{2}\|_{2,\alpha,(0,L)}\le \frac{1}{8}$, we have
\begin{equation}\label{N-est}
\begin{split}
\nabla\mathfrak{N}\cdot{\bf n}_{f_{\ast}}-\mu(x)\mathfrak{N}
&=\frac{2\mathfrak{a}}{\sqrt{1+|f_{\ast}'(x_1)|^2}}\left(-2x_2\cos\theta+\frac{-x_2^2+5}{f_{\ast}(x_1)}\right)\\
&\ge \frac{2\mathfrak{a}}{2}\left(-2+4\right)= 2\mathfrak{a}\quad\mbox{on}\quad\Gamma_{\rm cd}^{L,f_\ast}.
\end{split}
\end{equation}
Set
\begin{equation*}
W_j:={\bf W}\cdot{\bf e}_j,\quad G_j:=\tilde{G}{\bf e}_{\theta}\cdot{\bf e}_j,\quad B_j:=\tilde{\mathcal{B}}{\bf e}_{\theta}\cdot{\bf e}_j\quad\mbox{for }j=1,2,3.
\end{equation*}
Here, each ${\bf e}_j$ for $j=1,2,3$ denotes the unit vector in the positive direction of $x_j$-axis for ${\bf x}=(x_1,x_2,x_3)\in\overline{\mathcal{N}_{L,f_{\ast}}^-}$.
Then straightforward computations and \eqref{N-est} yield that
\begin{equation*}
\left\{\begin{split}
\Delta(\mathfrak{N}\pm W_j)=-4\mathfrak{a}\pm{G}_j \le 0\quad&\mbox{in}\quad\mathcal{N}_{L,f_\ast}^-,\\
\partial_x(\mathfrak{N}\pm W_j)=0\quad&\mbox{on}\quad\Gamma_{\rm en}^-\cup\Gamma^{L,f_{\ast}}_{\rm ex},\\
\nabla(\mathfrak{N}\pm W_j)\cdot{\bf n}_{f_{\ast}}-\mu(x)(\mathfrak{N}\pm W_j)\ge 2\mathfrak{a}\pm B_j\ge 0\quad&\mbox{on}\quad\Gamma_{\rm cd}^{L,f_\ast}.
\end{split}\right.
\end{equation*}
By the comparison principle and Hopf's lemma, we have
$$-\mathfrak{N}\le W_j\le\mathfrak{N}\quad\mbox{for }j=1,2,3\quad\mbox{in}\quad\mathcal{N}_{L,f_{\ast}}^-.$$
Therefore we get the estimate
\begin{equation}
\label{W-estimate-C0}
\|{\bf W}\|_{0,\mathcal{N}_{L,f_\ast}^-}\le C\left(\|\tilde{G}{\bf e}_{\theta}\|_{0,\mathcal{N}_{L,f_\ast}^-}+\|\tilde{\mathcal{B}}{\bf e}_{\theta}\|_{0,\Gamma_{\rm cd}^{L,f_\ast}}\right).
\end{equation}
By adjusting the proof of \cite[Theorem 3.13]{han2011elliptic} with using the $C^0$-estimate given right above, we obtain the estimate
\begin{equation}\label{W-est-1}
\|{\bf W}\|_{1,\alpha,\mathcal{N}_{L,f_\ast}^-}\le C\left(\|\tilde{G}{\bf e}_{\theta}\|_{\alpha,\mathcal{N}_{L,f_\ast}^-}+\|\tilde{\mathcal{B}}{\bf e}_{\theta}\|_{\alpha,\Gamma_{\rm cd}^{L,f_\ast}}\right).
\end{equation}
To obtain $C^{2,\alpha}$-estimate of ${\bf W}$ up to the boundary,
we use the method of reflection.
Define an extension of $f_{\ast}\in\mathcal{F}(M_2)$ into $-1\le x\le L+1$ by
\begin{equation*}
f_{\ast}^e(x):=\left\{\begin{split}
f_{\ast}(-x)\quad&\mbox{for }-1\le x<0,\\
f_{\ast}(x)\quad&\mbox{for }0\le x\le L,\\
f_{\ast}(2L-x)\quad&\mbox{for }L< x\le L+1.
\end{split}\right.
\end{equation*}
Since $f_{\ast}'(0)=f_{\ast}'(L)=0$, we have the estimate
$$\|f_{\ast}^e\|_{2,\alpha,(-1,L+1)}\le C\|f_{\ast}\|_{2,\alpha,(0,L)}.$$
We define an extended domain
\begin{equation*}\label{3D-N-ext}
\mathcal{N}_{\rm ext}:=\left\{(x_1,x_2,x_3)\in\mathbb{R}^3: -1<x_1<L+1,\mbox{ } 0\le \sqrt{x_2^2+x_3^2}< f^e_{\ast}(x_1)\right\}
\end{equation*}
and
$$\Gamma_{\rm ext}:=\partial\mathcal{N}_{\rm ext}\cap\left\{\sqrt{x_2^2+x_3^2}=f_{\ast}^e(x_1)\right\}.$$
We also define extensions of $({\bf W}, \tilde{G}, \tilde{\mathcal{B}})$ into $\mathcal{N}_{\rm ext}$ as follows:
\begin{equation*}
\begin{split}
	\left({\bf W}_{\rm ext},\mathfrak{G}_{\rm ext},\mathfrak{B}_{\rm ext}\right)({\bf x})
	&:=\left\{\begin{split}
	\left({\bf W},\tilde{G},\tilde{\mathcal{B}}\right)(-x_1,x_2,x_3)\quad&\mbox{for }-1\le x_1<0,\\
	\left({\bf W},\tilde{G},\tilde{\mathcal{B}}\right)(x_1,x_2,x_3)\quad&\mbox{for }0\le x_1\le L,\\
	\left({\bf W},\tilde{G},\tilde{\mathcal{B}}\right)(2L-x_1,x_2,x_3)\quad&\mbox{for }L< x_1\le L+1,
	\end{split}\right.\\
	\end{split}
\end{equation*}
Then $\mathfrak{G}_{\rm ext}{\bf e}_{\theta}\in  C^{\alpha}(\overline{\mathcal{N}_{\rm ext}})$ and
\begin{equation*}
\|\mathfrak{G}_{\rm ext}{\bf e}_{\theta}\|_{\alpha,\mathcal{N}_{\rm ext}}\le C\|\tilde{G}{\bf e}_{\theta}\|_{\alpha,\mathcal{N}_{L,f_\ast}^-}.
\end{equation*}
By the compatibility conditions of  $(\mathcal{W}_{\ast}, f_{\ast})$ given in \eqref{Ent-Ang-set} and \eqref{F-set},
$$\nabla\mathfrak{B}_{\rm ext}\cdot{\bm \tau}_{f_\ast}\equiv0\quad\mbox{on}\quad\Gamma_{\rm ext}\cap\{x_1=0,L\}.$$
From this and the definition of $\mathfrak{B}_{\rm ext}$, we have the estimate
\begin{equation*}
\|\mathfrak{B}_{\rm ext}{\bf e}_{\theta}\|_{1,\alpha,\Gamma_{\rm ext}}\le C\|\tilde{\mathcal{B}}{\bf e}_{\theta}\|_{1,\alpha,\Gamma_{\rm cd}^{L,f_\ast}}.
\end{equation*}

Consider a connected subdomain $\mathcal{N}_l$ of $\mathcal{N}_{\rm ext}$ such that
\begin{equation*}
\mathcal{N}_{\rm ext}\cap\left\{-\frac{1}{2}\le x_1\le \frac{1}{2}\right\}\subset\mathcal{N}_l\subset\mathcal{N}_{\rm ext}\cap\left\{-1\le x_1\le 1\right\}
\end{equation*}
and the boundary $\partial\mathcal{N}_l$ is smooth.
By the standard elliptic theory, the boundary value problem
\begin{equation}\label{3D-MP}
\left\{\begin{split}
-\Delta\mathfrak{W}=\mathfrak{G}_{\rm ext}{\bf e}_{\theta}\quad&\mbox{in}\quad\mathcal{N}_{l},\\
\nabla\mathfrak{W}\cdot{\bf n}_{f_{\ast}^e}-\mu_{\rm ext}(x)\mathfrak{W}=\mathfrak{B}_{\rm ext}{\bf e}_{\theta}\quad&\mbox{on}\quad\partial\mathcal{N}_{l}\cap\{r=f^e_{\ast}(x)\},\\
\mathfrak{W}={\bf W}_{\rm ext}\quad&\mbox{on}\quad\partial\mathcal{N}_{l}\backslash\{r=f_{\ast}^e(x)\},
\end{split}\right.
\end{equation}
for
\begin{equation*}
{\bf n}_{f_{\ast}^e}:=\frac{-(f_{\ast}^e)'(x){\bf e}_x+{\bf e}_r}{\sqrt{1+|(f_{\ast}^e)'(x)|^2}},\quad \mu_{\rm ext}(x):=\frac{-1}{f_{\ast}^e(x)\sqrt{1+|(f_{\ast}^e)'(x)|^2}},
\end{equation*}
has a unique solution $\mathfrak{W}\in C^{2,\alpha}(\overline{\mathcal{N}_{l}})$ that satisfies
$$\|\mathfrak{W}\|_{2,\alpha,\mathcal{N}_{\rm ext}\cap\left\{-\frac{1}{2}\le x_1\le \frac{1}{2}\right\}}\le C\left(\|\mathfrak{G}_{\rm ext}{\bf e}_{\theta}\|_{\alpha,\mathcal{N}_{l}}+\|\mathfrak{B}_{\rm ext}{\bf e}_{\theta}\|_{1,\alpha,\partial\mathcal{N}_{l}\cap\{r=f_\ast^e(x)\}}+\|{\bf W}\|_{C^0(\overline{\mathcal{N}_{L,f_\ast}^-})}\right).$$
By the definitions of $(\mathfrak{G}_{\rm ext}, \mathfrak{B}_{\rm ext},{\bf W}_{\rm ext})$ and the uniqueness of a solution to \eqref{3D-MP}, we have
$\mathfrak{W}(x_1,x_2,x_3)=\mathfrak{W}(-x_1,x_2,x_3)$ and $\partial_{x_1}\mathfrak{W}(0,x_2,x_3)=0$.
The uniqueness of a solution to \eqref{lin-psi} yields that $\mathfrak{W}={\bf W}$ in $\mathcal{N}_l\cap\{x_1\ge 0\}.$
By combining \eqref{W-estimate-C0} and the $C^{2,\alpha}$-estimate of $\mathfrak{W}$ given right above, we obtain that
\begin{equation}\label{left-W}
\|{\bf W}\|_{2,\alpha,\mathcal{N}_{L,f_{\ast}}^-\cap\left\{0\le x_1\le \frac{1}{2}\right\}}\le C\left(\|\tilde{G}{\bf e}_{\theta}\|_{\alpha,\mathcal{N}_{L,f_\ast}^-}+\|\tilde{\mathcal{B}}{\bf e}_{\theta}\|_{1,\alpha,\Gamma_{\rm cd}^{L,f_\ast}}\right).
\end{equation}
One can also similarly check that
\begin{equation}\label{right-W}
\|{\bf W}\|_{2,\alpha,\mathcal{N}_{L,f_{\ast}}^-\cap\left\{L-\frac{1}{2}\le x_1\le L\right\}}\le C\left(\|\tilde{G}{\bf e}_{\theta}\|_{\alpha,\mathcal{N}_{L,f_\ast}^-}+\|\tilde{\mathcal{B}}{\bf e}_{\theta}\|_{1,\alpha,\Gamma_{\rm cd}^{L,f_\ast}}\right).
\end{equation}
It follows from \eqref{left-W}-\eqref{right-W} that
\begin{equation}\label{W-2-est}
\|{\bf W}\|_{2,\alpha,\mathcal{N}_{L,f_\ast}^-}\le C\left(\|\tilde{G}{\bf e}_{\theta}\|_{\alpha,\mathcal{N}_{L,f_\ast}^-}+\|\tilde{\mathcal{B}}{\bf e}_{\theta}\|_{1,\alpha,\Gamma_{\rm cd}^{L,f_\ast}}\right).
\end{equation}

Fix ${\bf x}=(x, {\bm \xi}), {\bf x}'=(x',{\bm \xi'})\in \mathcal{N}_{L,f_{\ast}}^-$ with $x,x'\in(0, L)$ and ${\bm\xi}, {\bm\xi}'\in B_1({\bf 0})(\subset \mathbb{R}^2)$.
Without loss of generality, we assume that $|{\bm \xi}'|\le |{\bm\xi}|$.
Since ${\bf e}_{\theta}$ depends only on the unit vector lying on $\partial B_1({\bf 0})\subset \mathbb{R}^2$, we have
\begin{equation*}
\begin{split}
\frac{|\tilde{G}{\bf e}_{\theta}({\bf x})-\tilde{G}{\bf e}_{\theta}({\bf x}')|}{|{\bf x}-{\bf x}'|^{\alpha}}
&\le \frac{|\tilde{G}({\bf x})-\tilde{G}({\bf x'})|}{|{\bf x}-{\bf x}'|^{\alpha}}
+\frac{|\tilde{G}({\bf x}')||{\bf e}_{\theta}(\frac{\bm \xi}{|\bm \xi|})-{\bf e}_{\theta}(\frac{\bm \xi'}{|\bm \xi'|})|}{|{\bm\xi}-{\bm\xi}'|^{\alpha}}\\
&\le \|\tilde{G}\|_{\alpha,\mathcal{N}_{L,f_\ast}^-}+
\frac{|\tilde{G}(x', {\bm\xi'})|}{|{\bm \xi'}|^{\alpha}}\frac{|{\bf e}_{\theta}(\frac{\bm \xi}{|\bm \xi|})-{\bf e}_{\theta}(\frac{\bm \xi'}{|\bm \xi'|})|}{\left|\frac{\bm \xi}{|\bm\xi|}-\frac{\bm \xi'}{|\bm \xi'|}\right|^{\alpha}}.
\end{split}
\end{equation*}
Due to the compatibility condition $\partial_r\mathcal{W}_{\ast}\equiv {\bf 0}$ on $\mathcal{N}_{L,f_{\ast}}^-\cap\{r=0\}$, we have $\tilde{G}(x',{\bf 0})=0$, and this yields that
\begin{equation*}
\frac{|\tilde{G}(x', {\bm\xi'})|}{|{\bm \xi'}|^{\alpha}}=
\frac{|\tilde{G}(x', {\bm\xi'})-\tilde{G}(x', {\bf 0})|}{|{\bm \xi'}|^{\alpha}}
\le \|\tilde{G}\|_{\alpha,\mathcal{N}_{L,f_\ast}^-}.
\end{equation*}
So we get
\begin{equation*}
\frac{|\tilde{G}{\bf e}_{\theta}({\bf x})-\tilde{G}{\bf e}_{\theta}({\bf x}')|}{|{\bf x}-{\bf x}'|^{\alpha}}
\le C\|\tilde{G}\|_{\alpha,\mathcal{N}_{L,f_\ast}^-},
\end{equation*}
from which it is obtained that
\begin{equation}\label{G-2-est}
\|\tilde{G}{\bf e}_{\theta}\|_{\alpha,\mathcal{N}_{L,f_\ast}^-}\le C\|\tilde{G}\|_{\alpha,\mathcal{N}_{L,f_\ast}^-}.
\end{equation}
Since $f_{\ast}(x)\ge \frac{3}{8}>0$, ${\bf e}_{\theta}$ is smooth on $\Gamma_{\rm cd}^{L,f_{\ast}}$, and we have
\begin{equation}\label{B-2-est}
\|\tilde{\mathcal{B}}{\bf e}_{\theta}\|_{k-1,\alpha,\Gamma_{\rm cd}^{L,f_\ast}}\le C\|\tilde{\mathcal{B}}\|_{k-1,\alpha,\Gamma_{\rm cd}^{L,f_\ast}}\quad\mbox{for }k=1,2.
\end{equation}
It follows from \eqref{psi-W-est}, \eqref{W-est-1}, and \eqref{W-2-est}-\eqref{B-2-est} that
\begin{equation*}
\|\psi\|_{k,\alpha,\Omega_{L,f_\ast}^-}\le C\left(\|\tilde{G}\|_{\alpha,\mathcal{N}_{L,f_\ast}^-}+\|\tilde{\mathcal{B}}\|_{k-1,\alpha,\Gamma_{\rm cd}^{L,f_\ast}}\right)\quad\mbox{for }k=1,2.
\end{equation*}
The claim is verified.
\end{proof}

{\bf 3.}  {\emph{(Linearized boundary value problem for $\varphi$)}}
For $\xi\in\mathbb{R}$, ${\bf s}=(s_1,s_2,s_3)$, and ${\bf v}=(v_1,v_2,v_3)\in\mathbb{R}^3$, define $\widetilde{H}$ and ${\bf A}=(A_1,A_2,A_3)$ by
\begin{equation*}
\widetilde{H}(\xi, {\bf s},{\bf v}):=H(\xi,{\bf s}+{\bf v}),\quad
A_j(\xi,{\bf s},{\bf v}):=\widetilde{H}(\xi, {\bf s},{\bf v})s_j\quad\mbox{for}\quad j=1,2,3,
\end{equation*}
where  $H$ is defined by \eqref{def-H-G}.
Then the equation
$$\mbox{div}\left(H(S,{\bf q}(r,\psi,D\psi,D\varphi,\Lambda)){\bf q}(r,\psi,D\psi,D\varphi,\Lambda)\right)=0$$
can be rewritten as
\begin{equation}\label{re-conti}
\mbox{div}\left({\bf A}(S,D\varphi,{\bf t}(r,\psi,D\psi,\Lambda))\right)=-\mbox{div}\left(\widetilde{H}(S,D\varphi,{\bf t}(r,\psi,D\psi,\Lambda)){\bf t}(r,\psi,D\psi,\Lambda)\right).
\end{equation}
For $\varphi_0$ given by \eqref{def-varphi0-B}, denote ${\bf V}_0:=(S_0^-,D\varphi_0,{\bf 0})$ and set
\begin{equation}\label{aij-def}
a_{ij}:=\partial_{s_j}A_i({\bf V}_0)\quad\mbox{for}\quad i,j=1,2,3.
\end{equation}
Then the constant matrix $[a_{ij}]_{i,j=1}^3$ is strictly positive and diagonal, and there exists a constant $\nu\in(0,1/10]$ satisfying
\begin{equation}\label{3D-nu-aij}
\nu< a_{ii}<\frac{1}{\nu}\quad\mbox{for all}\quad i=1,2,3.
\end{equation}

Set $\phi:=\varphi-\varphi_0$.
Then \eqref{re-conti} can be rewritten as
\begin{equation*}
\mathcal{L}(\phi)=\mbox{div}{\bf F}(S-S_0^-,D\phi,{\bf t}(r,\psi,D\psi,\Lambda)),
\end{equation*}
where $\mathcal{L}$ and ${\bf F}=(F_1,F_2,F_3)$ are defined as follows:
\begin{equation}\label{def-F}
\left.\begin{split}
\mathcal{L}(\phi):=&\sum_{i=1}^3 a_{ii}\partial_{ii}\phi,\\
F_i(Q):=&-\widetilde{H}({\bf V}_0+Q)v_i-\int_0^1 D_{\xi,\bf v}A_i({\bf V}_0+tQ)dt\cdot(\xi,{\bf v})\\
&-{\bf s}\cdot\int_0^1D_{\bf s}A_i({\bf V}_0+tQ)-D_{\bf s}A_i({\bf V}_0)dt,
\end{split}
\right.
\end{equation}
with  $Q=(\xi, {\bf s},{\bf v})\in\mathbb{R}\times(\mathbb{R}^3)^2$.
Here, $\partial_{x_i}$ is abbreviated as $\partial_i$.

By the boundary conditions for $\varphi$ given in \eqref{Fixed-BVP} and the definition of $\varphi_0$, the boundary conditions for $\phi$ on $\partial\mathcal{N}_{L,f_{\ast}}^-\backslash\Gamma_{\rm cd}^{L,f_{\ast}}$ become
\begin{equation*}
\phi=\varphi_{\rm en }\quad\mbox{on}\quad \Gamma_{\rm en}^-\quad\mbox{and}\quad
\phi=0\quad\mbox{on}\quad \Gamma_{\rm ex}^{L,f_{\ast}}.
\end{equation*}
On $\Gamma_{\rm cd}^{L,f_{\ast}}$, the boundary condition for $\varphi$ given in \eqref{Fixed-BVP} implies that $\phi$ should be a constant along $\Gamma_{\rm cd}^{L,f_{\ast}}$.
Since we seek a solution $\phi$ to be continuous up to the boundary, and since $\varphi_{\rm en}(0,\frac{1}{2})=0$ by the definition \eqref{def-varphi-en}, we prescribe the boundary condition for $\phi$ on $\Gamma_{\rm cd}^{L,f_{\ast}}$ as
\begin{equation*}
\phi=0\quad\mbox{on}\quad \Gamma_{\rm cd}^{L,f_\ast}.
\end{equation*}

For a fixed $(\tilde{\phi},\tilde{\psi})\in\mathcal{K}^{f_{\ast}}(M_3,M_4)$, let  $\psi\in C^{2,\alpha}(\overline{\Omega_{L,f_\ast}^-})$ be the unique solution to the linear boundary value problem \eqref{lin-psi} associated with $(\tilde{\phi},\tilde{\psi})\in\mathcal{K}^{f_{\ast}}(M_3,M_4)$. For such $\psi$, we set
\begin{equation}\label{FF-def}
\mathfrak{F}:={\bf F}(S_\ast-S_0^-,D{\tilde{\phi}},{\bf t}(r,{\psi},D{\psi},\Lambda_{\ast})),
\end{equation}
where ${\bf F}$ is given by \eqref{def-F}. And, we consider the following linear boundary value problem
\begin{equation}\label{lin-phi}
\left\{\begin{split}
\mathcal{L}(\phi)=\mbox{div}\mathfrak{F}\quad&\mbox{in}\quad \mathcal{N}_{L,f_\ast}^-,\\
\phi=\varphi_{\rm en}\quad&\mbox{on}\quad\Gamma_{\rm en}^-,\\
\phi=0\quad&\mbox{on}\quad \Gamma_{\rm cd}^{L,f_\ast}\cup\Gamma^{L,f_{\ast}}_{\rm ex}.
\end{split}\right.
\end{equation}
In the next step, we prove the well-posedness of \eqref{lin-phi}.

{\bf 4.} {\emph{(The well-posedness of \eqref{lin-phi})}}
{\bf Claim:} For each $(\tilde{\phi},\tilde{\psi})\in\mathcal{K}^{f_{\ast}}(M_3,M_4)$, the linear boundary value problem
\eqref{lin-phi} associated with $(\tilde{\phi},\tilde{\psi})$  has a unique solution $\phi\in C^{2,\alpha}(\overline{\mathcal{N}_{L,f_\ast}^-})$, and the solution satisfies
\begin{equation}\label{3D-phi-est}
\|\phi\|_{k,\alpha,\mathcal{N}_{L,f_\ast}^-}
\le C\left(\|{\mathfrak F}\|_{k-1,\alpha,\mathcal{N}_{L,f_\ast}^-}+\|{\varphi_{\rm en}}\|_{k,\alpha,\Gamma_{\rm en}^-}\right)\quad\mbox{for }k=1,2.
\end{equation}
Moreover, the solution $\phi$ is axially symmetric, and it satisfies
$$\partial_{xx}\phi\equiv 0\quad\mbox{on}\quad
(\Gamma_{\rm en}^-\cap\{r\ge\frac{1}{2}-\epsilon\})\cup\Gamma^{L,f_{\ast}}_{\rm ex}.$$

\begin{proof}[Proof of Claim.]
For $\varphi_{\rm en}$ given by \eqref{def-varphi-en}, define a function ${\varphi_{\rm en}^{\ast}}$ by
\begin{equation}\label{def-var-ast}
{\varphi_{\rm en}^{\ast}}({\bf x}):=\eta(x_1)\varphi_{\rm en}\left(\frac{\sqrt{x_2^2+x_3^2}}{f_{\ast}(x_1)}\right)\quad\mbox{for}\quad{\bf x}=(x_1,x_2,x_3)\in\mathcal{N}_{L,f_\ast}^-,
\end{equation}
where $\eta$ is a $C^{\infty}$-function satisfying
\begin{equation}\label{eta-def}
\eta=1\quad\mbox{for } x_1<\frac{L}{10},\quad \eta=0\quad\mbox{for }x_1>\frac{9L}{10},\quad |\eta'(x_1)|\le2,\quad|\eta''(x_1)|\le 2.
\end{equation}
Set $\phi_{\rm hom}:=\phi-\varphi_{\rm en}^{\ast}$. Then the linear boundary value problem \eqref{lin-phi} can be rewritten as
\begin{equation}\label{3D-hom-eq}
\left\{\begin{split}
\mathcal{L}(\phi_{\rm hom})
=\mathfrak{F}^{\ast}\quad&\mbox{in}\quad\mathcal{N}_{L,f_\ast}^-,\\
\phi_{\rm hom}=0\quad&\mbox{on}\quad\partial\mathcal{N}_{L,f_\ast}^-,
\end{split}\right.
\end{equation}
for $\mathfrak{F}^{\ast}$ defined by
\begin{equation}\label{def-F-ast}
\mathfrak{F}^{\ast}:=\mbox{div}\mathfrak{F}-\sum_{i=1}^3a_{ii}\partial_{ii}\varphi_{\rm en}^{\ast},
\end{equation}
where $a_{ii}$ $(i=1,2,3)$ are given by \eqref{aij-def}.
By the standard elliptic theory, the linear boundary value problem \eqref{3D-hom-eq} has a unique solution $\phi_{\rm hom}\in C^{1,\alpha}(\overline{\mathcal{N}_{L,f_\ast}^-})\cap C^{2,\alpha}(\mathcal{N}_{L,f_{\ast}}^-)$.

To obtain a uniform $C^0$-estimate of $\phi_{\rm hom}$ for all $L$, we define a function $\mathfrak{M}$ by
$$\mathfrak{M}({\bf x}):=-\frac{3}{2}\left(\frac{\|\mathfrak{F}^{\ast}\|_{\alpha,\mathcal{N}_{L,f_\ast}^-}}{a_{22}}\right)x_2^2+\frac{2\|\mathfrak{F}^{\ast}\|_{\alpha,\mathcal{N}_{L,f_\ast}^-}}{a_{22}}.$$
Since $a_{22}>\nu>0$ in $\mathcal{N}_{L,f_{\ast}}^-$ by \eqref{3D-nu-aij}, $\mathfrak{M}$ is well-defined.
A direct computation yields
\begin{equation*}
\left\{\begin{split}
\mathcal{L}(\mathfrak{M}\pm\phi_{\rm hom})=-3\|\mathfrak{F}^{\ast}\|_{\alpha,\mathcal{N}_{L,f_\ast}^-}\pm \mathcal{L}(\phi_{\rm hom})\le0\quad&\mbox{in}\quad\mathcal{N}_{L,f_\ast}^-,\\
\mathfrak{M}\pm\phi_{\rm hom}=\mathfrak{M}\ge0\quad&\mbox{on}\quad\partial\mathcal{N}_{L,f_\ast}^-.\\
\end{split}\right.
\end{equation*}
Since $\mathcal{L}$ is uniformly elliptic, the comparison principle implies
$-\mathfrak{M}\le\phi_{\rm hom}\le\mathfrak{M}$ in $\mathcal{N}_{L,f_{\ast}}^-,$ from which it follows that
$$\|\phi_{\rm hom}\|_{0,\mathcal{N}_{L,f_\ast}^-}\le C\|\mathfrak{F}^{\ast}\|_{\alpha,\mathcal{N}_{L,f_\ast}^-}.$$
Then we obtain the estimate
$$\|\phi_{\rm hom}\|_{1,\alpha,\mathcal{N}_{L,f_\ast}^-}\le C\|\mathfrak{F}^{\ast}\|_{\alpha,\mathcal{N}_{L,f_\ast}^-}.$$
To obtain $C^{2,\alpha}$-estimate of $\phi_{\rm hom}$ up to the boundary, we use the method of reflection.
By the compatibility conditions of $(S_\ast,\Lambda_\ast,\tilde{\phi})$ given in \eqref{Ent-Ang-set} and \eqref{ell-set}, and $\partial_x\psi\equiv0$ on $(\Gamma_{\rm en}^-\cap\{r\ge\frac{1}{2}-\epsilon\})\cup\Gamma^{L,f_{\ast}}_{\rm ex}$ given from \eqref{3D-psi-BC}, we have
\begin{equation}\label{F-0-ex}
\mbox{div}\mathfrak{F}=\mbox{div}{\bf F}(S_{\ast}-S_0^-,D\tilde{\phi},{\bf t}(r,{\psi},D{\psi},\Lambda_{\ast}))\equiv0\quad\mbox{on }(\Gamma_{\rm en}^-\cap\{r\ge\frac{1}{2}-\epsilon\})\cup\Gamma^{L,f_{\ast}}_{\rm ex}.
\end{equation}
From the definition of $\varphi_{\rm en}^{\ast}$ given in \eqref{def-var-ast}, the compatibility conditions of $f_{\ast}$ given in \eqref{F-set}, and the definition of  $\eta$ given in \eqref{eta-def}, it can be directly checked that
\begin{equation}\label{partial-varphi-en}
\partial_{ii}\varphi_{\rm en}^{\ast}\equiv 0\quad\mbox{on}\quad  (\Gamma_{\rm en}^-\cap\{r\ge\frac{1}{2}-\epsilon\})\cup\Gamma^{L,f_{\ast}}_{\rm ex},\quad i=1,2,3.
\end{equation}
It follows from  \eqref{F-0-ex}-\eqref{partial-varphi-en}  and  the definition of $\mathfrak{F}^{\ast}$ given in \eqref{def-F-ast} that
\begin{equation*}
\mathfrak{F}^{\ast}\equiv 0\quad\mbox{on}\quad  (\Gamma_{\rm en}^-\cap\{r\ge\frac{1}{2}-\epsilon\})\cup\Gamma^{L,f_{\ast}}_{\rm ex}.
\end{equation*}
Then we can apply the method of reflection to obtain the estimate
$$\|\phi_{\rm hom}\|_{2,\alpha,\mathcal{N}_{L,f_\ast}^-}\le C\|\mathfrak{F}^{\ast}\|_{\alpha,\mathcal{N}_{L,f_\ast}^-},$$
and this implies that the linear boundary value problem \eqref{lin-phi} has a unique solution $\phi=\phi_{\rm hom}+\varphi_{\rm en}^{\ast}\in C^{2,\alpha}(\overline{\mathcal{N}_{L,f_{\ast}}^-})$ that satisfies
\begin{equation*}
\|\phi\|_{k,\alpha,\mathcal{N}_{L,f_{\ast}}^-}\le C\left(\|\mathfrak{F}\|_{k-1,\alpha,\mathcal{N}_{L,f_\ast}^-}+\|\varphi_{\rm en}\|_{k,\alpha,\Gamma_{\rm en}^-}\right)\quad\mbox{for }k=1,2.
\end{equation*}

For any $\theta\in[0,2\pi)$, define a function $\phi_{\rm hom}^{\theta}$ by
$$\phi_{\rm hom}^{\theta}({\bf x}):=\phi_{\rm hom}(x_1,x_2\cos\theta-x_3\sin\theta,x_2\sin\theta+x_3\cos\theta).$$
Then, we have $\phi_{\rm hom}^{\theta}=\phi_{\rm hom}$ on $\partial \mathcal{N}_{L,f_\ast}^-$.
By using \eqref{aij-def}, it can be directly checked that $a_{22}= a_{33}$. Therefore, $\mathcal{L}(\phi_{\rm hom}^{\theta})=\mathcal{L}(\phi_{\rm hom})$ holds in $\mathcal{N}_{L,f_\ast}^-$. This implies that $\phi_{\rm hom}^{\theta}$ is a solution to \eqref{3D-hom-eq}. By the uniqueness of a solution to \eqref{3D-hom-eq}, we conclude that
$\phi_{\rm hom}=\phi_{\rm hom}^{\theta}.$
Therefore $\phi_{\rm hom}$ is axially symmetric, and this implies that $\phi$ is axially symmetric.

Since $\phi_{\rm hom}\equiv 0$ and $\sum_{i=1}^3 a_{ii}\partial_{ii}\varphi_{\rm en}^{\ast}\equiv 0$ on $(\Gamma_{\rm en}^-\cap\{r\ge\frac{1}{2}-\epsilon\})\cup\Gamma^{L,f_{\ast}}_{\rm ex}$, we have
\begin{equation}\label{phi-0-ex}
\partial_{ii}\phi\equiv0 \quad\mbox{on}\quad(\Gamma_{\rm en}^-\cap\{r\ge\frac{1}{2}-\epsilon\})\cup\Gamma^{L,f_{\ast}}_{\rm ex}\quad\mbox{for}\quad i=2,3.
\end{equation}
It follows from \eqref{F-0-ex} and \eqref{phi-0-ex} that
$\mathcal{L}(\phi)=a_{11}\partial_{xx}\phi\equiv 0$ on $(\Gamma_{\rm en}^-\cap\{r\ge\frac{1}{2}-\epsilon\})\cup\Gamma^{L,f_{\ast}}_{\rm ex}.$
Since $a_{11}>0$, we conclude that
$\partial_{xx}\phi\equiv 0$ on $(\Gamma_{\rm en}^-\cap\{r\ge\frac{1}{2}-\epsilon\})\cup\Gamma^{L,f_{\ast}}_{\rm ex}.$
The proof of claim is completed.
\end{proof}

{\bf 5.}  {\emph{(The well-posedness of nonlinear boundary value problem  \eqref{Fixed-BVP})}}
For fixed $(\mathcal{W}_{\ast},f_{\ast})\in\mathcal{P}(M_1)\times\mathcal{F}(M_2)$,
define an iteration mapping $\mathcal{I}^{f_{\ast},\mathcal{W}_{\ast}}:\mathcal{K}^{f_\ast}(M_3,M_4)\rightarrow C^{2,\alpha}(\overline{\mathcal{N}_{L,f_\ast}^-})\times C^{2,\alpha}(\overline{\Omega_{L,f_{\ast}}^-})$ by
\begin{equation*}
\mathcal{I}^{f_\ast,\mathcal{W}_{\ast}}(\tilde{\phi},\tilde{\psi})=(\phi,\psi),
\end{equation*}
where $(\phi,\psi)$ is the solution to \eqref{lin-psi} and \eqref{lin-phi} associated with $(\tilde{\phi},\tilde{\psi})$.

By straightforward computations, one can easily check that  there exists a constant $\epsilon_1\in(0,\frac{1}{8})$ depending only on the data so that if
\begin{equation}\label{epsilon1}
M_1\sigma+M_2\sigma+M_3\sigma+M_4\sigma\le \epsilon_1,
\end{equation}
then we have
\begin{equation}\label{est-lem}
\left.
\begin{split}
&\|{\mathfrak F}\|_{1,\alpha,\mathcal{N}_{L,f_\ast}^-}\le C\left(M_1\sigma+(M_3\sigma)^2+M_4\sigma\right),\\
&\|\tilde{\mathcal{B}}\|_{1,\alpha,\Gamma_{\rm cd}^{L,f_\ast}}\le C\left(M_1\sigma+(M_2\sigma)^2\right),\\
&\|\tilde{G}\|_{\alpha,\mathcal{N}_{L,f_\ast}^-}\le CM_1\sigma,\\
\end{split}\right.
\end{equation}
where ${\mathfrak F}$, $\tilde{\mathcal{B}}$, and $\tilde{G}$ are given by  \eqref{FF-def}, and \eqref{def-GB}.
It follows from \eqref{psi-est-2}, \eqref{3D-phi-est}, and \eqref{est-lem} that
\begin{equation}\label{pp-est-M3}
\left.\begin{split}
&\|\phi\|_{2,\alpha,\mathcal{N}_{L,f_\ast}^-}\le C_1^{\flat}\left(M_1\sigma+(M_3\sigma)^2+M_4\sigma+\sigma\right),\\
&\|\psi\|_{2,\alpha,\Omega_{L,f_{\ast}}^-}\le C_1^{\flat}\left(M_1\sigma+(M_2\sigma)^2\right),
\end{split}\right.
\end{equation}
for a constant $C_1^{\flat}>0$ depending  on the data and $\alpha$ but independent of $L$.
We choose $M_3$, $M_4$, and $\sigma_6^{\ast}$ as
\begin{equation}\label{3D-sigma8}
\begin{split}
&M_3=4C_1^{\flat}(1+M_1+M_4),\quad M_4=2C_1^{\flat}M_1,\\
&\mbox{and}\quad \sigma_6^{\ast}=\min\left\{\frac{\epsilon_1}{M_1+M_2+M_3+M_4},\frac{M_4}{2C_1^{\flat}M_2^2},\frac{1}{4C_1^{\flat}M_3}\right\},
\end{split}
\end{equation}
where $\epsilon_1$ is given in \eqref{epsilon1},
so that \eqref{pp-est-M3} implies that
$(\phi,\psi)\in\mathcal{K}^{f_\ast}(M_3,M_4)$ for $\sigma\le\sigma_6^{\ast}.$
Under such choices of $(M_3, M_4, \sigma_6^{\ast})$, the iteration mapping
$\mathcal{I}^{f_{\ast},\mathcal{W}_{\ast}}$ maps $\mathcal{K}^{f_\ast}(M_3,M_4)$ into itself if $\sigma\le\sigma_6^{\ast}$.
Furthermore, $(\phi,\psi)$ satisfies the estimate
\begin{equation*}
\|\phi\|_{2,\alpha,\mathcal{N}_{L,f_{\ast}}^-}+\|\psi\|_{2,\alpha,\Omega_{L,f_{\ast}}^-}\le (M_3+M_4)\sigma\le C(M_1+1)\sigma.
\end{equation*}

Now we show that $\mathcal{I}^{f_\ast,\mathcal{W}_{\ast}}$ is a contraction mapping if $\sigma$ is a small constant depending only on the data and $(\alpha, M_1,M_2)$.

For each $j=1,2$, let
\begin{equation*}\left\{
\begin{split}
&(\phi^{(j)},\psi^{(j)}):=\mathcal{I}^{f_\ast,\mathcal{W}_\ast}(\tilde{\phi}^{(j)},\tilde{\psi}^{(j)})\quad\mbox{for }(\tilde{\phi}^{(j)},\tilde{\psi}^{(j)})\in\mathcal{K}^{f_\ast}(M_3,M_4),\\
&{\bf F}_{\ast}:={\bf F}(S_{\ast}-S_0^-,D\tilde{\phi}^{(1)},{\bf t}(r,\psi^{(1)},D\psi^{(1)},\Lambda_{\ast}))\\
&\qquad -{\bf F}(S_{\ast}-S_0^-,D\tilde{\phi}^{(2)},{\bf t}(r,\psi^{(2)},D\psi^{(2)},\Lambda_{\ast})),\\
&G_{\ast}:=G(\mathcal{W}_{\ast},\partial_r\mathcal{W}_{\ast},{\bf t}(r,\tilde{\psi}^{(1)},D\tilde{\psi}^{(1)},\Lambda_{\ast}),D\tilde{\phi}^{(1)}+D\varphi_0)\\
&\qquad -G(\mathcal{W}_{\ast},\partial_r\mathcal{W}_{\ast},{\bf t}(r,\tilde{\psi}^{(2)},D\tilde{\psi}^{(2)},\Lambda_{\ast}),D\tilde{\phi}^{(2)}+D\varphi_0),
\end{split}\right.
\end{equation*}
where ${\bf F}$ and $G$ are given by \eqref{def-F} and \eqref{def-H-G}, respectively.
By a direct computation, it can be checked that there exists a constant $\epsilon_2\in(0,\epsilon_1]$ depending only on the data so that if $$M_1\sigma+M_2\sigma+M_3\sigma+M_4\sigma\le\epsilon_2,$$ then we have
\begin{equation}\label{diff-est}
\left.\begin{split}
&\|{\bf F}_{\ast}\|_{1,\alpha,\mathcal{N}_{L,f_\ast}^-}\le C\|\psi^{(1)}-\psi^{(2)}\|_{2,\alpha,\Omega_{L,f_{\ast}}^-}+C(M_1+1)\sigma\|\tilde{\phi}^{(1)}-\tilde{\phi}^{(2)}\|_{2,\alpha,\mathcal{N}_{L,f_\ast}^-},\\
&\|G_{\ast}\|_{\alpha,\mathcal{N}_{L,f_\ast}^-}\le CM_1\sigma\left(\|\tilde{\psi}^{(1)}-\tilde{\psi}^{(2)}\|_{2,\alpha,\Omega_{L,f_{\ast}}^-}+\|\tilde{\phi}^{(1)}-\tilde{\phi}^{(2)}\|_{2,\alpha,\mathcal{N}_{L,f_\ast}^-}\right).
\end{split}\right.
\end{equation}
Then it follows from \eqref{psi-est-2}, \eqref{3D-phi-est}, and \eqref{diff-est} that
\begin{equation*}
\begin{split}
\|\phi^{(1)}&-\phi^{(2)}\|_{2,\alpha,\mathcal{N}_{L,f_\ast}^-}+\|\psi^{(1)}-\psi^{(2)}\|_{2,\alpha,\Omega_{L,f_{\ast}}^-}\\
&\le C_2^{\flat}(M_1+1)\sigma\left(\|\tilde{\psi}^{(1)}-\tilde{\psi}^{(2)}\|_{2,\alpha,\Omega_{L,f_\ast}^-}+\|\tilde{\phi}^{(1)}-\tilde{\phi}^{(2)}\|_{2,\alpha,\mathcal{N}_{L,f_\ast}^-}\right)
\end{split}
\end{equation*}
for a constant $C_2^{\flat}>0$ depending only on the data and $\alpha$ but independent of $L$.
Choose $\sigma_6$ as
\begin{equation}\label{Sigma6}
\sigma_6=\min\left\{\sigma_6^{\ast},\frac{1}{2C_2^{\flat}(M_1+1)},\frac{\epsilon_2}{M_1+M_2+M_3+M_4}\right\}
\end{equation}
with $\sigma_6^{\ast}$ defined in \eqref{3D-sigma8}. Thus if $\sigma\le \sigma_6$, then the mapping $\mathcal{I}^{f_{\ast},\mathcal{W}_\ast}$ is a contraction mapping so that $\mathcal{I}^{f_{\ast},\mathcal{W}_{\ast}}$ has a unique fixed point in $\mathcal{K}^{f_\ast}(M_3,M_4)$.
This gives the unique existence of a solution to \eqref{Fixed-BVP}.
The proof of Lemma \ref{Pro-fix-S} is completed.
\end{proof}

Next, we prove the unique solvability of Problem \ref{Prob4-Fix-S}, which is a free boundary problem.
\begin{proof}[Proof of Lemma \ref{Lem-S-free}.]
{\bf 1.}
Now we choose $M_2$ from \eqref{F-set}, and adjust $\sigma$ to find a solution of Problem \ref{Prob4-Fix-S} by the method of iteration.

Given $f_{\ast}\in\mathcal{F}(M_2)$ and $(S_{\ast},\Lambda_{\ast})\in\mathcal{P}(M_1)$, let $(\varphi,\psi)\in C^{2,\alpha}(\overline{\mathcal{N}_{L,f_{\ast}}^-})\times C^{2,\alpha}(\overline{\Omega_{L,f_\ast}^-})$ be the unique solution to the boundary value problem \eqref{Fixed-BVP}. Note that $(\varphi,\psi)$ satisfies the estimate \eqref{fix-est} given in Lemma \ref{Pro-fix-S}.
For simplicity, we set
\begin{equation*}
\begin{split}
&\rho^{\ast}:=H(S_{\ast},{\bf q}(r,\psi,D\psi,D\varphi,\Lambda_{\ast})),\\
&{\bf u}^{\ast}:=\left(\partial_x\varphi+\frac{1}{r}\partial_r(r\psi)\right){\bf e}_x+\left(\partial_r\varphi-\partial_x\psi\right){\bf e}_r,
\end{split}
\end{equation*}
where $H$ is given in \eqref{def-H-G}.
From the first equation in \eqref{Fixed-BVP}, we have
\begin{equation}\label{conti-u}
\frac{\partial_x(r\rho^{\ast}{\bf u}^{\ast}\cdot{\bf e}_x)+\partial_r(r\rho^{\ast}{\bf u}^{\ast}\cdot{\bf e}_r)}{r}=0.
\end{equation}
As in the proof of Lemma \ref{Pro-fix-S}, there exists a constant $\epsilon_3\in(0,1]$ depending only on the data and $\alpha$ so that if
\begin{equation*}\label{epsilon3}
M_1\sigma+M_2\sigma+\sigma\le\epsilon_3,
\end{equation*}
then we have
\begin{equation}\label{rhou-est-1}
\|\rho^{\ast}{\bf u}^{\ast}-\rho_0^-u_0{\bf e}_x\|_{1,\alpha,\mathcal{N}_{L,f_\ast}^-}\le C_{\star}(M_1+1)\sigma,
\end{equation}
where the constant $C_{\star}>0$ depends only on the data and $\alpha$ but independent of $L$.
If $\sigma\in(0,\sigma_6]$ satisfies
\begin{equation*}
\sigma\le \frac{\rho_0^-u_0}{2C_{\star}(M_1+1)},
\end{equation*}
then we obtain from \eqref{rhou-est-1} that
\begin{equation}\label{rhou-est-2}
\|\rho^{\ast}{\bf u}^{\ast}-\rho_0^-u_0{\bf e}_x\|_{1,\alpha,\mathcal{N}_{L,f_\ast}^-}\le \frac{\rho_0^-u_0}{2}.
\end{equation}

For each $x\in[0,L]$, we choose $f(x)\in\mathbb{R}^+$ to satisfy
\begin{equation}\label{3D-f-est1}
\int_{f_\ast(x)}^{f(x)} t \rho_0^-u_0 dt=\int_0^{1/2}  t\rho^{\ast}{\bf u}^{\ast}\cdot{\bf e}_x(0,t)dt-\int_0^{f_{\ast}(x)} t\rho^{\ast}{\bf u}^{\ast}\cdot{\bf e}_x(x,t)dt.
\end{equation}
If $f\equiv f_{\ast}$, then \eqref{3D-f-est1} yields that
\begin{equation}\label{fix-f-eq}
\int_0^{1/2}  t\rho^{\ast}{\bf u}^{\ast}\cdot{\bf e}_x(0,t)dt=\int_0^{f(x)} t\rho^{\ast}{\bf u}^{\ast}\cdot{\bf e}_x(x,t)dt.
\end{equation}
Differentiating \eqref{fix-f-eq} with respect to $x$, and using the equation \eqref{conti-u},
we have
\begin{equation*}
f'(x)=\frac{{\bf u}^{\ast}\cdot{\bf e}_r}{{\bf u}^{\ast}\cdot{\bf e}_x}(x,f(x))=\frac{\partial_r\varphi-\partial_x\psi}{\partial_x\varphi+\frac{1}{r}\partial_r(r\psi)}(x,f(x)).
\end{equation*}
Also, we have $f(0)=\frac{1}{2}$.
Thus $f$ satisfies the free boundary condition \eqref{g-free-cut} for $0<x<L$.

Since $\rho_0^-u_0>0$, \eqref{3D-f-est1} is equivalent to
\begin{equation}\label{3D-f-est2}
f^2(x)=f_\ast^2(x)+\frac{2}{\rho_0^-u_0}\int_0^{1/2}  t\rho^{\ast}{\bf u}^{\ast}\cdot{\bf e}_x(0,t)dt-\frac{2}{\rho_0^-u_0}\int_0^{f_{\ast}(x)} t\rho^{\ast}{\bf u}^{\ast}\cdot{\bf e}_x(x,t)dt.
\end{equation}
By \eqref{rhou-est-1} and \eqref{rhou-est-2},
\begin{equation}
\label{3D-f-est4}
\mbox{RHS of \eqref{3D-f-est2}}\ge
\frac{1}{16}>0\quad\mbox{if}\quad\sigma\le\min\left\{\frac{\epsilon_3}{M_1+M_2+1}, \frac{\rho_0^-u_0}{16 C_{\star}(M_1+1)}\right\}=:\sigma_5'.
\end{equation}
Then the function $f:[0,L]\rightarrow\mathbb{R}^+$ given by
\begin{equation}\label{3D-f-est3}
f(x):=\sqrt{f_\ast^2(x)+\frac{2}{\rho_0^-u_0}\int_0^{1/2}  t\rho^{\ast}{\bf u}^{\ast}\cdot{\bf e}_x(0,t)dt-\frac{2}{\rho_0^-u_0}\int_0^{f_{\ast}(x)} t\rho^{\ast}{\bf u}^{\ast}\cdot{\bf e}_x(x,t)dt}
\end{equation}
 is well-defined, and satisfies \eqref{3D-f-est1}.
And, $f(0)=\frac{1}{2}$, $f'(0)=f'(L)=0$.
Moreover, by a direct computation, we have the estimate
 \begin{equation}\label{3D-f-est7}
 \|f-\frac{1}{2}\|_{2,\alpha,(0,L)}\le C_{\star\star}(M_1+1)\sigma
 \end{equation}
 for a constant $C_{\star\star}>0$ depending only on the data and $\alpha$ but independent of $L$.

We define an iteration mapping $\mathcal{I}^{\mathcal{W}_{\ast}}:\mathcal{F}(M_2)\rightarrow C^{2,\alpha}([0,L])$ by
$$\mathcal{I}^{\mathcal{W}_{\ast}}(f_{\ast})=f$$
for $f$ given by \eqref{3D-f-est3}.
Choose $M_2$ and $\sigma_5^{\ast}$ as
\begin{equation}\label{Sigma5star}
M_2=C_{\star\star}(M_1+1)\quad\mbox{and}\quad\sigma_5^{\ast}=\min\left\{\sigma_6,\sigma_5'\right\}
\end{equation}
for $\sigma_6$ and $\sigma_5'$ defined by \eqref{Sigma6} and \eqref{3D-f-est4}, respectively.
Under such choices of $(M_2,\sigma_5^{\ast})$, the iteration mapping $\mathcal{I}^{\mathcal{W}_{\ast}}$ maps $\mathcal{F}(M_2)$ into itself if $\sigma\le\sigma_5^{\ast}$.

{\bf 2.}
The iteration set $\mathcal{F}(M_2)$ given by \eqref{F-set} is a convex and compact subset of $C^{2,\alpha/2}([0,L])$.
For each fixed $\mathcal{W}_{\ast}\in \mathcal{P}(M_1)$, the iteration map $\mathcal{I}^{\mathcal{W}_{\ast}}$ maps $\mathcal{F}(M_2)$ into itself where $M_2$ is chosen by \eqref{Sigma5star}, and $\sigma\le \sigma_5^{\ast}$ for $\sigma_5^{\ast}$ from \eqref{Sigma5star}.

Suppose that a sequence $\{f_{\ast}^{(k)}\}_{k=1}^{\infty}\subset\mathcal{F}(M_2)$ converges in $C^{2,\alpha/2}([0,L])$ to $f_{\ast}^{(\infty)}\in\mathcal{F}(M_2)$.
For each $k\in\mathbb{N}\cup\{\infty\}$, set
\begin{equation*}
f^{(k)}:=\mathcal{I}^{\mathcal{W}_{\ast}}(f_{\ast}^{(k)}).
\end{equation*}
And, let $\mathcal{U}^{(k)}:=(\varphi^{(k)},\psi^{(k)})\in C^{2,\alpha}(\overline{\mathcal{N}_{L,f_{\ast}^{(k)}}^-})\times C^{2,\alpha}(\overline{\Omega_{L,f_{\ast}^{(k)}}^-})$ be the unique solution of \eqref{Fixed-BVP} associated with $f_{\ast}=f_{\ast}^{(k)}$.
Define a transformation $T^{(k)}:\overline{\mathcal{N}_{L,f_{\ast}^{(\infty)}}^-}\rightarrow \overline{\mathcal{N}_{L,f_{\ast}^{(k)}}^-}$ by
\begin{equation*}
T^{(k)}(x_1,x_2,x_3)=\left(x_1,\sqrt{\frac{f_{\ast}^{(k)}(x_1)}{f_{\ast}^{(\infty)}(x_1)}}x_2,\sqrt{\frac{f_{\ast}^{(k)}(x_1)}{f_{\ast}^{(\infty)}(x_1)}}x_3\right).\end{equation*}
Then $\{\mathcal{U}^{(k)}\circ T^{(k)}\}_{k=1}^{\infty}$ is sequentially compact in $C^{2,\alpha/2}(\overline{\mathcal{N}_{L,f_\ast^{(\infty)}}^-})\times C^{2,\alpha/2}(\overline{\Omega_{L,f_\ast^{(\infty)}}^-})$ and the limit of each convergent subsequence of $\{\mathcal{U}^{(k)}\circ T^{(k)}\}_{k=1}^{\infty}$ in $C^{2,\alpha/2}(\overline{\mathcal{N}_{L,f_\ast^{(\infty)}}^-})\times C^{2,\alpha/2}(\overline{\Omega_{L,f_\ast^{(\infty)}}^-})$ solves \eqref{Fixed-BVP} associated with $f_{\ast}=f_{\ast}^{(\infty)}$.
By the uniqueness of a solution for the problem \eqref{Fixed-BVP}, $\{\mathcal{U}^{(k)}\circ T^{(k)}\}_{k=1}^{\infty}$ is convergent in $C^{2,\alpha/2}(\overline{\mathcal{N}_{L,f_\ast^{(\infty)}}^-})\times C^{2,\alpha/2}(\overline{\Omega_{L,f_\ast^{(\infty)}}^-})$.
It follows from \eqref{3D-f-est3}-\eqref{3D-f-est7} that $f^{(k)}$ converges to $f^{(\infty)}$ in $C^{2,\alpha/2}([0,L])$.
This implies that $\mathcal{I}^{\mathcal{W}_{\ast}}(f_{\ast}^{(k)})$ converges to $\mathcal{I}^{\mathcal{W}_{\ast}}(f_{\ast}^{(\infty)})$ in $C^{2,\alpha/2}([0,L])$.
Thus $\mathcal{I}^{\mathcal{W}_{\ast}}$ is a continuous map in $C^{2,\alpha/2}([0,L])$.
Applying the Schauder fixed point theorem yields that $\mathcal{I}^{\mathcal{W}_{\ast}}$ has a fixed point $f\in\mathcal{F}(M_2)$.
For such $f$, let $(\varphi,\psi)\in C^{2,\alpha}(\overline{\mathcal{N}_{L,f}^-})\times C^{2,\alpha}(\overline{\Omega_{L,f}^-})$ be the unique solution to the fixed boundary problem \eqref{Fixed-BVP} associated with $f_{\ast}=f$.
Then $(f,\varphi,\psi)$ is a solution to Problem \ref{Prob4-Fix-S}.
It follows from \eqref{fix-est} and \eqref{3D-f-est7} that
\begin{equation*}
\|f-\frac{1}{2}\|_{2,\alpha,(0,L)}+\|\varphi-\varphi_0\|_{2,\alpha,\mathcal{N}_{L,f}^-}+\|\psi{\bf e}_{\theta}\|_{2,\alpha,\mathcal{N}_{L,f}^-}\le C\left(M_1+1\right)\sigma.
\end{equation*}

{\bf 3.} Finally, it remains to prove the uniqueness of a solution to Problem \ref{Prob4-Fix-S}.
For a fixed $\mathcal{W}_{\ast}\in\mathcal{P}(M_1)$, let $(f^{(1)},\varphi^{(1)},\psi^{(1)})$ and $(f^{(2)},\varphi^{(2)},\psi^{(2)})$ be two solutions to Problem \ref{Prob4-Fix-S}, and suppose that each solution satisfies the estimate given in \eqref{3D-pps-est} of Lemma \ref{Lem-S-free}.
Define a  transformation $\mathfrak{T}:\overline{\mathcal{N}_{L,f^{(1)}}^-}\rightarrow \overline{\mathcal{N}_{L,f^{(2)}}^-}$ by
\begin{equation}\label{TT}
\mathfrak{T}(x_1,x_2,x_3)=\left(x_1,\sqrt{\frac{f^{(2)}(x_1)}{f^{(1)}(x_1)}}x_2,\sqrt{\frac{f^{(2)}(x_1)}{f^{(1)}(x_1)}}x_3\right).
\end{equation}
Since $f^{(j)}\ge \frac{3}{8}>0$ $(j=1,2)$, the transformation $\mathfrak{T}$ is invertible and
$$\mathfrak{T}^{-1}(y_1,y_2,y_3)=\left(y_1,\sqrt{\frac{f^{(1)}(y_1)}{f^{(2)}(y_1)}}y_2,\sqrt{\frac{f^{(1)}(y_1)}{f^{(2)}(y_1)}}y_3\right).
$$
Set
\begin{equation*}
\left\{\begin{split}
&\widetilde{\phi}:=\varphi^{(1)}-\left(\varphi^{(2)}\circ\mathfrak{T}\right),\quad\widetilde{\psi}:=\psi^{(1)}-\left(\psi^{(2)}\circ\mathfrak{T}\right),\\
& \widetilde{f}:=f^{(1)}-f^{(2)},\quad \widetilde{\mathcal{W}}:=\mathcal{W}_{\ast}-\left(\mathcal{W}_{\ast}\circ\mathfrak{T}\right).
\end{split}\right.
\end{equation*}
We first rewrite the nonlinear boundary value problem \eqref{S-Free-BP} for $(\varphi^{(2)},\psi^{(2)}{\bf e}_{\theta})$ in $\mathcal{N}_{L,f^{(2)}}^-$ as a nonlinear boundary value problem for $(\varphi^{(2)}\circ\mathfrak{T},\psi^{(2)}{\bf e}_{\theta}\circ\mathfrak{T})$ in $\mathcal{N}_{L,f^{(1)}}^-$, and subtract the resultant equations and boundary conditions from the nonlinear boundary value problem \eqref{S-Free-BP} for $(\varphi^{(1)},\psi^{(1)}{\bf e}_{\theta})$ in $\mathcal{N}_{L,f^{(1)}}^-$.
Then we get a nonlinear boundary value problem for $(\widetilde{\phi},\widetilde{\psi}{\bf e}_{\theta})$ in $\mathcal{N}_{L,f^{(1)}}^-$.
By adjusting the proof of Lemma \ref{Pro-fix-S} with using
\begin{equation*}
\|\widetilde{\mathcal{W}}\|_{\alpha,\mathcal{N}_{L,f^{(1)}}^-}\le CM_1\sigma\|\widetilde{f}\|_{1,\alpha,(0,L)}
\quad\mbox{and}\quad \|\partial_r\widetilde{\mathcal{W}}\|_{\alpha,\mathcal{N}_{L,f^{(1)}}^-}\le CM_1\sigma\|\widetilde{f}\|_{1,\alpha,(0,L)},
\end{equation*}
we obtain
\begin{equation*}
\begin{split}
\|\widetilde{\phi}\|_{1,\alpha,\mathcal{N}_{L,f^{(1)}}^-}+\|\widetilde{\psi}\|_{1,\alpha,\Omega_{L,f^{(1)}}^-}
\le& C_1^{\ast}(M_1+1)\sigma\left(\|\widetilde{\phi}\|_{1,\alpha,\mathcal{N}_{L,f^{(1)}}^-}+\|\widetilde{\psi}\|_{1,\alpha,\Omega_{L,f^{(1)}}^-}\right)\\
&+C(M_1+1)\sigma\|\widetilde{f}\|_{1,\alpha,(0,L)}
\end{split}
\end{equation*}
for a constant $C_1^{\ast}>0$ depending only on the data and $\alpha$ but independent of $L$.
In the above,
$\Omega_{L,f^{(1)}}^-:=\left\{(x,r)\in\mathbb{R}^2:\, 0<x<L,\, 0<r<f^{(1)}(x)\right\}$
is a two dimensional set.
If it holds that
$$\sigma\le \frac{1}{2C_1^{\ast}(M_1+1)},$$
then we obtain from the previous estimate that
\begin{equation}\label{RR-est}
\|\widetilde{\phi}\|_{1,\alpha,\mathcal{N}_{L,f^{(1)}}^-}+\|\widetilde{\psi}\|_{1,\alpha,\Omega_{L,f^{(1)}}^-}\le C(M_1+1)\sigma\|\widetilde{f}\|_{1,\alpha,(0,L)}.
\end{equation}
By using the free boundary condition \eqref{g-free-cut}, we can express $(\widetilde{f})'$ in terms of $(\widetilde{\phi},\widetilde{\psi},\mathfrak{T},D\mathfrak{T})$.
Then we apply \eqref{RR-est} to obtain the estimate
\begin{equation}\label{3D-g12_EST}
\|(\widetilde{f})'\|_{\alpha,(0,L)}\le C(M_1+1)\sigma\|\widetilde{f}\|_{1,\alpha,(0,L)}.
\end{equation}
To complete the estimate of $\|\widetilde{f}\|_{1,\alpha,(0,L)}=\|\widetilde{f}\|_{0,(0,L)}+\|(\widetilde{f})'\|_{\alpha,(0,L)}$, we now estimate $\|\widetilde{f}\|_{0,(0,L)}$.
Define $\rho^{(1)}$, $u_x^{(1)}$, $\rho^{(2)}$, and $u_x^{(2)}$ by
\begin{equation*}
\begin{split}
&\rho^{(k)}:=H(S_{\ast},{\bf q}(r,\psi^{(k)},D\psi^{(k)},D\varphi^{(k)},\Lambda_{\ast})),\\
& u_x^{(k)}:=\partial_x\varphi^{(k)}+\frac{1}{r}\partial_r(r\psi^{(k)})\quad\mbox{for}\quad k=1,2,
\end{split}
\end{equation*}
where $H$ is given by \eqref{def-H-G}.
By using \eqref{fix-f-eq}, we get
\begin{equation}\label{f12-rhou}
\begin{split}
\int_0^{1/2}&r\left(\rho^{(1)}u_x^{(1)}-\rho^{(2)}u_x^{(2)}\right)(0,r)dr\\
&=\int_0^{f^{(1)}(x)} r\rho^{(1)}u_x^{(1)}(x,r)dr-\int_0^{f^{(2)}(x)} r\rho^{(2)}u_x^{(2)}(x,r)dr.
\end{split}
\end{equation}
Fix $x_0\in[0,L]$.
Without loss of generality, we may assume that
$$f^{(1)}(x_0)<f^{(2)}(x_0).$$
Then \eqref{f12-rhou} can be rewritten as
\begin{equation*}
\begin{split}
\int_0^{1/2}&r\left(\rho^{(1)}u_x^{(1)}-\rho^{(2)}u_x^{(2)}\right)(0,r)dr\\
&=\int_0^{f^{(1)}(x_0)} r\left(\rho^{(1)}u_x^{(1)}-\rho^{(2)}u_x^{(2)}\right)(x_0,r)dr-\int_{f^{(1)}(x_0)}^{f^{(2)}(x_0)} r\rho^{(2)}u_x^{(2)}(x_0,r)dr.
\end{split}
\end{equation*}
By applying \eqref{RR-est}, we have
\begin{equation*}
0\le f^{(2)}(x_0)-f^{(1)}(x_0)\le C(M_1+1)\sigma\|\widetilde{f}\|_{1,\alpha,(0,L)}.
\end{equation*}
Combining this with \eqref{3D-g12_EST}, we finally get
\begin{equation}\label{final-f}
\|\widetilde{f}\|_{1,\alpha,(0,L)}\le C_2^{\ast}(M_1+1)\sigma\|\widetilde{f}\|_{1,\alpha,(0,L)},
\end{equation}
where the constant $C_2^{\ast}>0$ depends only on the data and $\alpha$ but independent of $L$.
We choose $\sigma_5$ as
\begin{equation}\label{Sigma5}
\sigma_5=\min\left\{\sigma_5^{\ast},\frac{1}{2C_1^{\ast}(M_1+1)},\frac{1}{2C_2^{\ast}(M_1+1)}\right\}
\end{equation}
for $\sigma_5^{\ast}$ defined in \eqref{Sigma5star},
so that \eqref{final-f} implies that $f^{(1)}=f^{(2)}$ for $\sigma\le\sigma_5$.
By Lemma \ref{Pro-fix-S}, $(\varphi^{(1)},\psi^{(1)})=(\varphi^{(2)},\psi^{(2)}).$
The proof  of Lemma \ref{Lem-S-free} is completed.
\end{proof}
\subsection{Proof of Proposition \ref{3D-Prop4.1}}
\label{subsection_4_3}
The proof of Proposition \ref{3D-Prop4.1} is divided into four steps.

{\bf 1.}
For a fixed $\mathcal{W}_{\ast}=(S_*, \Lambda_*)\in\mathcal{P}(M_1)$, let $(f,\varphi,\psi)\in C^{2,\alpha}([0,L])\times C^{2,\alpha}(\overline{\mathcal{N}_{L,f}^-})\times C^{2,\alpha}(\overline{\Omega_{L,f}^-})$ be a solution to Problem \ref{Prob4-Fix-S}.
By Lemma \ref{Lem-S-free}, if $\sigma\le\sigma_5$ for $\sigma_5$ given in \eqref{Sigma5}, then there exists a unique solution $(f,\varphi,\psi)$ that satisfies the estimate \eqref{3D-pps-est}.

\begin{lemma} \label{Pro-trans}
Under the same assumptions on $(S_{\rm en},\nu_{\rm en},u_r^{\rm en})$ as in Proposition \ref{3D-Prop4.1},
there exists a small constant $\sigma_4^{\ast\ast}\in(0,\sigma_5]$ depending only on the data and $\alpha$ so that if
$$\sigma=\|S_{\rm en}-S_0\|_{2,\alpha,\Gamma_{\rm en}^-}+\|\nu_{\rm en}\|_{2,\alpha,\Gamma_{\rm en}^-}+\|u_r^{\rm en}\|_{1,\alpha,\Gamma_{\rm en}^-}\le\sigma_4^{\ast\ast},$$
then the initial value problem \eqref{Ite-3D2} has a unique solution $\mathcal{W}=(S,\Lambda)$ satisfying
\begin{equation*}
\|(S,\Lambda)-(S_0^-,0)\|_{1,\alpha,\mathcal{N}_{L,f}^-}\le C^{\ast}\|(S_{\rm en},r\nu_{\rm en})-(S_0,0)\|_{1,\alpha,\Gamma_{\rm en}^-}
\end{equation*}
for a constant $C^{\ast}>0$ depending only on the data and $\alpha$ but independent of $L$.
Furthermore, regarding $(S,\Lambda)$ as functions of $(x,r)\in\Omega_{L,f}^-$, we have 
\begin{equation}\label{S-est-two}
\|(S,\Lambda)-(S_0^-,0)\|_{2,\alpha,\Omega_{L,f}^-}\le C^{\ast\ast}\|(S_{\rm en},r\nu_{\rm en})-(S_0,0)\|_{2,\alpha,\partial\Omega_{L,f}^-\cap\{x=0\}}
\end{equation}
for a constant $C^{\ast\ast}>0$ depending only on the data and $\alpha$ but independent of $L$.
\end{lemma}

\begin{remark}
The estimate \eqref{S-est-two} is needed in \eqref{W-est} to prove the uniqueness of solutions.
\end{remark}

\begin{proof}[Proof of Lemma \ref{Pro-trans}.]

Define a function $w:\overline{\Omega_{L,f}^-}\rightarrow\mathbb{R}$ by
\begin{equation}\label{def-w}
w(x,r):=\int_0^r s{\bf M}\cdot{\bf e}_x(x,s)ds\quad\mbox{for}\quad (x,r)\in\overline{\Omega_{L,f}^-}
\end{equation}
for
$${\bf M}=H(S_{\ast},\nabla\varphi+{\bf t}(r,\psi,D\psi,\Lambda_{\ast}))\left(\nabla\varphi+\frac{1}{r}\partial_r(r\psi){\bf e}_x-(\partial_x\psi){\bf e}_r\right),$$
where ${\bf t}$ and $H$ are given by \eqref{def-T} and \eqref{def-H-G}, respectively.
For such $w$, we consider an invertible function $\mathcal{G}:[0,1/2]\rightarrow [w(0,0),w(0,1/2)]$ satisfying
\begin{equation}\label{def-G0}
\mathcal{G}(r)=w(0,r),
\end{equation}
and define a function $\mathcal{R}_0:\overline{\Omega_{L,f}^-}\rightarrow[0,1/2]$ by
\begin{equation}\label{3D-R0}
\mathcal{R}_0(x,r):=\mathcal{G}^{-1}\circ w(x,r).
\end{equation}
By adjusting the proof of \cite[Proposition 3.5]{bae20183}, we can obtain a unique solution $\mathcal{W}$ of \eqref{Ite-3D2} represented in
\begin{equation}\label{def-W}
\mathcal{W}(x,r)=\mathcal{W}_{\rm en}(\mathcal{R}_0(x,r))\quad\mbox{for}\quad\mathcal{W}_{\rm en}:=(S_{\rm en},r\nu_{\rm en}),
\end{equation}
and the estimate
\begin{equation*}
\|\mathcal{W}-\mathcal{W}_0^-\|_{1,\alpha,\mathcal{N}_{L,f}^-}\le C^{\ast}\|\mathcal{W}_{\rm en}-\mathcal{W}_0^-\|_{1,\alpha,\Gamma_{\rm en}^-}\quad\mbox{for}\quad\mathcal{W}_0^-:=(S_0^-,0),
\end{equation*}
where the constant $C^{\ast}>0$ depends only on the data and $\alpha$ but independent of $L$.

Since $\mathcal{R}_0$ satisfies
\begin{equation*}\label{est-R0}
\|\mathcal{R}_0\|_{2,\alpha,\Omega_{L,f}^-}\le C\|{\bf M}\|_{1,\alpha,\Omega_{L,f}^-},
\end{equation*}
we also have 
\begin{equation*}
\begin{split}
\|\mathcal{W}-\mathcal{W}_0^-\|_{2,\alpha,\Omega_{L,f}^-}
&=\|\mathcal{W}_{\rm en}\circ\mathcal{R}_0-\mathcal{W}_0^-\|_{2,\alpha,\Omega_{L,f}^-}\\
&\le C^{\ast\ast}\|\mathcal{W}_{\rm en}-\mathcal{W}_0^-\|_{2,\alpha,\partial\Omega_{L,f}^-\cap\{x=0\}}
\end{split}
\end{equation*}
for a constant $C^{\ast\ast}>0$ depending only on the data and $\alpha$ but independent of $L$.
The proof of Lemma \ref{Pro-trans} is completed.
\end{proof}

{\bf 2.} (Extension of $(S,\Lambda)$ onto $\mathcal{N}_{L,3/4}^-$)
For $\mathcal{N}_{L,2f}^-:=\mathcal{N}_L\cap\{r<2f(x)\}$ and $\mathcal{N}_{L,2}^-:=\mathcal{N}_L\cap\{r<2\}$, consider a transformation
$\mathfrak{P}_{f}:\overline{\mathcal{N}_{L,2f}^-}\rightarrow \overline{\mathcal{N}_{L,2}^-}$ defined by
\begin{equation*}\label{trans-flat}
\mathfrak{P}_f(x_1,x_2,x_3)=\left(x_1,\frac{x_2}{f(x_1)},\frac{x_3}{f(x_1)}\right).
\end{equation*}
Note that we have shown that $f\in \mathcal{F}(M_2)$ for $\mathcal{F}(M_2)$ given by \eqref{F-set} therefore we have
$f(x_1)\ge \frac 38$ on [0, L], thus the mapping $\mathfrak{P}_f$ is well defined.
And, $\mathfrak{P}_{f}$ is invertible with
$$\mathfrak{P}_{f}^{-1}(y_1,y_2,y_3)=\left(y_1,f(y_1)y_2,f(y_1)y_3\right)\quad\mbox{for}\quad(y_1,y_2,y_3)\in\overline{\mathcal{N}_{L,2}^-}.$$
For the unique solution $\mathcal{W}$ of the initial-value problem \eqref{Ite-3D2},
define $\mathcal{W}^e$ by
\begin{equation*}
\mathcal{W}^e(y_1,y_2,y_3):=\sum_{i=1}^3 c_i\left(\mathcal{W}\circ\mathfrak{P}_f^{-1}\right)\left(y_1,\frac{{ r_{\rm e}}(y_2,y_3)y_2}{i},\frac{{r_{\rm e}}(y_2,y_3)y_3}{i}\right)
\end{equation*}
for $1<\sqrt{y_2^2+y_3^3}\le2,$
where ${r_{\rm e}}$ is defined by
$${r_{\rm e}}(y_2,y_3):=\frac{2-\sqrt{y_2^2+y_3^3}}{\sqrt{y_2^2+y_3^3}}.$$
Here,
$c_1=6$, $c_2=-32$, and $c_3=27$, which are constants determined by the system of equations
\begin{equation*}
\sum_{i=1}^3 c_i\left(-\frac{1}{i}\right)^m=1,\quad m=0,1,2.
\end{equation*}
For such $\mathcal{W}^e$, define an extension of $\mathcal{W}$ into $\mathcal{N}_{L,4/3}^-$ as follows:
\begin{equation}\label{ext-W-def}
\mathcal{E}_f(\mathcal{W})(x_1,x_2,x_3):=\left\{\begin{split}
\mathcal{W}(x_1,x_2,x_3)\quad\mbox{for }&  \sqrt{x_2^2+x_3^2}\le f(x_1),\\
\mathcal{W}^e\circ\mathfrak{P}_f(x_1,x_2,x_3)\quad\mbox{for }& f(x_1)< \sqrt{x_2^2+x_3^3}<\frac{3}{4}.
\end{split}\right.
\end{equation}
Since $f(x_1)\ge \frac 38$ on $[0,L]$, $\mathcal{E}_f$ is well defined by \eqref{ext-W-def}, and it satisfies
\begin{equation}\label{W-f-est}
\|\mathcal{E}_f(\mathcal{W})-\mathcal{W}_0^-\|_{1,\alpha,\mathcal{N}_{L,3/4}^-}\le C\|\mathcal{W}-\mathcal{W}_0^-\|_{1,\alpha,\mathcal{N}_{L,f}^-}.
\end{equation}
We define an iteration mapping $\mathcal{J}:\mathcal{P}(M_1)\rightarrow \left[C^{1,\alpha/2}(\overline{\mathcal{N}_{L,3/4}^-})\right]^2$ by
\begin{equation*}
\mathcal{J}(\mathcal{W}_{\ast})=\mathcal{E}_f(\mathcal{W}).
\end{equation*}
By \eqref{W-f-est} and Lemma \ref{Pro-trans}, we have the estimate
\begin{equation*}
\|\mathcal{E}_f(\mathcal{W})-\mathcal{W}_0^-\|_{1,\alpha,\mathcal{N}_{L,3/4}^-}\le C\|\mathcal{W}-\mathcal{W}_0^-\|_{1,\alpha,\mathcal{N}_{L,f}^-}\le C_1^{\star}\sigma
\end{equation*}
for a constant $C_1^{\star}>0$ depending only on the data and $\alpha$ but independent of $L$.

{\bf 3.} (Further estimate of $\frac{\Lambda}{r}$$(=\mathcal{V})$)
By \eqref{3D-Cut-BC} and \eqref{def-W}, $\Lambda$ is represented as
$$\Lambda(x,r)=\mathcal{R}_0(x,r)\nu_{\rm en}(\mathcal{R}_0(x,r))\quad\mbox{for}\quad(x,r)\in\Omega_{L,f}^-,$$
where $\mathcal{R}_0$ is given by \eqref{3D-R0}.
Set $\mathcal{V}$ as
\begin{equation}\label{V-est}
\mathcal{V}(x,r)=\left\{\begin{split}
\frac{\mathcal{E}_f(\Lambda)(x,r)}{r}\quad&\mbox{for}\quad (x,r)\in[0,L]\times(f(x),\frac{3}{4}),\\
\frac{\mathcal{R}_0(x,r)}{r}\nu_{\rm en}(\mathcal{R}_0(x,r))\quad&\mbox{for}\quad(x,r)\in[0,L]\times(0,f(x)],\\
0\quad&\mbox{for}\quad (x,r)\in[0,L]\times\{0\}.
\end{split}\right.
\end{equation}
By the compatibility condition $\nu_{\rm en}'(0)=0$ and  the representation
\begin{equation*}
\begin{split}
\partial_r\mathcal{V}(x,r)&=\frac{\mathcal{R}_0(x,r)}{r}\nu_{\rm en}'(\mathcal{R}_0(x,r))\partial_r\mathcal{R}_0(x,r)\\
&\quad+\left(\frac{\partial_r\mathcal{R}_0(x,r)}{r}-\frac{\mathcal{R}_0(x,r)}{r^2}\right)\nu_{\rm en}(\mathcal{R}_0(x,r)),
\end{split}
\end{equation*}
we get
\begin{equation*}
\lim_{r\rightarrow 0+}\partial_r\mathcal{V}(x,r)=\left(\partial_r\mathcal{R}_0(x,0)\right)^2\nu_{\rm en}'(0)=0.
\end{equation*}
With using this observation, it can be directly checked that
\begin{equation}\label{V-est-2}
\|\mathcal{V}\|_{1,\alpha,\Omega_{L,f}^-}\le C\sigma.
\end{equation}
By \eqref{V-est}-\eqref{V-est-2} and the definition of $\mathcal{E}_f(\Lambda)$, we have the estimate
$$\|\mathcal{V}\|_{1,\alpha,\Omega_{L,3/4}^-}\le C_2^{\star}\sigma$$
for a constant $C_2^{\star}>0$ depending only on the data and $\alpha$ but independent of $L$.

{\bf 4.}
In this step, we finally choose $(M_1,\sigma_4)$ so that $\mathcal{J}$ has a unique fixed point in $\mathcal{P}(M_1)$.

By a direct computation, one can easily check that there exists a constant $\epsilon_4>0$ depending only on the data and $\alpha$ so that if
$$(M_1+1)\sigma\le \epsilon_4,$$
 then
\begin{equation*}
\|H(S_{\ast},{\bf q}(r,\psi,D\psi,D\varphi,\Lambda_{\ast})){\bf q}(r,\psi,D\psi,D\varphi,\Lambda_{\ast})-\rho_0^-u_0{\bf e}_x\|_{0,\mathcal{N}_{L,f}^-}\le C_3^{\star}(M_1+1)\sigma
\end{equation*}
for a constant $C_3^{\star}>0$ depending only on the data and $\alpha$ but independent of $L$.
If it holds that
$$\sigma\le\frac{1}{2C_3^{\star}(M_1+1)} ,$$
then we obtain from the previous estimate that
\begin{equation}\label{rhou-positive}
\|H(S_{\ast},{\bf q}(r,\psi,D\psi,D\varphi,\Lambda_{\ast})){\bf q}(r,\psi,D\psi,D\varphi,\Lambda_{\ast})-\rho_0^-u_0{\bf e}_x\|_{0,\mathcal{N}_{L,f}^-}\le \frac{\rho_0^-u_0}{2}.
\end{equation}
Also, by the boundary conditions in \eqref{S-Free-BP} for $(\varphi,\psi)$ and the definition of $\varphi_{\rm en}$ given in \eqref{def-varphi-en},  we have
\begin{equation}\label{partial-x-S}
\partial_r\varphi-\partial_x\psi\equiv0\quad\mbox{on}\quad (\Gamma_{\rm en}^-\cap\{r\ge\frac{1}{2}-\epsilon\})\cup\Gamma_{\rm ex}^{L,f}.
\end{equation}
It follows from \eqref{Ite-3D2} and \eqref{rhou-positive}-\eqref{partial-x-S} that
$$\left(\partial_x\mathcal{E}_f(S),\partial_x\mathcal{E}_f(\Lambda)\right)\equiv {\bf 0}\quad \mbox{on}\quad (\Gamma_{\rm en}^-\cap\{r\ge\frac{1}{2}-\epsilon\})\cup \Gamma_{\rm ex}^{L,3/4}.$$

Choose $M_1$ and $\sigma_4^{\ast}$ as
\begin{equation}\label{Sigma4star}
M_1=2\left(C_1^{\star}+C_2^{\star}\right)\quad\mbox{and}\quad\sigma_4^{\ast}=\min\left\{\sigma_5,\sigma_4^{\ast\ast},\frac{\epsilon_4}{M_1+1}, \frac{1}{2C_3^{\star}(M_1+1)}\right\}
\end{equation}
with $\sigma_5$ defined in \eqref{Sigma5} and $\sigma_4^{\ast\ast}$ given in Lemma \ref{Pro-trans}.
Under such choices of $(M_1,\sigma_4^{\ast})$, the mapping $\mathcal{J}$ maps $\mathcal{P}(M_1)$ into itself whenever $\sigma\le\sigma_4^{\ast}$.

The iteration set $\mathcal{P}(M_1)$ given by \eqref{Ent-Ang-set} is convex and compact subset in $[C^{1,\alpha/2}(\overline{\mathcal{N}_{L,3/4}^-})]^2$.
Suppose that a sequence $\{\mathcal{W}_{\ast}^{(k)}\}_{k=1}^{\infty}:=\{(S_{\ast}^{(k)},\Lambda_{\ast}^{(k)})\}_{k=1}^{\infty}\subset\mathcal{P}(M_1)$ converges in $C^{1,\alpha/2}(\overline{\mathcal{N}_{L,3/4}^-})$ to $\mathcal{W}_{\ast}^{(\infty)}:=(S_{\ast}^{(\infty)},\Lambda_{\ast}^{(\infty)})\in\mathcal{P}(M_1)$.
For each $k\in\mathbb{N}\cup\{\infty\}$, set
\begin{equation*}
\mathcal{W}^{(k)}:=\mathcal{J}(\mathcal{W}_{\ast}^{(k)}).
\end{equation*}
And, let $(f^{(k)},\varphi^{(k)},\psi^{(k)})\in C^{2,\alpha}([0,L])\times C^{2,\alpha}(\overline{\mathcal{N}_{L,f^{(k)}}^-})\times C^{2,\alpha}(\overline{\Omega_{L,f^{(k)}}^-})$ be the unique solution of Problem \ref{Prob4-Fix-S} associated with $\mathcal{W}_{\ast}=\mathcal{W}_{\ast}^{(k)}$.
By the uniqueness of a solution for Problem \ref{Prob4-Fix-S}, $\{f^{(k)}\}_{k=1}^{\infty}$ is convergent in $C^{2,\alpha/2}([0,L])$.
Denote its limit by $f^{(\infty)}$ and the unique solution of \eqref{Fixed-BVP} associated with $(f_{\ast},\mathcal{W}_{\ast})=(f^{(\infty)},\mathcal{W}_{\ast}^{(\infty)})$ by $(\varphi^{(\infty)},\psi^{(\infty)})$.
Define a transformation $T^{(k)}:\overline{\mathcal{N}_{L,f^{(\infty)}}^-}\rightarrow\overline{\mathcal{N}_{L,f^{(k)}}^-}$ by
\begin{equation*}
T^{(k)}(x_1,x_2,x_3)=\left(x_1,\sqrt{\frac{f^{(k)}(x)}{f^{(\infty)}(x)}}x_2,\sqrt{\frac{f^{(k)}(x)}{f^{(\infty)}(x)}}x_3\right),
\end{equation*}
and set
\begin{equation*}
\begin{split}
{\bf M}^{(k)}:=H&\left(S_{\ast}^{(k)},\nabla\varphi^{(k)},{\bf t}(r,\psi^{(k)},D\psi^{(k)},\Lambda_{\ast}^{(k)})\right)\left(\nabla\varphi^{(k)}+\frac{1}{r}\partial_r(r\psi^{(k)}){\bf e}_x-(\partial_x\psi^{(k)}){\bf e}_r\right),
\end{split}
\end{equation*}
where  ${\bf t}$ and $H$ are given by \eqref{def-T} and \eqref{def-H-G}, respectively.
Then ${\bf M}^{(k)}\circ T^{(k)}$ converges to ${\bf M}^{(\infty)}$ in $C^{1,\alpha/2}(\overline{\mathcal{N}_{L,f^{(\infty)}}^-}).$
By Lemma \ref{Pro-trans}, $\mathcal{W}^{(k)}$ converges to $\mathcal{W}^{(\infty)}$ in $C^{1,\alpha/2}(\overline{\mathcal{N}_{L,3/4}^-})$.
This implies that $\mathcal{J}(\mathcal{W}_{\ast}^{(k)})$ converges to $\mathcal{J}(\mathcal{W}_{\ast}^{(\infty)})$ in $C^{1,\alpha/2}(\overline{\mathcal{N}_{L,3/4}^-})$.
Thus $\mathcal{J}$ is a continuous map in $[C^{1,\alpha/2}(\overline{\mathcal{N}_{L,3/4}^-})]^2$.
Applying the Schauder fixed point theorem yields that $\mathcal{J}$ has a fixed point $\mathcal{W}=\mathcal{E}_f(S,\Lambda)\in\mathcal{P}(M_1)$.
For such $\mathcal{W}$, let $(f,\varphi,\psi)$ be the unique solution of Problem \ref{Prob4-Fix-S}, and let us set $(S,\Lambda):=\left.\mathcal{E}_f(S,\Lambda)\right|_{\mathcal{N}_{L,f}^-}$.
Then $(f,S,\Lambda,\varphi,\psi)$ solves Problem \ref{Prob3-Cut} provided that $\sigma\le\sigma_4^{\ast}$.

Finally, we prove the uniqueness of a fixed point of $\mathcal{J}$.
Let $(f^{(1)},\mathcal{W}^{(1)},\varphi^{(1)},\psi^{(1)})$ and $(f^{(2)},\mathcal{W}^{(2)},\varphi^{(2)},\psi^{(2)})$ be two solutions to Problem \ref{Prob3-Cut}, 
and suppose that each solution satisfies the estimates given in \eqref{3D-Prop-est} of Proposition \ref{3D-Prop4.1}.
For a transformation $\mathfrak{T}:\overline{\mathcal{N}_{L,f^{(1)}}^-}\rightarrow\overline{\mathcal{N}_{L,f^{(2)}}^-}$ defined in \eqref{TT}, set
\begin{equation*}
\left\{\begin{split}
&\widetilde{\phi}:=\varphi^{(1)}-\left(\varphi^{(2)}\circ\mathfrak{T}\right),\quad \widetilde{\psi}:=\psi^{(1)}-\left(\psi^{(2)}\circ\mathfrak{T}\right),\\
&\widetilde{\mathcal{W}}:=\mathcal{W}^{(1)}-\left(\mathcal{W}^{(2)}\circ\mathfrak{T}\right),\quad \widetilde{f}:=f^{(1)}-f^{(2)}.
\end{split}\right.
\end{equation*}
By a direct computation, it can be checked that there exists a constant $\sigma_4'>0$ depending only on the data and $\alpha$ but independent of $L$ so that if $\sigma\le \sigma_4'$, then
\begin{equation}\label{W-est}
\begin{split}
\|\widetilde{\mathcal{W}}\|_{\alpha,\mathcal{N}_{L,f^{(1)}}^-}&+\|\partial_r\widetilde{\mathcal{W}}\|_{\alpha,\mathcal{N}_{L,f^{(1)}}^-}\\
&\le C\sigma\left(\|\widetilde{\phi}\|_{1,\alpha,\mathcal{N}_{L,f^{(1)}}^-}+\|\widetilde{\psi}\|_{1,\alpha,\Omega_{L,f^{(1)}}^-}+\|\widetilde{f}\|_{1,\alpha,(0,L)}\right)\\
&\le C\sigma\|\widetilde{f}\|_{1,\alpha,(0,L)}.
\end{split}
\end{equation}
By adjusting the proof of Lemma \ref{Lem-S-free} with using the estimate \eqref{W-est}, we have
\begin{equation}\label{3D-g}
\|\widetilde{f}\|_{1,\alpha,(0,L)}\le C_5^{\star}\sigma\|\widetilde{f}\|_{1,\alpha,(0,L)}
\end{equation}
for a constant $C_5^{\star}>0$ depending only on the data and $\alpha$ but independent of $L$.
We choose $\sigma_4$ as
\begin{equation*}\label{Sigma4}
\sigma_4=\min\left\{\sigma_4^{\ast},\sigma_4',\frac{1}{2C_5^{\star}}\right\}
\end{equation*}
with $\sigma_4^{\ast}$ defined in \eqref{Sigma4star}, so that we obtain from \eqref{3D-g} that $f^{(1)}=f^{(2)}$ for $\sigma\le\sigma_4$.
Then, by \eqref{W-est}, we have $\mathcal{W}^{(1)}=\mathcal{W}^{(2)}.$
Therefore
$$(f^{(1)},\mathcal{W}^{(1)},\varphi^{(1)},\psi^{(1)})=(f^{(2)},\mathcal{W}^{(2)},\varphi^{(2)},\psi^{(2)})$$
by Lemma \ref{Lem-S-free}.
The proof  of Proposition \ref{3D-Prop4.1} is completed.
\qed

\section{Free boundary problem in the infinitely long cylinder $\mathcal{N}$}\label{3D-sec-ex}
\subsection{Proof of Theorem \ref{3D-Thm-HD}}\label{5-1}
Let $\sigma_4$ be from Proposition \ref{3D-Prop4.1} and suppose that $\sigma\le \sigma_4$.
By Proposition \ref{3D-Prop4.1}, Problem \ref{Prob3-Cut} has a solution for each $L>0$.
For each $m\in\mathbb{N}$, let $(f^{(m)},S^{(m)},\Lambda^{(m)},\varphi^{(m)},\psi^{(m)})$ be a solution of Problem \ref{Prob3-Cut} in $\mathcal{N}_{m+20}:=\mathcal{N}\cap\{0<x<m+20\}$, and suppose that the solution satisfies the estimates \eqref{3D-Prop-est} given in Proposition \ref{3D-Prop4.1}.
Then, using the Arzel\'a-Ascoli theorem and a diagonal procedure, we can extract a subsequence, still written as $\{(f^{(m)},S^{(m)},\Lambda^{(m)},\varphi^{(m)},\psi^{(m)})\}_{m\in\mathbb{N}}$ so that the subsequence converges to functions $(f^{\ast},S^{\ast},\Lambda^{\ast},\varphi^{\ast},\psi^{\ast})$ in the following sense: for any $L>0$,
\begin{itemize}
\item[(i)] $f^{(m)}$ converges to $f^{\ast}$ in $C^2$ in $[0,L]$.
\item[(ii)] $(S^{(m)}\circ{T}^{(m)},\Lambda^{(m)}\circ{T}^{(m)})$ converges to $(S^{\ast},\Lambda^{\ast})$ in $C^1$ in $\overline{\mathcal{N}_{L,f^{\ast}}^-}$, where ${T}^{(m)}:\overline{\mathcal{N}_{m+20,f^{\ast}}^-}\rightarrow\overline{\mathcal{N}_{m+20,f^{(m)}}^-}$ is defined by
\begin{equation*}
T^{(m)}(x_1,x_2,x_3)=\left(x_1,x_2\sqrt{\frac{f^{(m)}(x_1)}{f^{\ast}(x_1)}},x_3\sqrt{\frac{f^{(m)}(x_1)}{f^{\ast}(x_1)}}\right).
\end{equation*}
\item[(iii)] $\left(\varphi^{(m)}\circ{T}^{(m)},(\psi^{(m)}{\bf e}_{\theta})\circ{T}^{(m)}\right)$ converges to $\left(\varphi^{\ast},\psi^{\ast}{\bf e}_{\theta}\right)$ in $C^2$ in $\overline{\mathcal{N}_{L,f^{\ast}}^-}$.
\item[(iv)] $\psi^{(m)}\circ{T}^{(m)}$ converges to $\psi^{\ast}$ in $C^2$ in $\overline{\Omega_{L,f^{\ast}}^-}$.
\end{itemize}
By a change of variables and passing to the limit $m\rightarrow\infty$, one can easily show that
$(f^{\ast},S^{\ast},\Lambda^{\ast},\varphi^{\ast},\psi^{\ast})$ is a solution to the free boundary problem \eqref{3D-H} with boundary conditions \eqref{g-free-cond} and \eqref{3D-BC-C}.
Furthermore, it follows from the $C^2$-convergence of $\{(f^{(m)},\varphi^{(m)},\psi^{(m)}{\bf e}_{\theta})\}_{m\in\mathbb{N}}$, $C^1$-convergence of $\{(S^{(m)},\Lambda^{(m)})\}_{m\in\mathbb{N}}$, and the estimates \eqref{3D-Prop-est} given in Proposition \ref{3D-Prop4.1} that  $(f^{\ast},S^{\ast},\Lambda^{\ast},\varphi^{\ast},\psi^{\ast}{\bf e}_{\theta})$ satisfy the estimates \eqref{Thm-HD-est} for a constant $C>0$ depending only on the data and $\alpha$.
\qed

\subsection{Proof of Theorem \ref{3D-MainThm}(a)}\label{5-2}
Let $\sigma_3$ be from Theorem \ref{3D-Thm-HD}, and suppose that the functions $(S_{\rm en},\nu_{\rm en},u_r^{\rm en})$ satisfy \eqref{enu-sigma}.
By Theorem \ref{3D-Thm-HD}, the free boundary problem \eqref{3D-H} with \eqref{g-free-cond} and \eqref{3D-BC-C} has a solution $(g_D,S,\Lambda,\varphi,\psi)$ that satisfies the estimates \eqref{Thm-HD-est}.
For such a solution, we define $({\bf u},\rho, p)$ by
\begin{equation*}
\left.\begin{split}
&{\bf u}:=\left(\partial_x\varphi+\frac{1}{r}\partial_r(r\psi)\right){\bf e}_x+(\partial_r\varphi-\partial_x\psi){\bf e}_r+\frac{\Lambda}{r}{\bf e}_{\theta},\\
&\rho:=H\left(S,{\bf u}\right),\quad p:=S\rho^{\gamma}\quad\mbox{in}\quad \overline{\mathcal{N}_{g_D}^-},
\end{split}\right.
\end{equation*}
where $H$ is given by \eqref{def-H-G}.
It follows from the estimates \eqref{Thm-HD-est} given in Theorem \ref{3D-Thm-HD} that $(g_D,{\bf u},\rho,p)$ satisfy the estimate \eqref{Thm2.1-uniq-est}.
Then, one can choose a small constant $\sigma_1\in(0,\sigma_3]$ depending only on the data and $\alpha$ such that if $\sigma\le\sigma_1$, then
$(g_D, {\bf u}, \rho, p)$ satisfy $\rho\ge\frac{1}{2}\rho_0^->0$ and $c^2-|{\bf u}|^2\ge\frac{1}{2}\left((c_0^-)^2-u_0^2\right)>0$ in $\overline{\mathcal{N}_{g_D}^-}$, thus solve Problem \ref{3D-Problem2}.
Here, $c_0^-$ is given by $c_0^-=\sqrt{\frac{\gamma p_0}{\rho_0^-}}$.
The proof of Theorem \ref{3D-MainThm}(a) is completed.
\qed

\subsection{Proof of Theorem \ref{3D-MainThm}(b)}\label{sec-far}

Let $\sigma_1$ be from Theorem \ref{3D-MainThm}(a).
By Theorem \ref{3D-MainThm}(a), if $\sigma\le\sigma_1$, then there exists a solution $(g_D,{\bf u},\rho,p)$ with ${\bf u}=u_x{\bf e}_x+u_r{\bf e}_r+u_{\theta}{\bf e}_{\theta}$ of Problem \ref{3D-Problem2} satisfying the estimate \eqref{Thm2.1-uniq-est}.

Set
\begin{equation*}
\begin{split}
\Omega_{g_D}^-&:=\left\{(x,r)\in\mathbb{R}^2: x> 0, 0< r<g_D(x)\right\},\\
\Gamma_{\rm en}^{g_D}&:=\partial\Omega_{g_D}^-\cap\{x=0\},\quad
\Gamma_{\rm cd}^{g_D}:=\partial\Omega_{g_D}^-\cap\{r=g_D(x)\}.\\
\end{split}
\end{equation*}
The equation
$\partial_x(\rho u_x)+\partial_r(\rho u_r)+\frac{\rho u_r}{r}=0$ in $\Omega_{g_D}^-$, stated in \eqref{3D-ang}, can be rewritten as
\begin{equation*}\label{conti-eq}
\frac{\partial_x(r\rho u_x)+\partial_r(r\rho u_r)}{r}=0\quad\mbox{in}\quad\Omega_{g_D}^-.
\end{equation*}
With using this equation, it can be directly checked that the function $\mathfrak{h}$ given by
\begin{equation*}
\mathfrak{h}(x,r):=\int_0^rt\rho u_x(x,t)dt\quad\mbox{for}\quad (x,r)\in\overline{\Omega_{g_D}^-}
\end{equation*}
satisfies
\begin{equation}\label{3D-st-f}
\partial_x\mathfrak{h}=-r\rho u_r,\quad\partial_r\mathfrak{h}=r\rho u_x.
\end{equation}
By \eqref{def-w}-\eqref{def-W}, the entropy $S(=p/\rho^{\gamma})$ and angular momentum density $\Lambda(=ru_{\theta})$ are represented as
\begin{equation}\label{S-Lambda}
\begin{split}
&S(x,r)=S_{\rm en}\circ\mathcal{G}^{-1}(\mathfrak{h}(x,r))=:{S}(\mathfrak{h}(x,r)),\\
&\Lambda(x,r)=\Lambda_{\rm en}\circ\mathcal{G}^{-1}(\mathfrak{h}(x,r)):={\Lambda}(\mathfrak{h}(x,r))\quad\mbox{for } (x,r)\in\overline{\Omega_{g_D}^-},
\end{split}
\end{equation}
where $\mathcal{G}$ is given by  \eqref{def-G0} associated with $w=\mathfrak{h}$ and $\Lambda_{\rm en}(r):=r\nu_{\rm en}(r)$ for $r\in[0,1/2]$.
Since $S_{\rm en}$, $\Lambda_{\rm en}$, and $\mathcal{G}^{-1}$
are differentiable, $S$ and $\Lambda$ are differentiable functions of $\mathfrak{h}$.
Set
$$\mathfrak{S}(\mathfrak{h}):=\frac{\gamma}{\gamma-1}{S}(\mathfrak{h}).$$
Then, by the definition of the Bernoulli invariant \eqref{Ber-inv}, we have
\begin{equation}\label{3D-Ber}
B_0^-r^2\rho^2=\frac{1}{2}\left(|\nabla\mathfrak{h}|^2+\Lambda(\mathfrak{h})^2\rho^2\right)+r^2\mathfrak{S}(\mathfrak{h})\rho^{\gamma+1}\quad\mbox{in}\quad \overline{\Omega_{g_D}^-},
\end{equation}
where $\nabla=(\partial_x,\partial_r)$.
By differentiating the equation \eqref{3D-Ber} with respect to $x$ and $r$, we have
\begin{equation}\label{3D-par-rho}
\begin{split}
&\partial_x\rho=-\frac{(\partial_x\mathfrak{h})(\partial_{xx}\mathfrak{h}+\Lambda\Lambda'\rho^2+r^2\mathfrak{S}'\rho^{\gamma+1})+(\partial_r\mathfrak{h})(\partial_{rx}\mathfrak{h})}{r^2(\gamma+1)\mathfrak{S}\rho^{\gamma}-2r^2B_0^-\rho+\Lambda^2\rho},\\
&\partial_r\rho=-\frac{(\partial_x\mathfrak{h})(\partial_{xr}\mathfrak{h}-\partial_x\mathfrak{h})+(\partial_r\mathfrak{h})(\partial_{rr}\mathfrak{h}+\Lambda\Lambda'\rho^2+r^2\mathfrak{S}'\rho^{\gamma+1}-\partial_r\mathfrak{h})-\Lambda^2\rho^2}{r^2(\gamma+1)\mathfrak{S}\rho^{\gamma}-2r^2B_0^-\rho+\Lambda^2\rho},
\end{split}
\end{equation}
where $'$ denotes the derivative with respect to $\mathfrak{h}$.
Using \eqref{3D-st-f}-\eqref{3D-par-rho}, the equation
\begin{equation}\label{ES-Far-eq}
\rho(u_x\partial_x+u_r\partial_r)u_r-\frac{\rho u_{\theta}^2}{r}+\partial_rp=0\quad\mbox{in}\quad\Omega_{g_D}^-
\end{equation}
in \eqref{3D-ang} can be rewritten as
\begin{equation}\label{3D-ES-st}
-\left(\frac{\partial_r\mathfrak{h}}{r}\right)\nabla\cdot\left(\frac{\nabla\mathfrak{h}}{r\rho}\right)-\frac{(\partial_r\mathfrak{h})\mathfrak{S}'\rho^{\gamma}}{\gamma}-\frac{(\partial_r\mathfrak{h})\Lambda\Lambda'\rho}{r^2}=0\quad\mbox{in}\quad\Omega_{g_D}^-.
\end{equation}
We multiply \eqref{3D-ES-st} by $r/(\partial_r\mathfrak{h})$ to get
\begin{equation}\label{3D-Stream}
\nabla\cdot\left(\frac{\nabla\mathfrak{h}}{r\rho}\right)=-\frac{r}{\gamma}\mathfrak{S}'\rho^{\gamma}-\frac{\Lambda\Lambda'\rho}{r}\quad\mbox{in}\quad\Omega_{g_D}^-.
\end{equation}
Set
$$\omega:=\partial_x\mathfrak{h}$$
and differentiate \eqref{3D-Stream} with respect to $x$ to get the following equation for $\omega$:
\begin{equation}\label{3D-St-Diff}
\partial_i\left(\frac{\mathfrak{q}_{ij}}{r\rho^2}\partial_j\omega\right)+\partial_i\left(\frac{\mathfrak{q}_1\partial_i\mathfrak{h}}{r\rho^2}\omega\right)
=\mathfrak{q}_2\omega+\mathfrak{q}_3(\partial_i\mathfrak{h})(\partial_i\omega)\quad\mbox{in}\quad \Omega_{g_D}^-,
\end{equation}
where
\begin{equation}\label{def-q12}
\begin{split}
&\mathfrak{O}:=r^2(\gamma+1)\mathfrak{S}\rho^{\gamma}-2r^2B_0^-\rho+\Lambda^2\rho,\\
&\mathfrak{q}_{ij}:=\rho\delta_{ij}+\frac{(\partial_i\mathfrak{h})(\partial_j\mathfrak{h})}{\mathfrak{O}},\\
&\mathfrak{q}_1:=\frac{\Lambda\Lambda'\rho^2+r^2\mathfrak{S}'\rho^{\gamma+1}}{\mathfrak{O}},\\
&\mathfrak{q}_2
:=-\frac{r}{\gamma}\mathfrak{S}''\rho^{\gamma}-\frac{(\Lambda')^2\rho}{r}-\frac{\Lambda\Lambda''\rho}{r}+\frac{(\Lambda\Lambda'\rho^2+r^2\mathfrak{S}'\rho^{\gamma+1})^2}{r\rho^2\mathfrak{O}},\\
&\mathfrak{q}_3:=\frac{1}{\mathfrak{O}}\left(r\mathfrak{S}'\rho^{\gamma-1}+\frac{\Lambda\Lambda'}{r}\right).
\end{split}
\end{equation}

Note that ${\bf u}$ is represented by \eqref{3D-u}, for $(\varphi, \psi, h)$ solving the equations \eqref{3D-H}. Similarly to \eqref{3D-psi}, we rewrite the second equation in \eqref{3D-H} as
\begin{equation*}
-\left(\partial_{xx}+\frac{1}{r}\partial_r(r\partial_r)-\frac{1}{r^2}\right)\psi=\frac{1}{{\bf u}\cdot{\bf e}_x}\left(\frac{H^{\gamma-1}(S,{\bf u})}{\gamma-1}\partial_rS+\frac{\Lambda}{r^2}\partial_r\Lambda\right)
\quad\mbox{in}\quad\Omega_{g_D}^-.
\end{equation*}
By Theorem \ref{3D-MainThm}(a) and Lemma \ref{Pro-trans}, the right-hand side of this equation is $C^{1,\alpha}$ in $\Omega_{g_D}^-$, therefore we have $\psi\in C^{3,\alpha}(\Omega_{g_D}^-)$. Next, we regard the first equation in \eqref{3D-H} as a second order quasilinear equation for $\varphi$. By Theorem \ref{3D-MainThm}(a), this equation is uniformly elliptic. Since $\varphi$ is $C^{2,\alpha}$ in $\mathcal{N}_{g_D}^{-}$, and $\psi\in C^{3,\alpha}(\Omega_{g_D}^-)$, we obtain that $\varphi$ is $C^{3,\alpha}$ in $\Omega_{g_D}^-$. 
And, this implies that $\mathfrak{h}\in C^{3,\alpha}(\Omega_{g_D}^-)$, thus the equation \eqref{3D-St-Diff} is well-defined.

By the boundary conditions \eqref{Prob2-BC-ent} and the compatibility condition $u_r=0$ on $\{r=0\}$, $\omega$ satisfies
\begin{equation}\label{omega-BC-0}
\omega=-r\rho u_r^{\rm en}\quad\mbox{on}\quad\Gamma_{\rm en}^{g_D},\quad\omega=0\quad\mbox{on}\quad\partial\Omega_{g_D}^-\cap\{r=0\}.
\end{equation}

Next, we compute a conormal boundary condition for \eqref{3D-St-Diff} on $\Gamma_{\rm cd}^{g_D}$.

We consider the expression
\begin{equation}\label{3D-fxr2}
(\partial_x\mathfrak{h})^2+(\partial_r\mathfrak{h})^2=\mathcal{C}_1(g_D(x))^2-\mathcal{C}_2\quad\mbox{on}\quad\Gamma_{\rm cd}^{g_D}=\partial\Omega_{g_D}^-\cap\{r=g_D(x)\}
\end{equation}
for
\begin{equation}\label{def-C12}
\mathcal{C}_1:=\frac{(\partial_x\mathfrak{h})^2+(\partial_r\mathfrak{h})^2+\Lambda^2\rho^2}{r^2}(x,g_D(x)),\quad\mathcal{C}_2:=\Lambda^2 \rho^2(x,g_D(x)).\end{equation}
Since we have
\begin{equation}\label{S-cont}
S=S_{\rm en}\left(\frac 12\right),\quad\Lambda=\Lambda_{\rm en}\left(\frac 12\right),\quad p=p_0\quad\mbox{on}\quad\Gamma_{\rm cd}^{g_D},
\end{equation}
we obtain that
\begin{equation}\label{rho-cont}
\rho=\left(\frac{p_0}{S_{\rm en}(\frac 12)}\right)^{1/\gamma},
\end{equation}
from which it follows that $\mathcal{C}_2$ in \eqref{def-C12} is given by
\begin{equation*}
\mathcal{C}_2=\Lambda_{\rm en}^2(\frac 12)p_0^{2/\gamma}S_{\rm en}^{-2/\gamma}(\frac 12).
\end{equation*}
A direct computation with using \eqref{Ber-inv}, \eqref{3D-st-f}, and \eqref{S-cont}-\eqref{rho-cont} yields that
\begin{equation*}
\begin{split}
\mathcal{C}_1
&=2\left(B_0^--\frac{\gamma}{\gamma-1}p_0^{1-1/\gamma}S_{\rm en}^{1/\gamma}(\frac 12)\right)p_0^{2/\gamma}S_{\rm en}^{-2/\gamma}(\frac 12).
\end{split}
\end{equation*}
By differentiating the equation \eqref{3D-fxr2} in the tangential direction along $\Gamma_{\rm cd}^{g_D}$,
we have
\begin{equation*}
(\partial_x\mathfrak{h})\left(\partial_{xx}\mathfrak{h}+g_D'(x)\partial_{xr}\mathfrak{h}\right)+(\partial_r\mathfrak{h})\left(\partial_{rx}\mathfrak{h}+g_D'(x)\partial_{rr}\mathfrak{h}\right)=\mathcal{C}_1g_D(x)g_D'(x)\quad\mbox{on}\quad\Gamma_{\rm cd}^{g_D}.
\end{equation*}
And, we solve this expression for $\partial_{xr}\mathfrak{h}$ to get
\begin{equation}\label{3D-FXR}
\partial_{xr}\mathfrak{h}
=-\left(\frac{\mathcal{C}_1g_D(x)}{\partial_r\mathfrak{h}}+\partial_{xx}\mathfrak{h}-\partial_{rr}\mathfrak{h}\right)\frac{\omega}{(\partial_x\mathfrak{h})g_D'(x)+(\partial_r\mathfrak{h})}\quad\mbox{on}\quad\Gamma_{\rm cd}^{g_D}.
\end{equation}
Substituting the expression of $\mathcal{C}_1$ in \eqref{def-C12} into \eqref{3D-FXR}, we have
\begin{equation}\label{q22term}
\partial_r\omega=\partial_{xr}\mathfrak{h}
=\left(\frac{-\partial_{xx}\mathfrak{h}+g_D(x)\mathfrak{E}}{(\partial_x\mathfrak{h})g_D'(x)+(\partial_r\mathfrak{h})}\right)\omega\quad\mbox{on}\quad\Gamma_{\rm cd}^{g_D}
\end{equation}
for
\begin{equation*}
\mathfrak{E}:=\frac{1}{\partial_r\mathfrak{h}}\left\{-\left(\frac{\partial_x\mathfrak{h}}{r}\right)^2-\left(\frac{\Lambda\rho}{r}\right)^2\right\}+\partial_r\left(\frac{\partial_r\mathfrak{h}}{r}\right).
\end{equation*}
By the definition of $\mathfrak{q}_{21}$ in \eqref{def-q12}, we also have
\begin{equation}\label{q21term}
\mathfrak{q}_{21}=\frac{(\partial_x\mathfrak{h})(\partial_r\mathfrak{h})}{\mathfrak{O}}=\frac{(\partial_r\mathfrak{h})\omega}{\mathfrak{O}}.
\end{equation}
Finally, a direct computation with using \eqref{q22term}-\eqref{q21term} yields the following conormal boundary condition for \eqref{3D-St-Diff} on $\Gamma_{\rm cd}^{g_D}$:
\begin{equation}\label{3D-omega-BC}
\left(\frac{\mathfrak{q}_{1j}}{r\rho^2}\partial_j\omega,\frac{\mathfrak{q}_{2j}}{r\rho^2}\partial_j\omega\right)\cdot{\bf n}_{g_D}=\widetilde{\mu}\omega\quad\mbox{on}\quad\Gamma_{\rm cd}^{g_D}
\end{equation}
for $\widetilde{\mu}$ defined by
\begin{equation*}\label{3D-omega-GammaD}
\begin{split}
\widetilde{\mu}:=&\frac{\mathfrak{q}_{11}\partial_{xx}\mathfrak{h}+\mathfrak{q}_{12}\partial_{xr}\mathfrak{h}}{r\rho^2(\partial_r\mathfrak{h})\sqrt{1+|g_D'|^2}}+\frac{1}{r\rho^2\sqrt{1+|g_D'|^2}}\left(\frac{(\partial_r\mathfrak{h})(\partial_{xx}\mathfrak{h})}{\mathfrak{O}}\right)\\
&+\frac{\mathfrak{q}_{22}}{r\rho^2\sqrt{1+|g_D'|^2}}
\left(\frac{-\partial_{xx}\mathfrak{h}+g_D\mathfrak{E}}{(\partial_x\mathfrak{h})g_D'+(\partial_r\mathfrak{h})}\right),\\
\end{split}
\end{equation*}
where we represent ${\bf n}_{g_D}$ as
\begin{equation*}
{\bf n}_{g_D}=\frac{1}{\sqrt{1+|g_D'(x)|^2}}\left(\frac{\omega}{\partial_r\mathfrak{h}},1\right).
\end{equation*}

Fix a constant $L>0$ and let $\eta$ be a $C^{\infty}$ function satisfying
$$\eta=1\quad\mbox{for}\quad |x|<L,\quad \eta=0\quad\mbox{for}\quad |x|>L+1,\quad\mbox{and}\quad|\eta'(x)|\le 2.$$
Multiply \eqref{3D-St-Diff} by $\eta^2\omega$ and integrate over the domain $\Omega_{g_D}^-$ to get
\begin{equation}\label{3D-sum-omega}
\iint_{\Omega_{g_D}^-}\frac{\eta^2|\nabla\omega|^2}{r\rho} drdx=\sum_{i=1}^6 I_i+\sum_{i=1}^2 B_i
\end{equation}
for
\begin{equation*}
\begin{split}
&I_1:=-\iint_{\Omega_{g_D}^-}\frac{|\nabla\mathfrak{h}\cdot\nabla\omega|^2\eta^2}{r\rho^2\mathfrak{O}}drdx,\\
&I_2:=-2\iint_{\Omega_{g_D}^-}\left(\frac{\mathfrak{q}_{ij}}{r\rho^2}\partial_j\omega\right)\eta(\partial_i\eta)\omega drdx,\\
&I_3:=2\iint_{\Omega_{g_D}^-}\frac{1}{\mathfrak{O}}\left(r\mathfrak{S}'\rho^{\gamma-1}+\frac{\Lambda\Lambda'}{r}\right)(\nabla\mathfrak{h}\cdot\nabla\eta)\eta\omega^2drdx,\\
&I_4:=-2\iint_{\Omega_{g_D}^-}\frac{1}{\mathfrak{O}}\left(r\mathfrak{S}'\rho^{\gamma-1}+\frac{\Lambda\Lambda'}{r}\right)(\partial_i\mathfrak{h})(\partial_i\omega)\eta^2\omega drdx,\\
&I_5:=\iint_{\Omega_{g_D}^-}\left(\frac{r}{\gamma}\mathfrak{S}''\rho^{\gamma}+\frac{(\Lambda')^2\rho}{r}+\frac{\Lambda\Lambda''\rho}{r}\right)\eta^2\omega^2 drdx,\\
&I_6:=-\iint_{\Omega_{g_D}^-}\frac{(r^2\mathfrak{S}'\rho^{\gamma+1}+\Lambda\Lambda'\rho^2)^2}{r\rho^2\mathfrak{O}}\eta^2\omega^2 drdx,\\
&B_1:=\int_{\Gamma_{\rm cd}^{g_D}\cup\Gamma_{\rm en}^{g_D}}\left(\frac{\mathfrak{q}_{ij}}{r\rho^2}\partial_j\omega\right)\eta^2\omega\cdot{\bf n}_{\rm out}ds,\\
&B_2:=-\int_{\Gamma_{\rm cd}^{g_D}\cup\Gamma_{\rm en}^{g_D}}\frac{\partial_i\mathfrak{h}}{\mathfrak{O}}\left(r\mathfrak{S}'\rho^{\gamma-1}+\frac{\Lambda\Lambda'}{r}\right)\eta^2\omega^2 \cdot{\bf n}_{\rm out}ds.
\end{split}
\end{equation*}
We will show that
\begin{equation}\label{3D-omega-claim}
\left\{\begin{split}
&I_1+I_4+I_6\le 0,\\
&|I_2|\le C\int_L^{L+1}\int_0^{g_D(x)}\left(1+\frac{1}{r^2}\right)|\nabla\omega|^2drdx,\\
&|I_k|\le C\sigma\int_L^{L+1}\int_0^{g_D(x)}\frac{|\nabla\omega|^2}{r}drdx\quad\mbox{for}\quad k=3,5,\\
&|B_1|\le C{\sigma}\int_0^{L+1}\int_0^{g_D(x)}{\frac{|\nabla\omega|^2}{r}}drdx+E,\\
&|B_2|\le C\sigma\int_0^{L+1}\int_0^{g_D(x)}{|\nabla\omega|^2}drdx+E,
\end{split}\right.
\end{equation}
where $E\ge0$ and $C>0$ are constants depending only on the data and $\alpha$.
From now on, the constant $C$ depends only on the data and $\alpha$, which may vary from line to line.

First, by the H\"older's inequality, we have
\begin{equation*}
\begin{split}
I_4&\le 2\left(\iint_{\Omega_{g_D}^-}\frac{|\nabla\mathfrak{h}\cdot\nabla\omega|^2\eta^2}{r\rho^2\mathfrak{O}}drdx\right)^{1/2}
\left(\iint_{\Omega_{g_D}^-}\frac{\left(r^2\mathfrak{S}'\rho^{\gamma+1}+{\Lambda\Lambda'\rho^2}\right)^2}{r\rho^2\mathfrak{O}}\eta^2\omega^2drdx\right)^{1/2}\\
&=2\sqrt{|I_1||I_6|},
\end{split}
\end{equation*}
from which we obtain that
$$I_1+I_4+I_6\le -|I_1|+2\sqrt{|I_1||I_6|}-|I_6|\le 0.$$

Before we prove the remaining estimates in \eqref{3D-omega-claim}, we compute estimates for $(\rho, \mathfrak{O}, \mathfrak{S}',\mathfrak{S}'',\Lambda',\Lambda'')$.
By a straightforward computations with using the estimate \eqref{Thm2.1-uniq-est} given in Theorem \ref{3D-MainThm}(a), one can easily check that there exists a constant $\sigma_{\star}\in(0,\sigma_1]$ depending only on the data and $\alpha$ so that if $\sigma\le \sigma_{\star}$, then we have
\begin{equation}\label{3D-far-Lem}
|\rho-\rho_0^-|\le \frac{\rho_0^-}{2}\quad\mbox{and}\quad|\mathfrak{V}_0-\mathfrak{V}|\le \frac{\mathfrak{V}_0}{2}\quad\mbox{in}\quad\overline{\Omega_{g_D}^-}
\end{equation}
for
\begin{equation*}
\mathfrak{V}_0:=c_0^2-u_0^2=\frac{\gamma p_0}{\rho_0^-}-u_0^2,\quad \mathfrak{V}:=c^2-|{\bf u}|^2.
\end{equation*}
By \eqref{3D-far-Lem}, it holds that
\begin{equation}\label{3D-O-lower}
\begin{split}
\mathfrak{O}&=r^2(\gamma+1)\mathfrak{S}\rho^{\gamma}-2r^2B_0^-\rho+\Lambda^2\rho
={r^2}{\rho}\left(\frac{\gamma p}{\rho}-|{\bf u}|^2\right)+\Lambda^2\rho\\
&=r^2\rho\left(c^2-|{\bf u}|^2+\left(\frac{\Lambda}{r}\right)^2\right)
=r^2\rho\left(\mathfrak{V}+\left(\frac{\Lambda}{r}\right)^2\right)
\ge \frac{r^2\rho_0^- \mathfrak{V}_0}{4}.
\end{split}
\end{equation}
By using the equations in \eqref{3D-ang} and the definition of $\mathfrak{h}$, it can be checked that
\begin{equation}\label{rg-rel}
\int_0^{\mathcal{G}^{-1}(\mathfrak{h}(x,r))}s\rho u_x(0,s)ds
=\int_0^r s\rho u_x(x,s)ds \quad\text{in $\Omega_{g_D}^-$},
\end{equation}
where $\mathcal{G}$ is given in \eqref{S-Lambda}.
One can also check that there exists a constant $\sigma_{\star\star}\in(0,\sigma_{\star}]$ depending only on the data and $\alpha$ so that if $\sigma\le\sigma_{\star\star}$, then
\begin{equation}\label{rhou-G}
|\rho u_x-\rho_0^-u_0|\le \frac{\rho_0^-u_0}{2},
\end{equation}
and it follows from \eqref{rg-rel}-\eqref{rhou-G} that
\begin{equation*}\label{G-inv-est}
0< \frac{1}{\sqrt{3}}\le \frac{\mathcal{G}^{-1}(\mathfrak{h}(x,r))}{r}\le \sqrt{3}\quad \text{in $\Omega_{g_D}^-$},
\end{equation*}
then we get
 \begin{equation}\label{S-Lam-est}
 \begin{split}
&|\mathfrak{S}'(\mathfrak{h})|\le \frac{C\sigma}{r},\quad|\mathfrak{S}''(\mathfrak{h})|\le  \frac{C{\sigma}}{r^3},\\
&|\Lambda'(\mathfrak{h})|\le C\sigma,\quad |\Lambda''(\mathfrak{h})|\le\frac{C{\sigma}}{r^2}
\end{split}
\end{equation}
in $\Omega_{g_D}^-$.

Now we are ready to estimate for $I_2$.
Since $\frac{\omega(\partial_j\omega)}{r}\le C\left(\omega^2+\frac{|\nabla\omega|^2}{r^2}\right)$ and $\rho\ge\frac{\rho_0^-}{2}$ in $\Omega_{g_D}^-$, we have
\begin{equation}\label{3D-I2-est}
|I_2|\le C\int_L^{L+1}\int_0^{g_D(x)}\left(\omega^2+\frac{|\nabla\omega|^2}{r^2}\right)drdx.
\end{equation}
By the boundary condition $\omega\equiv 0$ on $\{r=0\}$ stated in \eqref{omega-BC-0}, we have
\begin{equation*}\label{3D-omega-Poin}
\omega(x,t)=\int_0^{t}\partial_r\omega(x,r)dr\quad\mbox{for}\quad (x,t)\in\overline{\Omega_{g_D}^-}.
\end{equation*}
By the H\"older inequality, we have the following estimates:
\begin{equation}\label{3D-omega22}
\left\{\begin{split}
&\omega^2(x,t)\le  C t^2\int^{t}_0\frac{(\partial_r\omega)^2(x,r)}{r}dr,\\
&\omega^2(x,t)\le  t\int^t_0(\partial_r\omega)^2(x,r) dr\le C\int_0^{g_D(x)}|\nabla\omega|^2 dr\quad\mbox{for }(x,t)\in\overline{\Omega_{g_D}^-}.
\end{split}\right.
\end{equation}
Substituting the second estimate of \eqref{3D-omega22} into \eqref{3D-I2-est} yields
$$|I_2|\le C\int_L^{L+1}\int_0^{g_D(x)}\left(1+\frac{1}{r^2}\right)|\nabla\omega|^2drdx.$$

It follows from \eqref{3D-far-Lem}-\eqref{S-Lam-est} that
\begin{align}
\label{I3-est}|I_3|&\le C\sigma\int_L^{L+1}\int_0^{g_D(x)}\frac{\omega^2}{r^2}drdx,\\
\label{B2-est}|B_2|&\le C\sigma\int_{0}^{L+1}\omega^2(x,g_D(x))dx+E_2,
\end{align}
where the constant $E_2\ge 0$ depends only on the data and $\alpha$.
Substituting the first estimate of \eqref{3D-omega22} into \eqref{I3-est} gives
\begin{equation*}
|I_3|\le C\sigma\int_L^{L+1}\int_0^{g_D(x)}\frac{|\nabla\omega|^2}{r} drdx.
\end{equation*}
Similarly, substituting the second estimate of \eqref{3D-omega22} into \eqref{B2-est} gives
\begin{equation*}
|B_2|\le C\sigma\int_0^{L+1}\int_0^{g_D(x)}|\nabla\omega|^2drdx+E_2.
\end{equation*}
It follows from  \eqref{3D-omega-BC} and  \eqref{3D-far-Lem}-\eqref{S-Lam-est} that
\begin{equation}\label{3D_B1}
\begin{split}
|B_1|\le\int_{0}^{L+1}\widetilde{\mu}\omega^2(x,g_D(x))dx+E_1
\le C\sigma\int_{0}^{L+1}\omega^2(x,g_D(x))dx+E_1,
\end{split}
\end{equation}
where the constant $E_1\ge0$ depends only on the data and $\alpha$.
Substituting the first estimate of \eqref{3D-omega22} into \eqref{3D_B1} gives
\begin{equation*}
|B_1|\le C{\sigma}\int_0^{L+1}\int_0^{g_D(x)}{\frac{|\nabla\omega|^2}{r}}drdx+E_1.
\end{equation*}
Also, we obtain from \eqref{3D-far-Lem}, \eqref{S-Lam-est}, and the first estimate of \eqref{3D-omega22} that
$$|I_5|\le C{\sigma}\int_0^{L+1}\int_0^{g_D(x)}\frac{\omega^2}{r^2}drdx\le C{\sigma}\int_0^{L+1}\int_0^{g_D(x)}\frac{|\nabla\omega|^2}{r}drdx.$$

Now the estimates in \eqref{3D-omega-claim} are all verified.

From \eqref{3D-sum-omega}-\eqref{3D-omega-claim}, we have
\begin{equation*}
\begin{split}
\int_0^L&\int_0^{g_D(x)}\frac{|\nabla\omega|^2}{r} drdx\\
\le& C^{(\sharp)}{\sigma}\int_0^L\int_0^{g_D(x)}\frac{|\nabla\omega|^2}{r} drdx+C\int_L^{L+1}\int_0^{g_D(x)}\left(1+\frac{1}{r^2}\right)|\nabla\omega|^2drdx +E,
\end{split}
\end{equation*}
where the constants $C^{(\sharp)}>0$ and $E\ge0$ depend only on the data and $\alpha$.
If it holds that
$${\sigma}\le\frac{1}{2C^{(\sharp)}}, $$
then we obtain from the previous estimate that
\begin{equation*}\label{3D-urr}
\begin{split}
 \int_0^L\int_0^{g_D(x)}\frac{|\nabla\omega|^2}{r} drdx
&\le C\int_L^{L+1}\int_0^{g_D(x)}\left(1+\frac{1}{r^2}\right)|\nabla\omega|^2drdx+CE.
\end{split}
\end{equation*}
Since $|\nabla\omega|\le C$ and $\frac{|\nabla\omega|^2}{r^2}\le C$ in $\overline{\Omega_{g_D}^-}$ by \eqref{Thm2.1-uniq-est}, we have
\begin{equation*}
 \int_0^L\int_0^{g_D(x)}\frac{|\nabla\omega|^2}{r} drdx\le C.
\end{equation*}
Since $0<g_D(x)<1$, we have
\begin{equation*}
\int_0^L\int_0^{g_D(x)}|\nabla\omega|^2drdx\le \int_0^L\int_0^{g_D(x)}\frac{|\nabla\omega|^2}{r} drdx\le C
\end{equation*}
for some constant $C>0$ independent of $L$.
Passing to the limit $L\rightarrow\infty$ yields
\begin{equation*}
\int_0^\infty\int_0^{g_D(x)}|\nabla\omega|^2drdx\le C.
\end{equation*}
Hence
$$\int_L^{L+1}\int_0^{g_D(x)}|\nabla\omega|^2drdx\rightarrow0\quad\mbox{as}\quad L\rightarrow\infty.$$
Since $\omega\in C^{1,\alpha}(\overline{\mathcal{N}^-_{g_D}})$, we have
\begin{equation}\label{3D-omega-lim}
\|\nabla\omega(x,\cdot)\|_{C^0(\overline{\Omega_{g_D}^-\cap\{x> L\}})}\rightarrow 0\quad\mbox{as}\quad L\rightarrow\infty.
\end{equation}
By \eqref{3D-omega-lim} and the compatibility condition $\omega\equiv 0$ on $\{r=0\}$, we have
\begin{equation}\label{omega0}
\|\omega(x,\cdot)\|_{C^0(\overline{\Omega_{g_D}^-\cap\{x> L\}})}\rightarrow 0\quad\mbox{as}\quad L\rightarrow\infty.
\end{equation}
Since $\rho>\rho_0^-/2$ in $\Omega_{g_D}^-$ and $\omega=\partial_x\mathfrak{h}=-r\rho u_r$, \eqref{omega0} implies that
\begin{equation}\label{ur0}
\|ru_r(x,\cdot)\|_{C^0(\overline{\Omega_{g_D}^-\cap\{x> L\}})}\rightarrow 0\quad\mbox{as}\quad L\rightarrow\infty.
\end{equation}
By \eqref{3D-omega-lim} and \eqref{ur0}, we have
\begin{equation*}
\| ru_r(x,\cdot)\|_{C^1(\overline{\Omega_{g_D}^-\cap\{x> L\}})}\rightarrow 0\quad\mbox{as}\quad L\rightarrow\infty,
\end{equation*}
from which
\begin{eqnarray}
&\nonumber&\|g_D'(x)\|_{C^1({\{x\ge L\}})}\rightarrow 0,\\
&\label{nabla-ur}&\|u_r(x,\cdot)\|_{C^1(\overline{\mathcal{N}_{g_D}^-\cap\{x> L\}})}\rightarrow 0\quad\mbox{as}\quad L\rightarrow\infty.
\end{eqnarray}
It follows from the equation in \eqref{ES-Far-eq} and \eqref{nabla-ur} that
\begin{equation*}
\|\partial_r p(x,\cdot)-\frac{\rho u_{\theta}^2}{r}(x,\cdot)\|_{C^0(\overline{\mathcal{N}_{g_D}^-\cap\{x> L\}})}\rightarrow 0\quad\mbox{as}\quad L\rightarrow\infty.
\end{equation*}
The proof of Theorem \ref{3D-MainThm}(b) is completed by choosing $\sigma_2$ as
\begin{equation*}
\sigma_2=\min\left\{\sigma_1,\sigma_{\star\star}, \frac{1}{2C^{(\sharp)}}\right\}.
\end{equation*}
\qed


\vspace{.25in}
\noindent
{\bf Acknowledgements:}
The research of Myoungjean Bae was supported in part by  Samsung Science and Technology Foundation
under Project Number SSTF-BA1502-02.
The research of Hyangdong Park was supported in part by  Samsung Science and Technology Foundation
under Project Number SSTF-BA1502-02.

\end{document}